\documentclass[a4paper,twoside,11pt]
{amsart}
\usepackage{amssymb, amsfonts,amsmath}
\usepackage{epsf}
\usepackage{epsfig}
\usepackage[a4paper,margin=3cm]{geometry}
\usepackage[all]{xy}
\newtheorem{theorem}{Theorem}
\newtheorem{cor}{Corollary}
\newtheorem{lemma}{Lemma}
\newtheorem{proposition}{Proposition}
\newtheorem{definition}{Definition}
\newtheorem{remark}{Remark}

\newcommand{\gl}{\mathfrak{gl}}

\newcommand{\Z}{\mathbb{Z}}

\newcommand{\gr}{\operatorname{gr}}
\newcommand{\im}{\operatorname{Im}}
\newcommand{\R}{\mathbb R}

\newcommand{\E}{\mathcal{E}}
\newcommand{\JT}[1]{J^{#1}(\R,\R)}

\newcommand{\Fol}{\operatorname{Fol}}

\begin{document}\title{
Prolongation of quasi-principal frame bundles and geometry of flag structures on manifolds}
%On prolongation of quasi-principal frame bundles on filtered manifolds %and geometry of systems of ODE's of mixed order
%}
\author
{Boris Doubrov
\address{Belarussian State University, Nezavisimosti Ave.~4, Minsk 220030, Belarus;
 E-mail: doubrov@islc.org}
\and Igor Zelenko
\address{Department of Mathematics, Texas A$\&$M University,
   College Station, TX 77843-3368, USA; E-mail: zelenko@math.tamu.edu}}
\subjclass[2000]{58A30, 58A17}

\begin{abstract}
Motivated by the geometric theory of differential equations and the variational approach to the equivalence problem for geometric structures on manifolds, we consider the problem of equivalence for distributions with fixed submanifolds of flags on each fiber. We call them flag structures. The construction of the canonical frames for these structures can be given in the two prolongation steps:
the first step, based on our previous works \cite{flag2, flag1}, gives the canonical bundle of moving frames for the fixed submanifolds of flags on each fiber and the second step consists of the prolongation of the bundle obtained in the first step.
%As shown in our previous works \cite{flag2, flag1}
The bundle obtained in the first step is not as a rule a principal bundle so that the classical Tanaka prolongation procedure  for filtered structures can not be applied to it.  However, under natural assumptions on submanifolds of flags and on the ambient distribution, this bundle satisfies a nice weaker property. The main goal of the present paper is to formalize this property, introducing the so-called \emph{quasi-principle frame bundles}, and to generalize the Tanaka prolongation procedure to these bundles. Applications to the equivalence problems for systems of differential equations of mixed order,  bracket generating distributions, sub-Riemannian and more general structures on distributions are given.

\end{abstract}

\maketitle\markboth{Boris Doubrov and Igor Zelenko}{Prolongation of quasi-principal frame bundles}

\newcommand{\norm}[1]{\left\Vert#1\right\Vert}
\newcommand{\abs}[1]{\left\vert#1\right\vert}
\newcommand{\set}[1]{\left\{#1\right\}}
\newcommand{\Real}{\mathbb R}
\newcommand{\eps}{\varepsilon}
\newcommand{\To}{\longrightarrow}
\newcommand{\BX}{\mathbf{B}(X)}
\newcommand{\A}{\mathcal{A}}
\newcommand{\mg}{\mathfrak g}
\newcommand{\vf}{\varphi}
%\newcommand{\str}[1]{\Green{{\sout{#1}}}}
%\newcommand{\new}[1]{\Blue{#1}}
%% ----------------------------------------------------------------
%\begin{document}
%
%\title{A note on Tanaka's prolongation procedure for filtered structures of constant type}%
%\author{Igor Zelenko}%
%\address{    Department of Mathematics, Texas A$\&$M University
%    College Station, TX 77843-3368}%
%\email{zelenko@math.tamu.edu}%

%\thanks{}%
%\subjclass{}%
%\keywords{}%

%\date{}%
%\dedicatory{}%
%\commby{}%
% ----------------------------------------------------------------
%\begin{abstract}

%\end{abstract}
%\maketitle

%Please type here List of Keywords for your article separated by semicolon.

%\Classification{58A30, 58A17} % e.g. 35A30; 81Q05
% ----------------------------------------------------------------
\section{Introduction}
\setcounter{equation}{0}
\setcounter{theorem}{0}
\setcounter{lemma}{0}
\setcounter{proposition}{0}
\setcounter{definition}{0}
\setcounter{cor}{0}
\setcounter{remark}{0}
\setcounter{example}{0}
%A general prolongation procedure for construction of canonical frames for principal frame bundles on filtered manifold was developed by Tanaka \cite{tan}.
%This theory can be applied for distributions with so-called constant Tanaka symbol and distributions with various additional structures on them
%In the present paper we describe a prolongation procedure for construction of canonical frames for structures that are more general than principle frame %bunlde on filtered manifold.  The

%\subsection{Statement of the problem}
\label{stat}
\subsection{Flag structures, double fibrations, and  motivating examples}
\label{df}
Let $\Delta$ be a bracket generating distribution on a manifold $\mathcal S$, i.e. a vector subbundle of the tangent bundle $T\mathcal S$.
Assume that for any point $\gamma\in\mathcal S$ a submanifold $\mathcal Y^\gamma$ in a flag variety of the fiber $\Delta(\gamma)$ of $\Delta$ is chosen smoothly with respect to $\gamma$. We call such structure a \emph{flag structure} and denote it by $\bigl(\Delta, \{\mathcal Y^\gamma\}_{\gamma \in\mathcal S}\bigr)$. We are interested in the (local) equivalence problem for these structures with respect to the group of diffeomorphisms of $\mathcal S$.

The flag structures appear in the natural equivalence problems for differential equations via the so-called linearization procedure \cite{doub2, doub3, var} and in the variational approach to equivalence of vector distributions and more general geometric structures on manifolds via the so-called symplectification/linearization procedure \cite{agrachev, agrgam1,jac1, zelvar, doubzel1, doubzel2, doubzel3}.
These procedures  are described in this subsection from the point of view of the geometry of double fibrations.  The motivating examples are discussed briefly in  subsections \ref{difeqprelim}-\ref{dprelim} and in the full generality in section \ref{motsec}.

The cases when the additional structures (for example symplectic or Euclidean  structures) are  given on each space $\Delta(\gamma)$
fit into our theory as well.  In most of the applications  the dimension of submanifolds $\mathcal Y^\gamma$
is equal to one, i.e. each $\mathcal Y^\gamma$ is an unparameterized curve. The case, when each  curve $\mathcal Y^\gamma$ is parameterized, i.e. some parametrization on it is fixed, up to a translation, is discussed as well.

%Since in the most of applications the flag structures are not the original objects to study, but the equivalence problem for them  is obtained from the equivalence problems for other structures of a primary interest after some additional steps, we use the gothic letter $\mathfrak M$ for the ambient manifolds of flag structures in order to distinguish them and the ambient manifolds of the original structures.
In most of the applications the flag structures are not the original objects of study. They arise in a natural way from the equivalence problems for other geometric structures, often after some preliminary steps.
 %This motivates us to use the gothic letter $\mathcal S$ for the ambient  manifolds of flag structures in order to distinguish them from the smooth manifolds of the original geometric structures.
On many occasions one can intrinsically assign  to an original  geometric structure another  manifold endowed with two (or more) foliations.
For example, this situation naturally occurs in the geometry of differential equations.  A differential equation is considered in differential geometry as a submanifold
$\mathcal E$, the \emph{equation manifold},
of the corresponding  jet space.  On one hand, $\mathcal E$
is foliated by the solutions of the equation (prolonged to this jet space).
On the other hand, $\mathcal E$ is foliated by the fibers over the jet spaces of lower order.

Another, more involved instance of this situation appears in the geometry of submanifolds $\mathcal U$ of a tangent bundle $TM$  of a manifold $M$ with respect to the natural action of the group of diffeomorphisms. For example, this includes the geometry of vector distributions, sub-Riemannian and sub-Finslerian structures as particular cases. The insight here comes from the Geometric Control Theory: one associates to the submanifold $\mathcal U$ a variational problem
%(or a family of variational problems)
in a natural way and uses the Hamiltonian formalism provided by the Pontryagin Maximum Principle in order to describe the extremal of this variational problem.
The role of the manifold endowed with two foliation plays a special submanifold of the cotangent bundle of $M$ that is dual to $\mathcal U$  in a certain sense and is foliated by extremals of this variational problem. The second foliation is a foliation induced by the canonical fibration $T^*M\mapsto M$.

%the variational problem naturally associated with $\mathcal U$. }

Returning to the general situation, assume that a manifold $\widetilde{\mathcal S}$ is foliated by two foliations so that the corresponding quotient manifold $\mathcal S_1$ and $\mathcal S_2$ of the leaves of these foliations are well defined. Then we have the following double fibration:

%\begin{diagram}
%A &\rInto &B &\lInto &C\\
%&\rdOnto &\dDotsto &\ldOnto\\
%& &D
%\end{diagram}
$$\xymatrix{
& \widetilde{\mathcal S} \ar[ld]_{\pi_{1}} \ar[rd]^{\pi_{2}} &\\
\mathcal S_1  & &  \mathcal S_2 }$$
%\begin{diagram}
%&\widetilde{\mathcal S}&\\
%&\ldTo& &\rdTo&\\
%&\mathcal S_1  & & \mathcal S_2&,
%\end{diagram}
where $\pi_i:\widetilde{\mathcal S}\rightarrow \mathcal S_i$, $i=1,2$ are the canonical projections to the corresponding quotient manifolds.
Assume that  $\mathcal C_i$ is the distribution of tangent spaces to the fibers of $\pi_i$. For any $\gamma_1 \in \mathcal S_1$
%$\mathcal Y_1^{\gamma_1}(\lambda)$
(representing the fiber of $\pi_1$) let
%and $\gamma_2\in S_2$
%let
\begin{equation}
\label{Jacprelim}
\mathcal Y_1^{\gamma_1}(\lambda):=d\pi_1 \mathcal C_2(\lambda), \quad \forall \lambda \in \pi_1^{-1}(\gamma_1),\\
%&~&\mathcal Y_2^{\gamma_2}(\lambda)=d\pi_1 C_2(\lambda), \quad \forall \lambda \in \pi_1^{-1}(\gamma_2).
 \end{equation}

Then  $\mathcal Y_1^{\gamma_1}(\lambda)$ is a subspace in $T_{\gamma_1}\mathcal S_1$. The map $\lambda\mapsto \mathcal Y_1^{\gamma_1}(\lambda)$ is called the \emph{linearization of the fibration $\pi_2$ by the fibration $\pi_1$ along the fiber $\gamma_1$.} Speaking informally,
%\str {from the fiber of $\pi_1$ to the corresponding Grassmannian of subspaces of $T_{\gamma_1}S_1$}
the map
$\lambda\mapsto \mathcal Y_1^{\gamma_1}(\lambda)$ describes the dynamics of the fibers of the fibrarion $\pi_2$ by the foliation of the fibers of  $\pi_1$ along the fiber $\gamma_1$.  If (the image of) this linearization is a submanifold $\mathcal Y_1^{\gamma_1}$ in the corresponding Grassmannianfor any point $\gamma_1$, then we get the flag structure on $\mathcal S_1$ with $\Delta=T\mathcal S_1$.
In the same way for all $\gamma_2 \in \mathcal S_2$ one can consider the linearizations $\mathcal Y_2^{\gamma_2}$ of the fibration $\pi_1$ by the fibration $\pi_2$ along the fiber $\gamma_2$, which, under obvious assumption, define the flag structure on $\mathcal S_1$ with $\Delta=T\mathcal S_1$.

Often, at least one of submanifolds $\mathcal Y_1^{\gamma_1}$ and $\mathcal Y_2^{\gamma_2}$ in the corresponding Grassmannian has nontrivial local geometry and this gives an effective way to obtain differential invariants of the original geometric structures via the passage to the corresponding flag structure.
Besides, in some situations (as in the case of vector distributions), the manifold $\widetilde{\mathcal S}$ is endowed with the additional distribution $\widetilde \Delta$ such that both  distirbutions $\mathcal C_i$ are subdistribution of it and one of these distributions , say $\mathcal C_1$ , satisfies $[\mathcal C_i, \widetilde \Delta]\subset\widetilde \Delta$, i.e. any section of $\mathcal C_1$ is an infinitesimal symmetry of $\widetilde \Delta$. Then $\Delta:=\pi_{{1}_*}\widetilde \Delta$ is a well defined distribution in $\mathcal S_1$ and $\mathcal Y_1^{\gamma_1}(\lambda)$ is a submanifold in the
corresponding Grassmannian of subspaces in $\Delta(\gamma_1)$ for any $\gamma_1\in S_1$. This motivates us to consider the case of flag structures with $\Delta$ not equal to the tangent space of the ambient distribution.
%Usually $\mathcal S$ is a space of leaves of a foliation on a certain bigger manifold $\widetilde S$, naturally related to the original structure. For %example, in the case of differential equations $\mathcal S$ might be the space of (prolonged) solutions of a differential equations or the space of . %This is also the reason why we denote the points of $\mathcal S$ by $\gamma$, keeping in mind that they represent the leaves of a certain foliation.

\subsection{Flag structures in geometric theory of differential equations}
\label{difeqprelim}
 The more detailed description of this class of problems is given in  Examples 1 and 2 of section \ref{motsec}.
As was already mentioned,  a differential equation is considered in differential geometry as a submanifold $\mathcal E$, the \emph{equation manifold}, of the corresponding  jet space. Hence the equation manifold plays the role of $\widetilde S$ in the general scheme. The first foliation on $\mathcal E$ is given by the solutions (prolonged to this jet space), considered as the leaves. So the space of solutions  ${\rm Sol}$ plays the role of $\mathcal S_1$ in the general scheme. Since the jet space is a bundle over all jet spaces of lower order, we usually have not one but many (nested) additional foliations/fibrations. The linearization of all of these fibrations by the fibration $\pi_1$ along each solutions gives the flag structure on the space of solution with nontrivial local geometry of submanifolds in the corresponding flag variety.

For example, in the case of scalar ordinary differential equations of order $n\geq 3$ up to contact transformations one gets in this way the curve of complete flags in the tangent space to any point in the space of solution. Moreover, the picture can be simplified here, because  the curve of flags can be recovered by osculation from the curve of one-dimensional subspaces in these flags, i.e. from the non-degenerate curve in the projective space. The latter is the linearization by $\pi_1$ of the fibration given by the projection to the jet-space of the previous order. The fundamental set of invariants of curves in projective spaces were constructed by Wilczynski \cite{wil} in 1906 and they produce the contact invariants of the original differential equation.

\subsection{Flag structures associated with sub-Riemannian and sub-Finslerian structures}
\label{sRprelim}
The more detailed description of this class of problems is given in  Example 4 of section \ref{motsec}.
By a geometric structure on a manifolds $M$ we mean a submanifold of its tangent bundle $TM$ transversal to the fibers.

Let us give several examples. A distribution $D$ on $M$ is given by fixing a vector subspace $D(q)$ of $T_qM$, depending
smoothly on $q$. A sub-Riemannian
%(sub-Finslerian)
structure $\mathcal U$ on $M$ with underlying distribution $D$ is given by choosing on each space $D(q)$ an ellipsoid $\mathcal U(q)$ symmetric with respect to the origin.
%(the
%boundary of a strongly convex body containing the origin
%in their interior).
%Denote
%$\mathcal U(q):=\mathcal U\cap T_q M$. %, if $\mathcal
%%V_q$ is an $l$-dimesnional linear subspace of $T_qM$, depending
%%|smoothly on $q$, then $\mathcal V$ is called \emph{a rank $l$ vector
%%distribution on $M$}.
%if $\mathcal V_q$ is an intersection of an ellipsoid centered at the
%origin with a linear subspace ${\mathcal D}_q$ in $T_qM$ (where both
%the ellipsoids and the subspaces $\mathcal D_q$ depend smoothly on $q$),
%then $\mathcal V$ is called a \emph {sub-Riemannian structure on
% a
%manifold
%$M$ with underlying distribution $\mathcal D$}.
In this case
$\mathcal U(q)$ is the unit sphere w.r.t. the unique Euclidean norm on $D(q)$, i.e. fixing an ellipsoid in
$D(q)$ is equivalent to fixing an Euclidean norm on
$D(q)$ for any $q\in M$. If in the constructions above we replace the ellipsoids by the
boundaries of strongly convex bodies in $D(q)$ containing the origin
in their interior
(sometimes also assumed to be  symmetric w.r.t.
the origin) we will get a sub-Finslerian structure on $M$. In the case when $D=TM$ we obtain in this way classical Riemannian and Finslerian structures.

The key idea, due to A. Agrachev,  of the variational approach to the equivalence problem of geometric structures (or of the  sympelctification of the equivalence problem) is that invariants of a geometric structure  $\mathcal U$ on a manifold $M$ can be obtained by studying the flow of extremals of variational problems naturally associated with $\mathcal U$ \cite{agrgam1, jac1}.
For this first one can define  \emph{admissible  (or horizontal) curves of the structure $\mathcal U$}. A Lipschitzian curve $\alpha(t)$ is called \emph{admissible (horizontal)}  if $\dot\alpha(t)\in \mathcal U\cap T_{\alpha(t)}M$ for almost every $t$.
 Then one can associate with $\mathcal U$  the following family of the so-called time-minimal problems: given two points $q_0$ and $q_1$ in $M$ to steer from $q_0$ to $q_1$ in a minimal time moving along admissible curves of $\mathcal U$.
For sub-Riemannian (sub-Finslerian) structures these time-minimal problems are exactly the length minimizing problems.

The Pontryagin Maximum Principal of Optimal Control gives a very efficient way to describe extremals  of the time-minimal problems.
For simplicity assume
that the \emph{maximized Hamiltonian} of the Pontryagin Maximum Principle
\begin{equation}
\label{maxH} H(p, q)=\max_{v\in \mathcal U(q)} p(v), \quad q\in M,
p\in T_q^*M
\end{equation}
is well defined and smooth in an open domain $O\subset T^*M$ and for
some $c>0$ (and therefore for any $c>0$ by homogeneity of $H$ on
each fiber of $T^*M$) the corresponding level set
\begin{equation}
\label{Hc}
\mathcal
H_c=\{\lambda\in O: H(\lambda)=c\}
\end{equation}
 is nonempty and consists of
regular points of $H$ (for more general setting see \cite{jac1} or Remark \ref{nonmonotrem} below).

Consider the Hamiltonian vector field $\overrightarrow H$ on $\mathcal H_c$, corresponding to the Hamiltonian $H$, i.e. the
vector field satisfying $i_{\overrightarrow H}\hat\sigma=-dH$, where
$\hat\sigma$ is the canonical symplectic structure on $T^*M$. The
integral curves of this Hamiltonian system are normal Pontryagin
extremals of the time-optimal problem, associated with the geometric
structure $\mathcal U$, or, shortly, the normal extremals of $\mathcal
U$. For example, if $\mathcal U$ is a sub-Riemannian structure with
underlying distribution $D$, then the maximized Hamiltonian
satisfies $H(p,q)=||p|_{_{D(q)}}||_q$, i.e. $H(p,q)$ is
equal to the norm of the restriction of the functional $p\in T_q^*M$
on $D(q)$ w.r.t. the Euclidean norm $||\cdot||_q$ on
$D(q)$; $O=T^*M\backslash D^\perp$.
%, where $\mathcal
%D^\perp$ is the annihilator of $\mathcal D$,
%$$\mathcal D^\perp=\{(p,q)\in T^*M: p(v)=0\,\,\forall v\in \mathcal D_q\}.$$
The
projections of the trajectories of the corresponding Hamiltonian
systems to the base manifold $M$ are normal sub-Riemannian
geodesics. If $D=TM$, then they are exactly the Riemannian
geodesics of the corresponding Riemannian structure.

 %There are two types of Pontryagin extremals of an optimal control problem, \emph{normal extremals} and \emph{abnormal extremals}. The Lagrangian multiplier near the functional of the optimal problem is not equal to zero in the first case and equal to zero in the second one.
%Under certain regularity assumptions that in particular hold for sub-Riemannian and sub-Finslerian structures the normal Pontryagin extremals of these variational problems foliate certain codimension one submanifold $\mathcal H$ of the cotangent bundle $T^*M$ of the ambient manifold $M$ ($\mathcal H$ is a non-zero level set of the so-called maximized Hamiltonian of the time-minimal problem). In the case of Riemannian (Finslerian) structures the projections of these extremals to the base manifold are classical Riemannian (Finslerian) geodesics.

How the flag structures appear here? The Hamiltonian level set $\mathcal H_c$ with $c>0$ plays the role of the manifold $\widetilde{\mathcal S}$ in the general scheme. The first  foliation on $\mathcal H_c$ is given by the normal Pontryagin extremals and the second foliation is given by the fibers of the canonical projection from $T^*M$ to $M$, restricted to $\mathcal H_c$. The role of the manifold $\mathcal S_1$ in the general scheme is played by
the space $\mathfrak N$  of all normal
extremals. Note that the space  $\mathfrak N$ is endowed with the natural symplectic structure, induced from the canonical symplectic structure on $T^*M$ and the linearization of the second foliation by the first one along any normal extremal $\gamma$ is a curve of
%considered as the quotient of the Hamiltonian level set $\mathcal H$ by the foliation of these extremals, the distribution $\Delta$ again coincides with the whole tangent bundle of $\mathfrak N$. Besides,
% To an extremal $\gamma$ a curve $\widetilde{ \mathcal Y}^\gamma$
of Lagrangian subspace of $T_\gamma \mathfrak N$ with respect to this symplectic structure. This defines the flag structure on $\mathcal N$, where the
 the distribution $\Delta$ again coincides with the whole tangent bundle of $\mathfrak N$.
%can be intrinsically assigned.
  It corresponds to the linearziation of the flow of extremals along the extremal $\gamma$, i.e to the Jacobi equation along $\gamma$. Therefore it is called the \emph{Jacobi curve of $\gamma$} \cite{agrgam1, jac1}.
%the distribution $\Delta$ again coincides with the whole tangent bundle of $\mathfrak N$.
%Geometrically it  describes the dynamics of the fibers of the Hamiltonian level set $\mathcal H$ over the base manifold $M$ by the flow of extremals along the extremal $\gamma$.
 Collecting the osculating spaces of this curve of any order together with their skew symmetric complements w.r.t to the natural symplectic structure
%(the refinement procedure in \cite{flag2})
one finally assigns to $\gamma$
%{\bf Step 2:} \textsf {Construction of Jacobi curves of abnormal extremals.}
%Recall that on each contact hyperplane $\Delta(\gamma)$ of $T_\gamma N$ the symplectic structure is defined canonically up to a nonzero multiple.
%To any abnormal extremal $\gamma$ one can assign
a curve of isotropic/coisotropic subspaces  (or the curves of symplectic flags in the terminology of \cite{flag2}) in each space $T_\gamma N$
%called the \emph{Jacobi curve of  the
%abnormal extremal $\gamma$}.
This curve plays the role of $\mathcal Y^\gamma$ in the flag structure associated with the geometric structure $\mathcal U$ and it will be called the Jacobi curve of the normal extremal $\gamma$ as well, where the role of the distribution $\Delta$ plays the tangent bundle $T\mathcal N$. An  additional  feature of this example is that all curves of flags here are parametrized, because all normal extremals are parametrized.
%  Note also that in the considered case the extremals are parameterized by a natural parameter, so the Jacobi curves $\mathcal Y^\gamma$ are parameterized as well.
 % of extremals is endowed with the canonical contact distribution $\Delta$  induced by the tautological $1$-form (the Liouville form )  in $T^*M$.
%the canonical contact distribution of $\mathbb P T^*M$.
%The set $\mathcal H$ inherits the structure of a fiber bundle (over $M$) from $T^*M$. \emph{The dynamics of the fibers of $W_D$ by the ``flow of extremals'' along the given extremal $\gamma$ can be encoded by a curve of isotropic subspaces (Lagrangian subspaces in the case of sub-Riemannian structures and in the case of rank 2 distributions) in the hyperplane $\Delta(\gamma)$ of $T_\gamma N$}.

\subsection {Flag structures for bracket generating distributions}
\label{dprelim}
The more detailed description of this class of problems is given in  Example 3 of section \ref{motsec}.
 If as a geometric structure $\mathcal U$ a bracket  generating distribution $D$ (without any additional structure on it) is considered, then the time-minimal problem associated with $\mathcal U$ does not make sense: any two points can be connected by an admissible curve to $D$ in an arbitrary small time. Instead, one can consider any variational problem on a space of admissible curves of this distribution with fixed endpoints.  Among all Pontryagin  extremals in most of the cases there are plenty of abnormal extremals that are the extremals with vanishing Lagrangian multiplier near the functional.
 % of the optimal problem is not equal to zero in the first case and equal to zero in the second one
 %that are exactly \emph{Pontryagin extremals} of this variational problem
%with zero Lagrange multiplier near the
%functional.
% Such extremal are called \emph{abnormal}.
%Abnormal extremals
Therefore they do not depend on the functional but on the distribution $D$ only.

Abnormal extremals live on the zero level set of the maximized Hamiltonian $H$ from  \eqref{maxH}, which is also called the annihilator of the distribution $D$ and is denoted by $D^\perp$.  It is more convenient to projectivize the picture by working in the projectivized cotangent bundle $\mathbb P T^*M$.
Then abnormal extremals foliate certain even-dimensional submanifold %$W_D$
$W_D$ of $\mathbb P D^\perp$ (see Example 3 of section \ref{motsec} for more detail), which plays the role of the manifold $\widetilde S$ in the general scheme. Similarly to the previous case, the second foliation is given
by the fibers of the canonical projection from $\mathbb P T^*M$ to $M$, restricted to $W_D$.
%of the  projectivized cotangent bundle $\mathbb P T^*M$
%(note that in Example 3 of section 3 we use a different notation for $\mathcal H_D$).
The role  of the manifold $\mathcal S_1$
%for the flag structure
in the general scheme
is played here by the space $\mathfrak A$ of the abnormal
extremals considered as the quotient of $W_D$ by the foliation of these extremals. The role of the distribution $\Delta$ in the flag structure on $\mathcal A$ is played by the natural contact distribution  induced on $\mathfrak A$ by the tautological $1$-form (the Liouville form )  in $T^*M$. The canonical symplectic form is defined, up to a multiplication by a non-zero constant, on each space $\Delta(\gamma)$. The linearization of the second foliation by the first one along any abnormal extremal $\gamma$ is
%the canonical contact distribution of $\mathbb P T^*M$.
%The set $\mathcal H_D$ inherits the structure of a fiber bundle (over $M$) from $\mathbb PT^*M$.
%The dynamics of the fibers of $\mathcal H_D$ by the flow of abnormal extremals along the given extremal $\gamma$ can be encoded by
a curve of isotropic subspaces with respect to this symplectic form in $\Delta(\gamma)$. %Keeping in mind the identification of tangent spaces to grassmannianqith spaces of certain linear operators, one can construct form this curve a curve of isotro
Collecting the osculating spaces of this curve of any order together with their skew symmetric complements w.r.t to the natural symplectic structure
%(the refinement procedure in \cite{flag2})
one finally assigns to $\gamma$
%{\bf Step 2:} \textsf {Construction of Jacobi curves of abnormal extremals.}
%Recall that on each contact hyperplane $\Delta(\gamma)$ of $T_\gamma N$ the symplectic structure is defined canonically up to a nonzero multiple.
%To any abnormal extremal $\gamma$ one can assign
a curve of isotropic/coisotropic subspaces  (or the curves of symplectic flags in the terminology of \cite{flag2}) in each space $\Delta(\gamma)$ called the \emph{Jacobi curve of  the
abnormal extremal $\gamma$}. This curve plays the role of $\mathcal Y^\gamma$ in the flag structure associated with the distribution $D$.

\subsection{Passage to flag structures versus Tanaka approach for distributions}
Why the passage to flag structures via the linearization or the sympelctification/linearization procedures is useful and even crucial in some cases?

First of all, making this passage, we immediately arrive to the (extrinsic) geometry of submanifolds  of a flag variety of a vector space $W$ with respect to the action of a subgroup $G$ of $GL(W)$, which  is simpler in many respects than the original equivalence problem. Assume that we reduced some geometric structure to a flag structure $(\Delta,\{\mathcal Y^\gamma\}_{\gamma \in\mathcal S}\bigr)$. Then any invariant of a submanifold $\mathcal Y^\gamma$ with respect to the natural action of the group $GL\bigl(\Delta(\gamma)\bigr)$ is obviously an invariant of the original equivalence problem. Moreover, in general a subgroup of $GL\bigl(\Delta(\gamma)\bigr)$, which is in fact isomorphic to the group of automorphism of the Tanaka symbol of the distribution $\Delta$ at $\gamma$, acts naturally on $\Delta(\gamma)$ (see subsection \ref{prelim} below for detail). Therefore, any invariant of a submanifold of $\mathcal Y^\gamma$ with respect to the natural action of this group, which in general might be a proper subgroup of $GL\bigl(\Delta(\gamma)\bigr)$,   is an invariant of the original equivalence problem. In many situation \emph{this gives a very fast and efficient way to construct and compute important invariants of the original structures.}

For example, in the cases of scalar differential equations up to contact transformations (Example 1 of section \ref{motsec}) and of rank $2$ distributions (a particular case of Example 3 of section \ref{motsec}) the curves $\mathcal Y^\gamma$ are curves of complete flags that can be recovered by osculation from the curves of their one-dimensional subspaces, i.e. from curves in projective spaces.
The classical  Wilczynski invariants of curves in projective spaces ~\cite{wil} immediately produce invariants of the original structures.

In the case of rank 2 distributions these curves in projective spaces are not arbitrary but they are so-called self-dual
%in projective spaces
\cite{var}.
In particular, the  first nontrivial Wilczynski invariant of self-dual curves in projective spaces produces the invariant of rank $2$ distributions in $\mathbb R^5$  which coincides with the famous  \emph{Cartan covariant binary biquadratic
form} of rank $2$ distributions in $\mathbb R^5$  \cite{cartan,zelcart, zelnur}. It gives the new and quite effective way to compute this Cartan invariant and to generalize it to rank $2$ distributions in $\mathbb R^n$ for arbitrary $n\geq 5$ \cite{zelvar}.
%Note that in the case of rank 2 distributions these curves in projective spaces are not arbitrary but so-called self-dual curves in projective spaces %\cite{var}.
%The algebra of all invariants admits a basis of so-called fundamental invariants $W_3, \dots, W_{N+1}$ of order $3,\dots,N+1$ respectively.
%The Wilczynski invariants of the curve $\mathcal Y_{-1}^\gamma$, taken for every prolonged solution $\gamma$, define the contact invariants of the %original ODE, the so-called \emph{generalized Wilczinski invariants} (\cite{doub3}).

Moreover, the passage from a geometric structure to the corresponding flag structures via the linearization or the symplectification/linearization procedures allows one not only to construct some invariants but provides an effective  way to assign to this geometric structure a canonical (co)frame
%(or the structure of absolute parallelism)
on some (fiber) bundle over the ambient manifold. In some cases this way is much more uniform, i.e can be applied simultaneously to a much wider class of structures, than the classical approaches such as the Cartan equivalence method (or its algebraic version, developed by N. Tanaka \cite{tan, tan2}, see also surveys \cite{aleks, zeltan}) applied to a geometric structure directly.
%By more uniform in the previous sentence we mean that the construction of canonical frames can me implemented simultaniuo

Let us clarify the last point.
%what we mean by a more uniform way in the previous sentence.
In general, in order to construct canonical frames for  geometric structures, first, one needs to choose a basic characteristic of these geometric structures, then to choose the most simple homogeneous model among all structures with this characteristic, if possible, and finally to imitate  the construction of the canonical frame for all structures with this characteristic by the construction of such frame for this simplest model. And it is desirable that the latter can be done without a further branching.
The main question is what basic characteristic to choose for this goal?

For example, in the Tanaka theory \cite{tan}, applied to distributions, as such basic characteristic one takes the so-called Tanaka symbol of a distribution at a point   (see also subsection \ref{prelim} below).  Algebraically a Tanaka symbol is a graded nilpotent Lie algebra. The simplest model among all distributions with the given constant Tanaka symbol is the corresponding left-invariant distribution on the corresponding Lie group. The construction of the canonical frames for all distributions with the given constant Tanaka symbol (i.e. such that their Tanaka symbols at all points are isomorphic one to each other as graded nilpotent Lie algebras) can be indeed imitated by its construction for the simplest homogeneous model and can be described purely algebraically in terms of the so-called universal algebraic prolongation of the Tanaka symbol (see subsection \ref{quasiprosec} below, especially Theorem \ref{tantheor} there).

However, it is hopeless in general to classify all possible graded nilpotent Lie algebras and the set of all graded  nilpotent Lie algebras contains moduli (continuous parameters). Therefore,  first,  generic distributions may have non-isomorphic Tanaka symbols at different points so that in this case Tanaka theory cannot be directly applied and, second, even if we restrict ourselves to distributions with a constant Tanaka symbol only, without the classification of these symbols we do not have a complete picture about the geometry of distributions.

Applying the symplectification/linearization procedure to particular classes of distributions (to rank $2$ distributions in \cite{doubzel2} and to rank $3$ distributions in \cite{doubzel3}) we realized that we can distinguish another basic characteristic of a distribution, which is coarser than the Tanaka symbol and, more importantly, classifiable and does not depend on continuous parameters.
This is the so-called symbol of the Jacobi curve of a generic abnormal extremal of the distribution at the generic point or shortly, the \emph{Jacobi symbol} of a distribution (see
%Example 3 of section \ref{motsec} and
our recent preprint \cite{jacsymb} for more detail). The notion of the symbol of a submanifold in a flag variety at a point was introduced in \cite{flag2, flag1}, see also subsection \ref{quasiflag} below. Algebraically such symbol is a subspace (a line in the case of curves) of degree $-1$ endomorphisms of a graded vector space, up to a natural conjugation by endomorphisms of nonnegative degrees, i.e. it is much more simple algebraic object than a Tanaka symbol. Informally speaking, it represents a type of the tangent space to the submanifold of flags.

For Jacobi curves of abnormal extremals the symbol at a point is a line of a degree $-1$ of a so-called graded symplectic space (see \cite[subsection 7.2]{flag2}) from a symplectic algebra of this space. All such lines, up to the conjugation by symplectic transformations, were classified in \cite[subsection 7.2]{flag2}. In particular, the set of all equivalence classes of such lines is discrete. This in turn gives the classification of all Jacobi symbols of distributions and leads to the following new formulation: \emph{to  construct uniformly canonical frames for all distributions with given constant Jacobi symbols}.

This problem leads in turn to a more general problem of construction of canonical frames for flag structures with given constant flag symbol and it was the main motivation for all developments of the present paper. The solution of the latter general problem is given by Theorem \ref{cor1} below, which is the direct consequence of the main result of the present paper, given by Theorem \ref{maintheor}.
Theorem \ref{cor1} shows that  the construction of canonical frames for distributions
 %He expect to give a solution to this problem
can be described in terms of natural algebraic operations on the Jacobi symbols in the category of graded Lie algebras.
%The solution of this problem is given, as a particular case
Note that from the fact that the set of all Jacobi symbols of distributions is discrete it follows that the assumption of the constancy of a Jacobi symbol holds automatically in a neighborhood of a generic point of an ambient manifold. More detailed applications of Theorem \ref{cor1} to the flag structures appearing after the symplectification/linearization procedure of distributions will be given in \cite{jacsymb}.

%, for example, to list all possible bundles, where these canonical frame lives, the number of prolongation steps, the properties of this frames, the symmetries group of the homogeneo

\subsection{Flag structures and $G$-structures on filtered manifolds}

Classical $G$-structures are defined as the reduction of the principal frame bundle $\mathcal F(\mathcal S)$ to a certain subgroup $G$ in $GL(V)$, where $V$ is a model tangent space to $\mathcal S$. In many cases the action of $G$ on a variety of flags of $V$
%$\mathcal F_\alpha(V)$
has a unique closed orbit $\mathcal Y$, and $G$ itself can be recovered as a symmetry group of this orbit.

A typical example is the irreducible action of $GL(2,\R)$ on any finite-dimensional vector space $V$. The induced action of $GL(2,\R)$ on the projectivization $P(V)$ has a rational normal curve as a unique closed orbit.

Assume that $\Delta$ coincides with the tangent bundle $T\mathcal S$. We say that the flag structure is of type $\mathcal Y$, if its fibers are equivalent to $\mathcal Y$ at all points of $M$. As the orbit $\mathcal Y$ in the flag variety
%\subset \mathcal F_\alpha(V)$
defines the subgroup $G\subset GL(V)$ uniquely an vice versa, we see that there is a one-to-one correspondence between flag structures of type $\mathcal Y$ and $G$-structures on $M$. For example, $GL(2,R)$-structures are defined as reductions of the principal frame bundle to the irreducible subgroup $GL(2,R)$ of $GL(V)$. Equivalently, in our terminology they are flag structures defined by a family of rational curves $\mathcal Y^\gamma \subset \mathbb P(T_\gamma \mathcal S)$ smoothly depending on the point $\gamma$. For more details on $GL(2,R)$-structures and their relationship with invariants of ODEs see \cite{doub3, dungod, duntod, godnur}.

Flag structures of type $\mathcal Y$ constitute a very special type of flag structures we study here, because we consider quite arbitrary distributions $\Delta$ (with constant Tanaka symbol) and we do not assume that the submanifolds $\mathcal Y^\gamma$ are isomorphic at different points. However, in some sense,  we imitate the construction of canonical frame for general flag structures by approximating them by
%constructions of such frames for
the flag structure with  submanifolds $\mathcal Y^\gamma$
%being isomorphic at different points and moreover
being the orbits of a certain subgroup of the group of automorphism of the Tanaka symbol of  $\Delta$.
This allows us to combine together both our version of construction of moving frames for submanifolds in flag varieties  \cite{flag1, flag2} and the prolongation theory for $G$-structures on filtered manifolds \cite{aleks,tan,zeltan}. As a result we get a powerful technique for solving the local equivalence problem for arbitrary flag structures.

\subsection{Construction of canonical frames for
%distributions with distinguished submanifolds of flags
flag structures: preliminary steps}
\label{prelim}
After considering separately the equivalence problems for several particular classes of differential equations and geometric structures via the linearization and the symplectification/linerization procedure we arrived to the necessity to develop a general approach to the equivalence problem of flag structures.

Let $\bigl(\Delta, \{\mathcal Y^\gamma\}_{\gamma \in\mathcal S}\bigr)$ be a flag structure.
To begin with, let us discuss the geometry of the distribution $\Delta$ itself.
%Let $\Delta$ be a rank $l$ distribution on a manifold $M$; that is, a rank $l$ subbundle of the tangent bundle $TM$.
%Two vector distributions
%$D_1$ and $D_2$ are called equivalent if there exists a
%diffeomorphism $F:M\rightarrow M$ such that
%$F_*D_1(x)=D_2(F(x))$ for any $x\in M$. Two germs of vector
%distributions $D_1$ and $D_2$ at the point $x_0\in M$ are
%called equivalent, if there exist neighborhoods $U$ and
%$\tilde U$ of $x_0$ and a diffeomorphism $F:U\rightarrow \tilde
%U$ such that
%\begin{equation*}
%%\begin{array}{c}
%F_*D_1=
%%(x)=
%D_2
%%(F(x)),\,\,\forall x\in U
%, \quad F(x_0)=x_0.
%%\end{array}
%\end{equation*}
%The general question is: When are two germs of distributions
 %%locally
 %equivalent?
%\subsection{Weak derived flags and symbols of distributions}
%Taking Lie brackets of vector fields tangent to a distribution $\Delta$ (i.e. sections of $\Delta$)   one can define a filtration
Let $\Delta^{-1}\subset \Delta^{-2}\subset\ldots$
be the \emph{weak derived flag}  (of $\Delta$), defined as follows:
%More precisely, set $\Delta=\Delta^{-1}$ and define recursively
%%The obvious (but very rough in the most cases) discrete invariant of a
%%distribution $D$ at $q$ is so-called \emph{ the small growth vectors} at $q$.
%%It is the tuple
%%%$\bigl(\dim D(q),\dim D^2(q),\dim D^3(q),\ldots\bigr)$,
%%$\{\dim D^j(q)\}_{j\in{\mathbb N}}$, where $D^j$ is the $j$-th power
%%of the distribution $D$, i.e.,
%$\Delta^{-j}=\Delta^{-j+1}+[\Delta,\Delta^{-j+1}]$, $j>1$.
Let  $X_1,\ldots X_l$ be $l$ vector fields constituting a local basis of a distribution $\Delta$, i.e. $\Delta= {\rm span}\{X_1, \ldots, X_l\}$
in some open set in $\mathcal S$. Then $\Delta^{-j}(\gamma)$ is the linear span of all iterated Lie brackets of these vector fields, of length not greater than  $j$,  evaluated at a point $\gamma$.
%A distribution $\Delta$ is called \emph{bracket-generating} (or \emph{completely nonholonomic}) if for any $x$ there exists $\mu(x)\in\mathbb N$ such that $\Delta^{-\mu(x)}(x)=T_x M$. The number $\mu(x)$ is called the \emph{degree of nonholonomy} of $\Delta$ at a point $x$.
%A distribution $\Delta$ is called \emph{regular} if for all $j<0$, the dimensions of subspaces $\Delta^j(x)$ are independent of the point $x$.
%From now on we assume that $\Delta$ is regular bracket-generating distribution with degree of nonholonomy $\mu$.

The basic characteristic of a distribution $\Delta$ at a point $\gamma$ is its \emph{Tanaka symbol}.
To define it let $\mg^{-1}(\gamma)\stackrel{\text{def}}{=}\Delta^{-1}(\gamma)$ and $\mg^{j}(\gamma)\stackrel{\text{def}}{=}\Delta^{j}(\gamma)/\Delta^{j+1}(\gamma)$ for $j<-1$. Consider the graded space
\begin{equation}
\label{symbdef}
\mathfrak{m}(\gamma)=\bigoplus_{j=-\mu}^{-1}\mg^j(\gamma),
\end{equation}
corresponding to the filtration
\begin{equation*}
\Delta(\gamma)=\Delta^{-1}(\gamma)\subset \Delta^{-2}(\gamma)\subset\ldots\subset \Delta^{-\mu+1}(\gamma)
\subset \Delta^{-\mu}(\gamma)=T_\gamma\mathcal S.
\end{equation*}
 This space is endowed naturally with the structure of a graded nilpotent Lie algebra, generated by
$\mg^{-1}(\gamma)$. Indeed, let $\mathfrak p_j:\Delta^j(\gamma)\mapsto \mg^j(\gamma)$ be the canonical projection to a factor space. Take $Y_1\in\mg^i(\gamma)$ and $Y_2\in \mg^j(\gamma)$. To define the Lie bracket $[Y_1,Y_2]$ take a local section $\widetilde Y_1$ of the distribution $\Delta^i$ and
a local section $\widetilde Y_2$ of the distribution $\Delta^j$ such that $\mathfrak p_i\bigl(\widetilde Y_1(\gamma)\bigr) =Y_1$
and $\mathfrak p_j\bigl(\widetilde Y_2(\gamma)\bigr)=Y_2$. It is clear that $[Y_1,Y_2]\in\mg^{i+j}(\gamma)$. Put
\begin{equation}
\label{Liebrackets}
[Y_1,Y_2]\stackrel{\text{def}}{=}\mathfrak p_{i+j}\bigl([\widetilde Y_1,\widetilde Y_2](\gamma)\bigr).
\end{equation}
It is easy to see that the right-hand side  of \eqref{Liebrackets} does not depend on the choice of sections $\widetilde Y_1$ and
$\widetilde Y_2$.
Besides, $\mg^{-1}(\gamma)$ generates the whole algebra $\mathfrak{m}(\gamma)$.
A graded Lie algebra satisfying the last property is called \emph{fundamental}.
The graded nilpotent Lie algebra $\mathfrak{m}(\gamma)$ is called the \emph {Tanaka symbol of the distribution $\Delta$ at the point $\gamma$}.

For simplicity assume that the Tanaka symbols $\mathfrak{m}(\gamma)$ of the distribution $\Delta$ at the point $\gamma$ are isomorphic, as graded Lie algebra, to a fixed fundamental graded Lie algebra $\mathfrak m= \displaystyle{\bigoplus_{i=-\mu}^{-1}} \mg^i$.
%Fix a fundamental graded nilpotent Lie algebra $\mathfrak{m}=\displaystyle{\bigoplus_{i=-\mu}^{-1} \mg^i}$.
In this case  $\Delta$ is said to be of \emph{constant symbol $\mathfrak{m}$} or of \emph{constant type $\mathfrak{m}$}.
% if for any
% $x$ the symbol $\mathfrak{m}(x)$ is isomorphic to $\mathfrak{m}$ as a nilpotent graded Lie algebra.
Note that in all our motivating Examples 1-4 given in section \ref{motsec} the distribution $\Delta$ is of constant type: in Examples 1, 2, and 4 $\Delta=T\mathcal S$ and the symbol  $\mathfrak m$ is graded trivially, $\mathfrak m=\mg^{-1}$, i.e. $\mathfrak m$ is the commutative Lie algebra of dimension equal to $\dim \mathcal S$; in example 3
%  $\mg^0=\gl(\mathfrak m)$ and any filtration of $\mg^{-1}$ is compatible
%with the algebra $\mg^0(\mathfrak m)$ with respect to the grading. Now assume that $M$ is odd-dimensional and
$\Delta$ is a contact distribution on $\mathcal S$
%Let $D$ be the contact distribution in $\mathbb R^{2n+1}$.
and its symbol
%$\mathfrak m_{\text{cont},n}$
is isomorphic to the Heisenberg algebra of dimension equal to $\dim \mathcal S$ with the grading $\mg^{-1}\oplus\mg^{-2}$, where $\mg^{-2}$ is the center.
%$\eta_{2n+1}$.

Note also that one can distinguish the so-called \emph{standard or flat distribution $D_{\mathfrak{m}}$ of constant type $\mathfrak{m}$}.
For this let $M(\mathfrak{m})$ be the simply connected Lie group with the
Lie algebra $\mathfrak{m}$ and let $e$ be its identity. Then $D_\mathfrak{m}$ is the left invariant distribution on $M(\mathfrak{m})$ such that $D_{\mathfrak{m}}(e)=\mg^{-1}$. The distribution $D_{\mathfrak{m}}$ is in a sense the most simple one among all distributions of constant type $\mathfrak{m}$. Note that in all our motivating Examples 1-4 given in section \ref{motsec} the distribution $\Delta$ has a trivial local geometry, i.e. it is locally equivalent to the flat distribution with the same symbol. However, we do not need to assume such local triviality to develop our theory.
%our general theory works for distribution $\Delta$ with non trivial local geometry as well

Further, to a distribution $\Delta$ with constant symbol $\mathfrak m$ one can assign a natural principle bundle over $\mathcal S$.
%In general this assumption is quite restrictive.
%For example, in the case of rank two distributions on manifolds with $\dim\,M\geq 9$,
%symbol algebras  depend on continuous parameters, which implies that generic rank 2 distributions in these dimensions do not have a
%constant symbol.
%For rank 3 distributions with $\dim D^{-2}=6$  the same holds in the case $\dim M=7$ as was shown in \cite{kuz}.
%%the continuous parameters appear in symbols of rank
%%(3,6,\ldots )-distributions
%%(see models $m7\_3\_3(\alpha)$ and
%%$m7\_3\_13(\alpha)$ there), and there are 6 more non-isomorphic symbols in
%%addition to that.
%Following Tanaka, and for simplicity of presentation, we consider here distributions of constant type $\mathfrak{m}$ only.
%
%\subsection{The bundle $P^0(\mathfrak{m})$ and its reductions}
%To a distribution of type $\mathfrak{m}$ one can assign a principal bundle in the following way.
Let $G^0(\mathfrak{m})$ be the group of automorphisms
of the graded Lie algebra $\mathfrak{m}$; that is, the group of all automorphisms $A$ of the linear space $\mathfrak{m}$ preserving both the Lie brackets ($A([v,w])=[A(v),A(w)]$
for any $v,w\in \mathfrak{m}$) and the grading ($A (\mg^i)=\mg^i$ for any $i<0$).
%Let
%$\text{Iso}\bigl(\mathfrak{m},\mathfrak m(x)\bigr)$ be the set of all graded Lie algebra isomorphisms $\vf:
%\mathfrak m\mapsto \mathfrak m(x)$ and
Let
%\begin{equation}
%\label{P0}
$P^0(\mathfrak m)$ be the set of all pairs $(\gamma,\vf)$, where  $\gamma\in \mathcal S$ and  $\vf:\mathfrak{m}\to\mathfrak m(\gamma)$
is an isomorphism of the graded Lie algebras $\mathfrak {m}$ and $\mathfrak m(\gamma)$.
 Then $P^0(\mathfrak m)$ is a principal $G^0(\mathfrak m)$-bundle over $\mathcal S$. The right action $R_A$ of an automorphism $A\in G^0(\mathfrak{m})$ is as follows:  $R_A$ sends $(\gamma,\vf)\in P^0(\mathfrak m)$ to $(\gamma,\vf\circ A)$, or shortly $(\gamma,\vf)\cdot R_A=(\gamma,\vf\circ A)$. Note that since $\mg^{-1}$ generates
 %the whole
 $\mathfrak m$, the group $G^0(\mathfrak{m})$ can be identified with a subgroup of $\text{GL}(\mg^{-1})$. By the same reason a point $(\gamma,\vf)\in P^0(\mathfrak m)$ of a fiber of $P^0(\mathfrak m)$ is uniquely defined by $\vf|_{\mg^{-1}}$. So one can identify
$P^0(\mathfrak m)$ with the set of pairs $(\gamma,\psi)$, where  $\gamma\in \mathcal S $ and $\psi:\mg^{-1}\to \Delta(\gamma)$  can be extended to an automorphism of the graded Lie algebras $\mathfrak {m}$ and $\mathfrak m(\gamma)$.
% }\text{Iso}\bigl(\mathfrak{m},\mathfrak m(x)\bigr)\}$$
Speaking informally, $P^0(\mathfrak m)$ can be seen as a $G^0(\mathfrak{m})-$reduction of the bundle of all frames of the distribution $\Delta$.
Besides, the corresponding Lie algebra $\mg^0(\mathfrak m)$ is the algebra of all derivations $d$ of $\mathfrak m$, preserving the grading
(i.e. $d \mg^i\subset \mg^i$ for all $i<0$) and it can be identified with a subalgebra of $\gl(\mg^{-1})$.

%Let, as before, $D_{\mathfrak m}$ be the left invariant distribution on $M(\mathfrak{m})$ such that $D_{\mathfrak m}(e)=\mg^{-1}$. Denote by $L_x$ the left translation on $M(\mathfrak m)$ by an element $x$. Finally, let $P^0(\mathfrak m, \mathfrak g^0)$
%be the set of all pairs $(x,\vf)$,
%where  $x\in M(\mathfrak m)$ and  $\vf:\mathfrak{m}\to\mathfrak m(x)$ is an isomorphism of the graded Lie algebras $\mathfrak {m}$ and $\mathfrak m(x)$ such that
%$(L_{x^{-1}})_*\vf\in G^0$. The bundle $P^0(\mathfrak m, \mathfrak g^0)$ is called \emph{the flat structure of constant type $(\mathfrak m, \mg^0)$}.

If $\Delta=T\mathcal S$ (as in Examples 1, 2, and 4 of section \ref{motsec}), then  $G^0(\mathfrak m)={\rm GL}(\mathfrak m)$, and $P^0(\mathfrak m)$ coincides with the bundle $\mathcal F(\mathcal S)$ of all frames on $\mathcal S$. In this case $P^0$ is nothing but a usual $G^0(\mathfrak m)$-structure. If $\Delta$ is a contact distribution
(as in example 3), then a non-degenerate skew-symmetric form $\Omega$ is well defined on $\mg^{-1}$, up to a multiplication by a nonzero constant.
The group $G^0(\mathfrak m)$ of automorphisms of $\mathfrak m$ is isomorphic to the group $\text{CSP}(\mg^{-1})$ of conformal symplectic transformations of $\mg^{-1}$, i.e. transformations preserving the form $\Omega$, up to a multiplication by a nonzero constant.

Note that on each fiber $\Delta(\gamma)$ a group $G^0_\gamma$ of automorphism of the symbol $\mathfrak m(\gamma)$ acts naturally. Obviously, by constructions $G^0_\gamma$ is a subgroup of $GL\bigl(D(\gamma)\bigr)$  and it is isomorphic by conjugation to $G^0(\mathfrak m)$. For example, in the case when $\Delta$ is contact $G^0_\gamma$ is a group of all conformal symplectic transformations of $\Delta(\gamma)$ with respect to the natural conformal symplectic structure on $\Delta(\gamma)$.

\subsection{Quasi-principal bundles associated with flag structures}
Additional structures on the distribution $\Delta$  can be encoded as fiber subbundles of the bundle $P^0(\mathfrak m)$.
Since in our case the submanifolds of flags
%(equivalently, of filtrations)
on each fiber of $\Delta$ are given, it is natural to fix a filtration on the space $\mg^{-1}$
(nonincreasing by inclusion):
\begin{equation}
\label{filtg}
\{\mg_{j}^{-1}\}_{j\in \mathbb Z},\quad   \mg_{j}^{-1}\subset \mg_{j-1}^{-1},\quad  \mg_{j}^{-1}\subseteq\mg^{-1}, \quad \dim\, \mg_{j}^{-1}=\dim\,{\mathfrak  J}^{\gamma}_j(x), x\in\gamma.
\end{equation}
Then we can consider the subbundle $P^0_+(\mathfrak m)$ of $P^0(\mathfrak m)$ consisting of the pairs $(\gamma,\vf)$ such that $\vf:\mathfrak m\rightarrow \mathfrak m(\gamma)$ is an isomorphism  of the graded Lie algebras $\mathfrak m$ and $\mathfrak m(\gamma)$ such that the flag
$\{\vf\bigl(\mg_{j}^{-1}\bigr)\}_{j\in \mathbb Z}$ belongs to the submanifold of flags $\mathcal Y^\gamma$.

Next step is to distinguish  a fiber subbundle of $P^0_+(\mathfrak m)$ of the minimal possible dimension naturally associated with  the flag structure $(\Delta,\{\mathcal Y^\gamma\}_{\gamma\in\mathcal S})$.
%a subbundle of $P^0_+(\mathfrak m)$ of the minimal possible dimension.
For this  one needs  to take a closer look to the (extrinsic) geometry of submanifold of flags $\mathcal Y^\gamma$ with respect to the natural action of  the group $G^0_\gamma$ ($\sim G^0(\mathfrak m)$). In \cite{flag2, flag1} we developed an algebraic version of Cartan's method of equivalence or an analog of Tanaka prolongation for submanifold of flags of a vector space $W$ with respect to the action of a subgroup $G$ of $GL(W)$.
Under some natural assumptions on the subgroup $G$ and on the flags (that are valid in all our motivating examples), one can pass from the filtered objects to the corresponding graded objects and describe the construction of canonical bundles of moving frames for these submanifolds in the language of pure linear algebra.

Since the paper is rather long and contains many technicalities, we find it worth to describe briefly these natural assumptions in the Introduction. For the detailed description see section \ref{quasisec} below. First of all, we assume that all flags in $\mathcal Y^\gamma$ lie in the same orbit under the natural action of the group $G^0_\gamma$. Further, let $\Lambda$ be a point in $\mathcal Y^\gamma$. Then $\Lambda$ defines a filtration of $\Delta(\gamma)$. This filtration induces the natural filtration of $\gl\bigl(\Delta(\gamma)\bigr)$ and therefore of any its subalgebra (see subsection \ref{comptsubsec} below). In particular, it induces the filtration on the Lie algebra $\mg_\gamma^0$ of the Lie group $G^0_\gamma$. Let $\gr \mg _\gamma^0$ be the graded space corresponding to this filtration. This space can be identified with a Lie subalgebra of $\gl\bigl(\gr \Delta(\gamma)\bigr)$, where  $\gr \Delta(\gamma)$ is the graded space corresponding to the filtration $\Lambda$.
%and , in general it is not isomorphic as a Lie algebra, to $\mg _\gamma^0$.

Our next assumption called \emph{compatibility with respect to the grading} is that for any $\Lambda\in \mathcal Y^\gamma $  and any $\gamma \in\mathcal S$ the algebra  $\gr \mg _\gamma^0$
is conjugate to the algebra $\mg_\gamma^0$ via an isomorphism between  $\gr \Delta(\gamma)$ and $\Delta(\gamma)$ (for more rigorous formulation see Definition \ref{compatdefin} below).

Consequently, the algebra $\gr \mg _\gamma^0$ is conjugate to the algebra $\mg^0(\mathfrak m)$ of all derivation of the Tanaka symbol $\mathfrak m$ of the distirbution $\Delta$. This assumption guaranties that the passage to the graded objects does not change the structure group of the bundle $P^0(\mathfrak m)$. In symplectic and conformal cases this is equivalent to considering only isotropic/coisotropic flags, which are preserved by the skew orthogonal or orthogonal complement respectively.

Another assumption is called the \emph{compatibility with respect to the differentiation}. For simplicity first describe it briefly in the case when $\mathcal Y^\gamma$ are curves (for more detail see condition (F3) in subsection \ref{quasiflag} below).
Assume that the curve  $\mathcal Y^\gamma$ is parametrized somehow: $x\mapsto \mathcal Y^\gamma(x)$.
Each curve $\mathcal Y^\gamma_i$ in the corresponding Grassmannian of $\Delta(\gamma)$ can be considered as the tautolological vector bundle over itself with the fiber over a point $\mathcal Y^\gamma_i(x)$ equal to the space $\mathcal Y^\gamma_i(x)$. Taking the linear span of $\mathcal Y^\gamma_i(x)$ with all tangent lines to the sections of this vector bundle (evaluated at points over $\mathcal Y^\gamma_i(x)$), we obtain the \emph{first osculating subspace of  $\mathcal Y^\gamma_i$ at $\mathcal Y^\gamma_i(x)$}. If $\mathcal Y^\gamma$ is a submanifold of flags of arbitrary dimension, the first osculating subspace of  $\mathcal Y^\gamma_i$ at $\mathcal Y^\gamma_i(x)$ is defined as the linear combination of the first osculating subspaces  at $\mathcal Y^\gamma_i(x)$ of all smooth curves on $\mathcal Y^\gamma_i$ passing through this point.
Our next assumption is that  for any integer $i$ this osculating subspace belongs to the space $\mathcal Y^\gamma_{i-1}(x)$ at any point of at $\mathcal Y^\gamma_i(x)$  (see equation \eqref{compatcurve} below). This assumption is natural, because in many cases one starts with a flag structure such that   $\mathcal Y^\gamma$ is a submanifold in a Grassmannian and then using consequent osculations one produces the flag structure compatible with respect to the differentiation (see also the refinement (osculation) procedure in \cite[section 6]{flag2}). Similar conditions appear as infinitesimal period relations on variations of
Hodge structures (\cite{grif}).

The last assumption implies that the tangent space to $\mathcal Y^\gamma$ at $\mathcal Y^\gamma_i(x)$ can be identified with a commutative subalgebra in a space of degree $-1$ endomorphisms of the graded space $\gr \mg_\gamma^0$ ($\sim \mg^0(\mathfrak m)$) corresponding to the filtration on $\mg_\gamma^0$ induced by the $\mathcal Y^\gamma_i(x)$ on $\Delta(\gamma)$.
%Moreover, this subalgebra belongs to the degree $-1$ component of the algebra $\gr \mg_\gamma^0$ ($\sim \mg^0(\mathfrak m)$).
This subalgebra is called the \emph{(flag) symbol of the flag submanifold  $\mathcal Y^\gamma$ at $\mathcal Y^\gamma_i(x)$}.
Identifying fibers of the distribution $\Delta$ with $\mg^{-1}$, one can associate the flag symbol with an orbit of a commutative subalgebra of $\mg^0(\mathfrak m)$ under the adjoint action of the subgroup $G^0_+(\mathfrak m)$ of the $G^0(\mathfrak m)$ preserving the filtration \eqref{filtg}. Moreover, this subalgebra consists of degree $-1$ endomorphisms under an identification of $\mg^{-1}$ with $\gr \mg^{-1}$.
Finally, we assume that all flag
%subalgebras of endomorphisms
symbols (for any $\gamma\in \mathcal S$ and any parameter $x$) are conjugated one to each other (via isomorphisms of the appropriate vector spaces) and, consequently,  correspond to the orbit  of one commutative subalgebra $\delta$ of $\mg^0(\mathfrak m)$ under the adjoint action of the subgroup $G^0_+(\mathfrak m)$.
%where the corresponding endomorphisms act).
In this case we say that the flag structure $\bigl(\Delta, \{\mathcal Y^\gamma\}_{\gamma \in\mathcal S}\bigr)$ has the \emph{constant (flag) symbol with a representative $\delta$.} In the sequel for shortness we will omit the word ``representative'' here.
As explained in Remark \ref{vinberg},  the assumption of constancy of flag symbol is not restrictive, at least locally, in the case when $\mathcal Y_\gamma$ are curves.

The flag symbol plays the same role in the geometry of submanifolds of flag varieties  as Tanaka symbols in the geometry of distributions. One can define the flat submanifold with constant symbol $\delta$ as an orbit of the filtration \eqref{filtg} under the natural action of the connected, simply connected subgroup of $\gl\bigl(\mg^{-1}\bigr)$ with the Lie algebra $\delta$. Denote by $\mathfrak u^F(\delta)$ the Lie subalgebra of $\mg^0(\mathfrak m)$ of the infinitesimal symmetries of the flat submanifold. Then, as follows from \cite{flag2}, to the flag structure $\bigl(\Delta, \{\mathcal Y^\gamma\}_{\gamma \in\mathcal S}\bigr)$  with constant Tanaka symbol $\mathfrak m$ of $\Delta$ and constant flag symbol $\delta$ one can assign canonically a fiber subbundle $P$ of the bundle $P^0(\mathfrak m)$ with the fibers of dimension equal to $\dim\, \mathfrak u^F(\delta)$, where the bundle $P^0(\mathfrak m)$ is defined in subsection \ref{prelim} above. Note that the algebra $\mathfrak u^F(\delta)$ can be described pure algebraically via a recursive procedure (see relations \eqref{kprolong} and \eqref{AUF} below).
%called the \emph {(flag) symbol of the flag structure $\bigl(\Delta, \{\mathcal Y^\gamma\}_{\gamma \in\mathcal S}\bigr)$.}
%Thus the general scheme
%for construction of the canonical frame for the distribution $\Delta$  with the distinguished submanifolds $\mathcal Y^\gamma$ of flags on each fiber
%consist of steps:

%\begin{enumerate}

%\item [Step 1] For each $\gamma\in \mathcal S$  to associate with the submanifold $F^\gamma$ the canonical bundle of moving frames of $\Delta_\gamma$ %with respect to the action $G^0_\gamma$, i.e. to distinguish a fiber subbundle $P$ of the bundle $P^0(\mathfrak m)$ of the smallest possible dimension. %This defines  the canonical bundle of frames over the original manifold $M$

     %(note that in general not for all types of curves \eqref{curve} this bundle is a $G$-structure).
    % Note that in general ithis canonical bundle is not a $G$-structure
%\item [Step 2] To make the prolongation of the bundle $P$ obtained in the previous step. If this bundle is a $G$-structure, then this reduces to the classical prolongation of $G$-structures.
%\end{enumerate}
%In other words, the whole prolongation procedure is given in two steps: the prolongation for curves and the prolongation for the frame bundle obtain from %the prolongation for curves.

The important point here is that \emph{this subbundle $P$ is as a rule  not  a principal subbundle of the bundle $ P^0(\mathfrak m)$} : the tangent spaces to the fiber of $P$ are identified with subspaces of $\mg^0(\mathfrak m)$ and these subspaces in general vary from one point to another point
of $P$.
%(see Remark \ref{nice} below for more detail).
%\str{For example,  if the subgroup of $G_\gamma^0$ preserving the submanifold $\mathcal Y^\gamma$, i.e. the group of symmetries of  $\mathcal Y^\gamma$ from $G_\gamma^0$, does not act transitively on $\mathcal Y^\gamma$, then $P$ is not a principal subbundle of the bundle $ P^0(\mathfrak m)$.}
%\str{Note that from the constructions of
%\cite{flag1, flag2}
%it follows that if
%$\widetilde{\mathcal S}$
%denotes the so-called
%\emph{tautological bundle over $\mathcal S$},
%i.e the bundle with the fibers over the $\gamma\in \mathcal S$ consisting of the points of the curve $\mathcal Y^\gamma$, then $P$ can be considered also as the bundle over  $\widetilde {\mathcal S}$.
%It turns out that $P$ can not always be chosen  as a principal bundle over $\widetilde {\mathcal S}$ as well.}
%\Red{Here I omitted two sentences compared to the last version, because they look superfluous (they appear in Remark \ref{nice})}
Hence, the classical Tanaka prolongation procedure  for the construction of the canonical frame for principal reductions of the bundle $P^0(\mathfrak m)$ (\cite{aleks, tan, zeltan}) in general can not be applied here.
However,  due to the presence of the additional filtration \eqref{filtg} on the space $\mg ^{-1}$ the subbundle $P\rightarrow \mathcal S$ satisfies a nice weaker property: although the tangent spaces to the fibers of $P$ at different points are different subspaces of $\mg^0(\mathfrak m)$ in general, the corresponding graded spaces (with respect to the filtration on these tangent spaces induced by the filtration on $\mg ^{-1}$) are the same. We call a fiber subbundles of $P^0(m)$ satisfying the last property the \emph{quasi-principle frame bundle}.  If, taking into account the identification of $\mg^{-1}$ and $\gr \mg^{-1}$, the graded spaces corresponding to the tangent spaces to the fibers are equal to a subalgebra $\mg^0$ of $\mg^0(\mathfrak m)$, then the quasi-principle frame bundle is said to be of type $(\mathfrak m, \mg^0)$ (for more detail see Definition \ref{quasidef} below). Note that in the case of the flag structure $\bigl(\Delta, \{\mathcal Y^\gamma\}_{\gamma \in\mathcal S}\bigr)$  with constant Tanaka symbol $\mathfrak m$ of $\Delta$ and constant flag symbol $\delta$ the canonical quasi-principle bundle $P$ is of type  $(\mathfrak m, \mathfrak u^F(\delta))$, i.e. $\mg^0=\mathfrak u^F(\delta)$ (Theorem \ref{quasidelta} below).

The main goal of the present paper is
%\str{to formalize this property, introducing the so-called \emph{quasi-principle subbundles of $P^0(\mathfrak m)$} (Definition \ref{quasidef}), and }
to generalize the Tanaka prolongation procedure to the quasi-principle subbundles of type $(\mathfrak m, \mg^0)$.
We show (Theorem \ref{maintheor}) that one can assign to such bundle the canonical frame on the bundle of dimension equal to the dimension of the universal Tanaka prolongation of the pair  $(\mathfrak m, \mg^0)$ (as in Definition \ref{univdef}).
 As a consequence we obtain a general procedure for the construction of canonical frames for flag structures of constant flag symbol.
 %$\bigl(\Delta, \{\mathcal Y^\gamma\}_{\gamma \in\mathcal S}\bigr)$  with constant Tanaka symbol $\mathfrak m$ of $\Delta$ and constant flag symbol %$\delta$.
In section \ref{motsec}  we give applications of this theory in a unified way (Theorem \ref{main2}) to natural equivalence problems for various classes of differential equations, bracket generating distributions, sub-Riemannian and more general structures on distributions. Sections \ref{firstprolongsec} and \ref{highprolongsec} are devoted to the proof of the main Theorem \ref{maintheor}.

\section{Quasi-principle frame bundles and main results}
\label{quasisec}
\setcounter{equation}{0}
\setcounter{theorem}{0}
\setcounter{lemma}{0}
\setcounter{proposition}{0}
\setcounter{definition}{0}
\setcounter{cor}{0}
\setcounter{remark}{0}
\setcounter{example}{0}
%Now we introduce the notion of a quasi-principle frame subbundle of the bundle $P^0(\mathfrak m)$.

\subsection{Compatibility of flags with respect to the grading}
\label{comptsubsec}
First, let us recall some basic notions on filtered and graded vector spaces. Here we follow \cite[section 2]{flag2}.
%As before assume that $\mathcal O$ is an orbit in $F_{k_1,\ldots,k_{\mu-1}}(W)$ with respect to the action of the group $G$.
Let $\{\Lambda_j\}_{j\in\mathbb Z}$
%\subseteq\Lambda_{-1}\subseteq \Lambda_{-2}\subseteq\ldots\subseteq\Lambda_{-\mu}=W\}$$ of $\mathcal O$ is
be a decreasing (by inclusion) filtration (flag) of a vector space $W$: $\Lambda_j\subseteq\Lambda_{j-1}$. It induces the decreasing filtration $\{(\gl(W))_i\}_{i\in \mathbb Z}$ of $\gl(W)$,
\begin{equation}
\label{grgldef}
 (\gl(W))_i=\{A\in \gl(W): A(\Lambda_j)\subset \Lambda_{j+i} \text{ for all j}\},\quad (\gl(W))_i\subset (\gl(W))_{i-1}.
\end{equation}
It also induces the filtration on any subspace of $\gl(W)$.
Further, let   $\gr W$ be  the graded
space corresponding to the filtration $\{\Lambda_j\}_{j\in\mathbb Z}$,
$$\gr W=\bigoplus_{i\in\mathbb Z}\Lambda_i/\Lambda_{i+1}$$
and  let $\gr\,\mathfrak {gl}(W)$ be the graded space corresponding to the filtration \eqref{grgldef},
$$\gr\gl(W) =\bigoplus_{i\in\mathbb Z}(\gl(W))_i/(\gl(W))_{i+1}.$$
The space $\gr\,\mathfrak {gl}(W)$ can be naturally
identified with the space $\mathfrak {gl}\,(\gr W)$,
\begin{equation}
\label{idgr}
\gr\,\mathfrak {gl}(W)\cong \mathfrak {gl}\,(\gr W).
\end{equation}
Indeed, if $A_1$ and  $A_2$ from $(\gl(W))_i$ belong to the same coset of
$(\gl(W))_i/ (\gl(W))_{i+1}$, i.e. $A_2-A_1 \in (\gl(W))_{i+1}$, and if $w_1$ and $w_2$ from $\Lambda_j$ belong to the same coset of $\Lambda_j/\Lambda_{j+1}$,
i.e. $w_2-w_1\in \Lambda_{j+1}$, then $A_1 w_1$ and $A_2 w_2$ belong to the same coset of $\Lambda_{j+i}/\Lambda_{j+i+1}$.
%In this way to any coset of $(\gl(W))_{f_0,i}/ (\gl(W))_{f_0,i+1}$ one assigns  an element of $\mathfrak {gl}\,(\gr_{f_0} W)$ of degree $i$.
This defines a linear map from $\gr\,\mathfrak {gl}(W)$ to $\mathfrak {gl}\,(\gr W)$. It is easy to see that this linear map is an isomorphism.
 % where ${\rm gr}_{f_0} W$ is

Now let $g$
%\subset \mathfrak {gl}(W)$
be a Lie subalgebra of $\mathfrak {gl}(W)$.
%the group $G$.
The
filtration $\{\Lambda_j\}_{j\in\mathbb Z}$ induces the filtration $\{g_i\}_{i\in\mathbb Z}$ on $g$, where
$$g_i=(\gl(W))_{i}\cap g.$$ Let $\gr\, g$ be the graded space corresponding to this filtration.
Note that the space $g_{i}/g_{i+1}$ is naturally embedded into  the space $(\gl(W))_{i}/(\gl(W))_{i+1}$. Therefore,  $\gr\, g$ is naturally embedded into $\gr\,\gl(W)$ and, by above, $\gr\, g$
can be considered as a subspace of $\mathfrak {gl}\bigl(\gr W\bigr)$. It is easy to see that it
is a subalgebra of $\mathfrak {gl}\bigl(\gr W\bigr)$.

In general, the algebra $\gr g$ is not isomorphic to the algebra $g$ (see \cite[Example 2.1]{flag2}).
In order to develop an algebraic version of the Cartan prolongation procedure it is very important that the passage to the graded objects will not change the group in the equivalence problem. Therefore we have to impose that $\gr\, g$ and $g$ are conjugate as in the following
%for some (and therefore any) $f_0\in \mathcal O$. More precisely we will assume in the sequel the following
\medskip

\begin{definition}
\label{compatdefin}
We say that the filtration (the flag) $\{\Lambda_j\}_{j\in\mathbb Z}$ of $W$ is compatible with the algebra $g\subset \gl(W)$ with respect to the grading if
there exists an isomorphism $J:\gr W\mapsto W$ such that
\begin{enumerate}
\item $J(\Lambda_i/\Lambda_{i+1})\subset \Lambda_i$, $i\in\mathbb Z$;
% (here we set $\Lambda_0:=0$);
\item
$J$ conjugates
%the algebras $\mathfrak g$ and $g$,
the Lie algebras $\gr\,g$ and $g$
i.e.
%$g=\{A_t\circ X\circ A_t^{-1}: X\in \mathfrak g\}$
\begin{equation}
\label{gradg}
g=\{J\circ X\circ J^{-1}: X\in \gr\, g\}.
\end{equation}
\end{enumerate}
\end{definition}

%{\bf Assumption 1}\emph{(compatibility with respect to the grading) For some $f_0\in \mathcal O$ , $f_0=\{0=\Lambda_0\subseteq\Lambda_{-1}\subseteq \Lambda_{-2}\subseteq\ldots\subseteq\Lambda_{-\mu}=W\}$, there exists an isomorphism $J:\gr_{f_0} W\mapsto W$ such that
%\begin{enumerate}
%\item $J(\Lambda_i/\Lambda_{i+1})\subset \Lambda_i$, $-\mu\leq i\leq -1$;
%% (here we set $\Lambda_0:=0$);
%\item
%$J$ conjugates
%%the algebras $\mathfrak g$ and $g$,
%the Lie algebras $\gr_{f_0}\, g$ and $g$
%i.e.
%%$g=\{A_t\circ X\circ A_t^{-1}: X\in \mathfrak g\}$
%\begin{equation}
%\label{gradg}
%g=\{J\circ x\circ J^{-1}: x\in \gr_{f_0}\, g\}.
%\end{equation}
%\end{enumerate}}
%\medskip
%Note that from the transitivity of the action of $G$ on $\mathcal O$ it follows that if Assumption 1 holds for some $f_0\in\mathcal O$ then it holds for any other $f_0\in\mathcal O$.
%If Assumption 1 holds we say that the pair $(G,\mathcal O)$ is \emph{compatible with respect to the grading}.
% Besides, the Lie algebra $g$ has a grading via formula \eqref{gradg} and this grading is defined up to a conjugation.

\begin{remark}
\label{compgradrem}
{\rm Obviously, if $G=GL(W)$ or $SL(W)$, then any filtration of $W$ is compatible with $\mathfrak{gl}(W)$ or $\mathfrak{sl}(W)$ with respect to the grading. On the other hand, if $W$ is endowed with a symplectic form or a conformal symplectic form (i.e. a symplectic form defined up to a multiplication by a nonzero constant) and $G=Sp(W)$ or $CSp(W)$, then,  according to \cite[Proposition 2.2]{flag2}, a decreasing filtration  $\{\Lambda_j\}_{j\in\mathbb Z}$ of $W$ is compatible with $\mathfrak {sp}(W)$ or $\mathfrak{csp}(W)$ if and only if, up to a shift in the indices,  $(\mg^{-1}_{-i})^\angle=\mg^{-1}_{i-\nu}$
for some integer $\nu$, where $L^\angle$ denotes the skew-symmetric complement of a subspace $L$ with respect the symplectic form on $W$ ($\nu$ can be taken as $0$ or $1$). Filtration (flags), satisfying this property are called \emph{symplectic filtrations (flags)}. Note that all flags appearing in Examples 3 and 4 are symplectic.} $\Box$
\end{remark}

% of $W$ $G$ acts transitively on any flag variety $F_{k_1,\ldots,k_{\mu-1}}(W)$  and the pair  $(G,F_{k_1,\ldots,k_{\mu-1}}(W))$ is compatible with respect to the grading.

%Note that if $\Delta=T\mathcal S$ then $\mathfrak m=\mg^{-1}$, $\mg^0=\gl(\mathfrak m)$ and any filtration of $\mg^{-1}$ is compatible
%with the algebra $\mg^0(\mathfrak m)$ with respect to the grading. Now assume that %$\mathcal S$ is odd-dimensional and
%$\Delta$ is a contact distribution on $\mathcal S$. As previously mentioned
%%Let $D$ be the contact distribution in $\mathbb R^{2n+1}$.
%%Its symbol
%%%$\mathfrak m_{\text{cont},n}$
%%is isomorphic to the Heisenberg algebra of dimension equal to $\dim \mathcal S$ with grading $\mg^{-1}\oplus\mg^{-2}$, where $\mg^{-2}$ is the center.
%%%$\eta_{2n+1}$.
%%Obviously,
%a symplectic form $\Omega$ is well defined on $\mg^{-1}$, up to a multiplication by a nonzero constant and
%%The group $G^0(\mathfrak m
%%_{\text{cont},n}
%%)$ of automorphisms of $\mathfrak m
%%_{\text{cont},n}
%%$ is isomorphic to the group $\text{CSP}(\mg^{-1})$ of conformal symplectic transformations of $\mg^{-1}$, i.e. transformations preserving the form %$\sigma$, up to a multiplication by a nonzero constant.
%and $\mg^0=csp(\mg^{-1})$. According to \cite[Proposition 2.2]{flag2} a filtration \ref{filtg} of $\mg^{-1}$ is compatible with $\mg^0(\mathfrak m)$ with respect to the grading if and only if $(\mg^{-1}_{-i})^\angle=\mg^{-1}_{i-\nu}$
%for any $0\leq i\leq \nu$, where $L^\angle$ denotes the skew-symmetric complement of a subspace $L$ with respect to the natural symplectic form $\sigma$.

\subsection{Quasi-principal subbundles of $P^0(\mathfrak m)$}
Now we are ready to introduce the notion of a quasi-principle bundle. Let, as before, $\Delta$ be a distribution with a constant Tanaka symbol
$\mathfrak m=\displaystyle{\bigoplus_{i=-\mu}^{-1}}\mg^i$ and assume that the space $\mg^{-1}$ has an additional filtration
\begin{equation}
\label{filtg1}
\{\mg_{j}^{-1}\}_{j\in \mathbb Z},\quad   \mg_{j}^{-1}\subset \mg_{j-1}^{-1},\quad  \mg_{j}^{-1}\subseteq\mg^{-1}.
 %\quad \dim\, \mg_{j}^{-1}=\dim\,{\mathfrak  J}^{\gamma}_j.
\end{equation}

Let the bundle $P^0(\mathfrak m)$ be the bundle defined in subsection \ref{prelim}.
Let $P$ be a fiber subbundle of $P^0(\mathfrak m)$ and
$P(\gamma)$ be the fiber of $P$ over the point $\gamma$.
%let again $P$ be a fibre subbundle of $\F_0(\mathfrak D)$ and
%let $P_x=\pi^{-1}(x)\cap P$ be its fiber over $x$. Using the
%identifications above and the immersion $\mathfrak F$, one  gets
%that
Take $\vf \in P(\gamma)$.
The tangent space $T_\vf\bigl(P(\gamma)\bigr)$ to the fiber $P(\gamma)$ at
a point $\vf$ can be identified with a subspace of $\mathfrak g^0(\mathfrak m)$.
%, whichwill be denoted by $L_\psi$.
Indeed, define the following $\mathfrak g^0(\mathfrak m)$-valued $1$-form $\Omega$ on $P$: to any vector $X$ belonging to
$T_\vf\bigl(P(\gamma)\bigr)$ we assign an element $\Omega(\vf)(X)$ of
%$\gl(V)$
$\mathfrak g^0(\mathfrak m)$
as follows: if $s\to \vf(s)$ is a smooth curve in $P(\gamma)$ such that $\vf(0)=\vf$ and
$\vf'(0)=X$ then let
\begin{equation}
\label{omega}
\Omega(\vf)(X)=\vf^{-1}\circ X,
\end{equation}
where in the right hand side of the last
formula $\vf$ is considered as an isomorphism between $\mathfrak m$ and $\mathfrak m(\gamma)$ .
%corresponding to the frame
%$\mathfrak F\bigl(p(s)\bigr)$.
Note that the linear map $\Omega(\vf): T_\vf\bigl(P(\gamma)\bigr)\mapsto \mathfrak g^0(\mathfrak m)$ is injective.
Set

\begin{equation}
\label{Lpsi}
L_\vf^0:=\Omega(T_\vf\bigl(P(t)\bigr)
\end{equation}

If $P$ is a principle bundle, which is a reduction of the bundle $P^0(\mathfrak m)$ and $\mg^0\subset \mg^0(\mathfrak m)$ is the Lie algebra of the structure group of the bundle $P$, then the space $L_\vf^0$ is independent of $\vf$ and equal to $\mg^0$. We call this bundle $P$ a \emph{principle bundle of type $(\mathfrak m,\mg^0)$} or a \emph{Tanaka structure of type $(\mathfrak m,\mg^0)$} (as in \cite{aleks}).
%If this Lie algebra is $\mg^0$, then Tanaka called such bundle the \emph {structure of type (\mathfrak m, \mg^0)}.
%Here we consider more general class of subbundles of $P^0(\mathfrak m)$ that we call quasi.
For general subbundle $P$  subspaces $L_\vf^0$ may vary from point to point.
From the constructions of the previous subsection, the filtration \eqref{filtg1} induces the filtration on each space $L_\varphi$ defined in \eqref{Lpsi}
and the corresponding graded spaces
%%${\rm gr}\, \mg^{-1}$ the  corresponding to filtration \eqref{filtg}. Also, for any subspace $\mathcal L$ of  $\gl(\mg^{-1})$ denote
${\rm gr}\,L^0_\vf$
%%the graded space corresponding to the induced filtration on $\mathcal L$. Note that ${\rm gr} \gl(\mg^{-1})$ is identified with %$\gl({\rm gr} \mg^{-1})$ in a natural way. So ${\rm gr} \mathcal L$
can be considered as subspaces of $\gl({\rm gr}\, \mg^{-1})$. In many important applications these subspaces are independent of $\vf$. This motivates the following

%\begin{definition}
%\label{compatdefin}
%We say that filtration \eqref{filtg} of $\mg^{-1}$ is compatible with the algebra $\mg^0(\mathfrak m)\subset \gl(\mg^{-1})$ with respect to the grading if
%there exists an isomorphism $J:\gr \mg^{-1}\mapsto \mg^{-1}$ such that
%\begin{enumerate}
%\item $J(\mg^{-1}_i/\mg^{-1}_{i+1})\subset \mg^{-1}_i$, $-\nu\leq i\leq -1$;
%% (here we set $\Lambda_0:=0$);
%\item
%$J$ conjugates
%%the algebras $\mathfrak g$ and $g$,
%the Lie algebras $\gr\,\mg^0(\mathfrak m)$ and $\mg^0(\mathfrak m)$
%i.e.
%%$g=\{A_t\circ X\circ A_t^{-1}: X\in \mathfrak g\}$
%\begin{equation}
%\label{gradg}
%\mg^0(\mathfrak m)=\{J\circ X\circ J^{-1}: X\in \gr\, \mg^0(\mathfrak m)\}.
%\end{equation}
%\end{enumerate}
%\end{definition}

\begin{definition}
\label{quasidef}
Assume that $\mg^0$ is a subalgebra of $\mg^0(\mathfrak m)$. A fiber subbundle $P$ of the bundle $P^0(\mathfrak m)$ is called a quasi-principle bundle of type $(\mathfrak m, \mg^0)$ if the following three conditions hold:

\begin{enumerate}
\item The filtration \eqref{filtg} of $\mg^{-1}$ is compatible with the algebra $\mg^0(\mathfrak m)$ with respect to the grading;
\item The subspaces $\gr\, L_\vf^0$ of $\gl({\rm gr}\, \mg^{-1})$ coincide for any $\vf\in P$;
\item There exist an isomorphism $J:\gr \mg^{-1}\mapsto \mg^{-1}$ as in Definition \ref{compatdefin} (with $W=\mg^{-1}$) such that $\mg^0=\{J\circ X\circ J^{-1}: X\in \gr\,  L_\vf^0\}$ for any $\vf\in P$.
    \end{enumerate}
\end{definition}
%Further, we assume that there exists an isomorphism $J:{\rm gr}\, \mg^{-1}\rightarrow \mg^{-1}$ conjugating ${\rm gr}\mg^0(\mathfrak %m)$ and $\mg^0(\mathfrak m)$, i.e. such that $\mg^0(\mathfrak m)=\{{\rm Ad}\,J\circ X: X\in {\rm gr}\,\mg^0(\mathfrak m)\}$, where %${\rm Ad} \,J X=J\circ X\circ J^{-1}$.
%Such $J$ exists, for example, if the algebra $\mg^0(\mathfrak m)$ is reductive. The subspace
%${\rm Ad} J{\rm gr} L_\vf$ of $\mg^{0}(\mathfrak m)$ is called
%the \emph{symbol of the bundle $P$
%at a point $\vf$}. If spaces ${\rm Ad} J{\rm gr}L_\vf$ are all the same and equal to a given subalgebra $\mg^0$ of $\mg^0(\mathfrak %m)$, we say that the bundle $P$ is  \emph{quasi-principal frame bundles of type $(\mathfrak m, \mg^0)$ subordinated to the filtration %\eqref{filtg} of $\mg^{-1}$}.
Note that if the filtration \eqref{filtg} on $\mg^{-1}$ is trivial (i.e. it does not contain any nonzero proper subspace of $\mg^{-1}$) then $P$ is, at least  locally, the principle bundle of type $(\mathfrak m, \mg^0)$ in the Tanaka sense.

%Our original motivation to study the quasi-principle bundles comes from so-called linearization/symplectification procedure for local differential geometry of differential equations and bracket-generating distributions.(\emph{Here I plan to write in more detail})
\subsection{Quasi-principal bundles associated with flag structures
%structures $\bigl(\Delta,\{\mathcal Y^\gamma\}_{\gamma\in\mathcal S}\bigr)$
with constant flag symbol}
\label{quasiflag}
 Our next goal is to show that under some natural assumptions and based on \cite{flag2}, one can assign to a structure $\bigl(\Delta,\{\mathcal Y^\gamma\}_{\gamma\in\mathcal S}\bigr)$ a quasi-principal bundle in a canonical way.

Recall that a group $G^0_\gamma$ (conjugate to $G^0(\mathfrak m)$) acts naturally on each $\Delta(\gamma)$. Let $g^0_\gamma$ be the Lie algebra
of the Lie group $G^0_\gamma$. Assume that the submanifolds $\mathcal Y^\gamma$ satisfy the following additional properties

\begin{enumerate}
\item[{\bf (F1)}] \emph{(transitivity)}
All flags in $\mathcal Y^\gamma$ lie in the same orbit under the natural action of the group $G^0_\gamma$;
\item[{\bf (F2)}]\emph{(compatibility with respect to the grading)}
Any flag in $\mathcal Y^\gamma$ is compatible with $g^0_\gamma$ with respect to the grading.
\end{enumerate}

Then, first we can fix a filtration on the space $\mg ^{-1}$ as in \eqref{filtg} such that it is compatible with $\mg^0(\mathfrak m)$ with respect to the grading. For this it is enough to fix one of the flags $\mathcal Y^\gamma(x)$ in the space $\Delta(\gamma)$ and to identify somehow the space $\mg^{-1}$ and the space $\Delta(\gamma)$ with this fixed flag.

 Further we will assume that the submanifolds of flags $\mathcal Y^\gamma$ are not arbitrary but satisfy a special property called the \emph {compatibility with respect to differentiation} in the terminology of \cite[Section 3]{flag2}.

First describe this property for curves of flags. Given a curve $t\mapsto L(t)$  in a certain Grassmannian of a vector space $W$, denote by $C(L)$ the canonical vector bundle over $L$: the fiber of $C(L)$ over the point $L(t)$ is the vector space $L(t)$. Let $\Gamma(L)$ be the space of all sections of the bundle $C(L)$ . Set
 %$\Lambda^{(i)}$ (or the $i$-th osculating space) and the $i$-th contraction $\Lambda_{(i)}$ of the curve $\Lambda(\cdot)$ as
 $L^{(0)}(t):=L(t)$ and  define inductively
$$L^{(j)}(t)={\rm span}\{\frac{d^k}{d\tau^k}\ell(t): \ell\in\Gamma(L), 0\leq k\leq j\}$$ for $j\geq0$.
The space $L^{(j)}(t)$ is called the \emph{$j$th extension} or the \emph{$j$th osculating subspace} of the  curve $L$ at the point $t$.
%\str{With this terminology, relation \eqref{compat} of property (P3) implies the following property.}
\medskip

 {\bf (F3)} %\str{If $\mathcal Y^\gamma$ is a curve of flags, we assume that}
 For any $i\in \mathbb Z$ the first extension (the first osculating subspace) of the curve $x\mapsto \mathcal Y_i^\gamma(x)$
 %\str{$,x\in \gamma$ at any point}
 for any parameter $x_0$ is contained in the space $\mathcal Y^\gamma_{i-1}(x_0)$:

\begin{equation}
\label{compatcurve}
({\mathcal Y}_i^{\gamma})^{(1)}(x_0)\subset \mathcal Y^\gamma_{i-1}(x_0), \quad \forall \text{ parameter } x_0 \text { and  }i\in\mathbb Z.
\end{equation}
This relation exactly means that the curve $\mathcal Y^\gamma=\{\mathcal Y_i^\gamma(x)\}_{i\in \mathbb Z}$ is compatible with respect to differentiation in the sense of \cite{flag2}. If $\mathcal Y^\gamma$ is a submanifold of flags of arbitrary dimension, we assume
% More generally, if $C$ has an arbitrary rank then any submanifold $\mathcal Y^\gamma=\{\mathcal Y_i^\gamma(x)\}_{i\in \mathbb Z}, x\in \gamma$
%is also compatible with respect to differentiation in a sense
that any smooth curve on it is compatible with respect to differentiation as above.
\medskip

Now let $G_+^0(\mathfrak m)$ be the subgroup of $G^0(\mathfrak m)$ consisting of all elements of $G^0(\mathfrak m)$ preserving the filtration \eqref{filtg1}.
The group $ G_+^0(\mathfrak m)$ acts naturally on the space $\bigl({\rm gr} \mg^{0}(\mathfrak m)\bigr)_{-1}$ of all degree $-1$ endomorphisms of the graded space ${\rm gr} \mg^{-1}$ (corresponding to the filtration \eqref{filtg}) belonging to ${\rm gr} \mg^{0}(\mathfrak m)$. This action is in essence the adjoint action, namely $A\in  G_+^0(\mathfrak m)$ sends $x\in\bigl({\rm gr} \mg^{0}(\mathfrak m)\bigr)_{-1}$  to the degree $-1$ component of $({\rm Ad} A)\,x$. This action induces the natural action on the Grassmannians of $({\rm gr} \mg^{0}(\mathfrak m))_{-1}$.

Further, recall that given a Grassmannian of subspaces in a vector space $W$ the tangent space at any point ($=$ a subspace)  $\Lambda$ to this Grassmannian can be naturally identified with the space  ${\rm
 Hom}\bigl(\Lambda, W/\Lambda\bigr)$. Taking into account that the submanifolds of flags $\mathcal Y^\gamma$ satisfies the compatibility with respect to differentiation property (F3), we can conclude from here that the tangent space to the submanifold $\mathcal Y^\gamma$ at a point  $\mathcal Y^\gamma(x)$ can be identified with
 a subspace in the space $$ \displaystyle{\bigoplus_{i\in\mathbb Z} {\rm
 Hom}\bigl(\mathcal Y^\gamma_i(x)/\mathcal Y^\gamma_{i+1}(x),\mathcal Y^\gamma_{i-1}(x) /\mathcal Y^\gamma _i(x)\bigr)}$$ or, in other words,
a subspace in the space of degree $-1$ endomorphisms of the graded space corresponding to the filtration $\mathcal Y^\gamma(x)$.
%${\rm gr}\mathcal Y^\gamma(x)$.

Moreover, taking into account properties (F1) and (F2) and the construction of the filtration \eqref{filtg}, we can identify the tangent space
 to the submanifold $\mathcal Y^\gamma$ at a point  $\mathcal Y^\gamma(x)$ with a subspace $\delta^\gamma_x$ of the space $\bigl({\rm gr} \mg^{0}(\mathfrak m)\bigr)_{-1}$ of all degree $-1$ endomorphisms of ${\rm gr} \mg^{-1}$ from ${\rm gr} \mg^{0}(\mathfrak m)$, defined up to the aforementioned natural action of the group $G_+^0(\mathfrak m)$.
 This subspace (or, more precisely, the orbit of this subspace with respect to this action) is called the \emph{symbol} of the submanifold $\mathcal Y_\gamma$ at the point $\mathcal Y^\gamma(x)$.

The symbols $\delta^\gamma_x$ play in the geometry of submanifolds of flags the same role as Tanaka symbols in the geometry of distributions.
 Note that the symbol $\delta^\gamma_x$  must be a commutative subalgebra of ${\rm gr} \mg^{0}(\mathfrak m)$ belonging to $\bigl({\rm gr} \mg^{0}(\mathfrak m)\bigr)_{-1}$. The latter condition follows from the involutivity of the tangent bundle to the submanifold $\mathcal Y^\gamma$.

Finally, we need the following property
\medskip

{\bf (F4)}\emph{(constancy of the flag symbol)}  Symbols $\delta^\gamma_x$ of the submanifold of flags $\mathcal Y_\gamma$ at a point $\mathcal Y^\gamma$  are independent of $x$ and $\gamma$ or, more precisely,  lie in the same orbit under the action of the group $G_+^0(\mathfrak m)$ on the corresponding Grassmannian of the space $\bigl({\rm gr} \mg^{0}(\mathfrak m)\bigr)_{-1}$. If $\delta$ is a point in this orbit we will say that the structure $\bigl(\Delta,\{\mathcal Y^\gamma\}_{\gamma\in\mathcal S}\bigr)$ has the \emph{constant flag symbol $\delta$ (and the constant Tanaka symbol $\mathfrak m$).}

\begin{remark}
\label{vinberg}
{\rm Under some natural assumptions the condition of constancy of the flag symbol is not restrictive at least in the case when $\mathcal Y_\gamma$ are curves of flags.
 If $G^0(\mathfrak m)$ is semismple (or, more generally, reductive), as  direct consequence of results  E.B. Vinberg (\cite{vinb})
 %in this case
 the set of all possible symbols of curves of flags %(for given group $G$)
 is finite.
 Hence, the symbol of a curve of flags with respect to  a semisimple (reductive) group  $G$  is constant in a neighborhood of a generic point.
 Note also that all flag symbols that may appear in Examples 1-4 of section \ref{motsec} are classified in \cite{flag2}, see subsection 7.1 there, corresponding to Examples 1 and 2 , and subsection 7.2 there, corresponding to Examples 3 and 4.
%the condition of constancy of symbol holds
}
$\Box$
\end{remark}

%\subsection{Algebraic prolongation of symbol}
%\label{univsect}
%Fix $\mathfrak m$ a line in $\mathfrak g_{-1}$ as
%in the previous section, representing the symbol of the curve \eqref{filt}.

As in Tanaka theory, one can define the notion of the flat submanifold of flags with constant symbol $\delta$. Taking into account the identification between the spaces ${\rm gr}\mg^{-1}$ and $\mg^{-1}$, consider the connected, simply connected subgroup $H(\delta)$ of $G^0(\mathfrak m)$ with the Lie algebra $\delta$. The \emph{flat submanifold with constant symbol $\delta$} is a submanifold equivalent to the orbit of the flag \eqref{filtg1} with respect to the natural action of $H(\delta)$ on the corresponding flag variety.

Again by analogy with the Tanaka theory, one can define the universal algebraic prolongation of the flag symbol $\delta$ (for the description of the algebraic prolongation of a Tanaka symbol see the next subsection and also \cite{tan,yam, zeltan}).
Let $\bigl({\rm gr} \mg^{0}(\mathfrak m)\bigr)_{k}$ be the component of degree $k$ of the space ${\rm gr} \mg^{0}(\mathfrak m)$ or equivalently the endomorphisms of degree $k$ of ${\rm gr} \mg^{-1}$ belonging to ${\rm gr} \mg^{0}(\mathfrak m)$.
Set $\mathfrak u_{-1}^F(\delta):=\delta$ and define by induction in $k$
\begin{equation}
\label{kprolong}
\mathfrak u_k^F(\delta):=\{X\in \bigl({\rm gr} \mg^{0}(\mathfrak m)\bigr)_{k}:[X,Y]\in \mathfrak u_{k-1}^F(\delta),\,Y \in\delta\},\quad
k\geq 0,
\end{equation}
where $\bigl({\rm gr} \mg^{0}(\mathfrak m)\bigr)_{k}$ denotes the space of all degree $k$ endomorphisms of the graded space ${\rm gr} \mg^{-1}$ (corresponding to the filtration \eqref{filtg}) belonging to ${\rm gr} \mg^{0}(\mathfrak m)$.
The space $\mathfrak u_k^F(\delta)$ is called the \emph {$k$th algebraic prolongation of the symbol $\delta$}.
Then by construction
\begin{equation}
\label{AUF}
\mathfrak u^F(\delta)=\displaystyle{\bigoplus_{k\geq -1}\mathfrak u_k^F(\delta)}
\end{equation}
 is a graded
subalgebra of ${\rm gr}\mg^0(\mathfrak m)$. It can be shown that it is \emph{the largest
graded subalgebra of the space ${\rm gr}\mg^0(\mathfrak m)$ such that its component
corresponding to the negative degrees coincides with
%$\mathbb R\delta$.
$\delta$.}
The algebra $\mathfrak u^F(\delta)$ is called the \emph{universal algebraic prolongation of the flag symbol $\delta$} (of a commutative subalgebra of ${\rm gr} \mg^{0}(\mathfrak m)$ belonging to $\bigl({\rm gr} \mg^{0}(\mathfrak m)\bigr)_{-1}$). As it is shown in \cite{doubkom}, the algebra of infinitesimal symmetries of the flat submanifold with the constant symbol $\delta$ (with respect to the action of the group $G^0(\mathfrak m)$) is isomorphic to $\mathfrak u^F(\delta)$.

Since by the constructions the filtration \eqref{filtg} is compatible with $\mg^0(\mathfrak m)$ with respect to the grading, the subalgebra  $\mathfrak u^F(\delta)$ of ${\rm gr}\mg^0(\mathfrak m)$ is conjugate to a subalgebra of $\mg^0(\mathfrak m)$ (see condition (2) of Definition \ref{compatdefin}). Taking into account this conjugation, from now on we will look on $\mathfrak u^F(\delta)$ as on the subalgebra of $\mg^0(\mathfrak m)$. The procedure for the construction of a canonical bundle of moving frames for a submanifold in a flag variety with a constant symbol is described  in \cite{flag2,flag1}.
As a direct consequence of \cite[Theorem 4.4]{flag2} in the case of curves of flags  and its generalization to submanifolds of flags described in subsection 4.6 there, applied to each submanifold $\mathcal Y^\gamma$,  we obtain the following

\begin{theorem}
\label{quasidelta}
Given a structure $\bigl(\Delta,\{\mathcal Y^\gamma\}_{\gamma\in\mathcal S}\bigr)$ with the constant Tanaka symbol $\mathfrak m$  (of $\Delta$) and with  the constant flag symbol $\delta$ one can assign to it in a canonical way a quasi-principal subbundle of $P^0(\mathfrak m)$ of type $\bigl(\mathfrak m, \mathfrak  u^F(\delta)\bigr)$.
\end{theorem}

Note that the assignment in Theorem \ref{quasidelta} is uniquely determined by a choice  of a so-called normalization condition for the structure equation of the moving frames associated with the submanifolds $\mathcal Y^\gamma$.
%, \cite{flag2, flag1}
In the case of curves the normalization conditions are given  by a subspace  $W$  complementary to the subspace $\mathfrak u^F_+(\delta)+[\delta, \mathfrak g^0_+(\mathfrak m)]$ in $\mathfrak g^0_+(\mathfrak m) $ , where $\mg^0_+(\mathfrak m)$ is the Lie algebra of the group $G^0_+(\mathfrak m) $ and $\mathfrak u^F_+(\delta)= u(\delta)\cap \mathfrak g^0_+(\mathfrak m) $. Similar description of normalization conditions can be given in the case of submanifolds of flags \cite{flag2, flag1}.

The important point is that as a rule the resulting quasi-principal subbundle $P$ of $P^0(\mathfrak m)$ in Theorem \ref{quasidelta} is not a principle bundle of type
$\bigl(\mathfrak m, \mathfrak  u^F(\delta)\bigr)$.
For example,  if the subgroup of $G_\gamma^0$ preserving the submanifold $\mathcal Y^\gamma$, i.e. the group of symmetries of  $\mathcal Y^\gamma$ from $G_\gamma^0$, does not act transitively on $\mathcal Y^\gamma$, then $P$ is not a principle bundle of type
$\bigl(\mathfrak m, \mathfrak  u^F(\delta)\bigr)$.
%a principal subbundle of the bundle $ P^0(\mathfrak m)$.

\begin{remark} {(On nice flag symbols.)}
\label{nice}
{\rm Let $P$ denote  the resulting quasi-principal subbundle of Theorem \ref{quasidelta}.
One can distinguish a special class of flag symbols for which the normalization conditions for the construction of the bundle $P$ can be chosen such that $P$ has a nicer property at least on the level of the submanifolds of flags. More precisely,
assign to our flag structure $\bigl(\Delta,\{\mathcal Y^\gamma\}_{\gamma\in\mathcal S}\bigr)$ the bundle $\widetilde {\mathcal S}$ over $\mathcal S$ such that the fiber over $\gamma \in\mathcal S$ consists of the points of the curve $\mathcal Y^\gamma$. Let us call this bundle the \emph{tautological bundle over $\mathcal S$}. Then  from the same  \cite[Theorem 4.4]{flag2} it follows that $P$ can be considered as the bundle over $\widetilde{\mathcal S}$ as well.

However, even the bundle $P\rightarrow\widetilde{\mathcal S}$ can not always be chosen as a principle bundle.
To formulate the precise conditions for the latter, let $U^F_+(\delta)$ be the subgroup of the group of symmetries  of the flat submanifold with the constant symbol $\delta$
(with respect to the action of the group $G^0(\mathfrak m)$), which preserves the filtration \eqref{filtg1}. The bundle  $P\rightarrow \widetilde {\mathcal S}$ is the principal bundle if and if the complementary subspace $W$ defining the normalization condition for the construction of $P$ is invariant with respect to the adjoint action of $U^F_+(\delta)$. In the latter case $U^F_+(\delta)$ is the structure group of this bundle.
%Note also that in some case a normalization condutuin can be chosen
%\emph{it is not always possible to choose the complementary subspace $W$ such that  the resulting quasi-principal subbundle of $P^0(\mathfrak m)$ is in fact  principle bundle of type
%$\bigl(\mathfrak m, \mathfrak  u^F(\delta)\bigr)$.}
If such complementary space $W$ can be chosen then the symbol $\delta$ and the normalization condition $W$ will be called \emph{nice}. The fact that not any symbol of submanifolds of flags is nice is the main reason why we need to introduce the quasi-principle bundles and to develop the prolongation procedure for them.
%If such choice of a normalization condition can be maid we call it a \emph{nice normalization condition}. As pointed out in \cite{flag1, flag2} a normalization condition $W$ is nice if and only if $W$ is invariant w.r.t. the adjoint action of $\mathfrak u^F_+(\delta)^+$
%Such complementary space defines
%The criteria for the possibility to choose nice normalization conditions so that the resulting  subbundle is a principal reduction of the bundle %$P^0(\mathfrak m)$ are also given in there.

For example, the symbol of curve of flags obtained after the linearization procedure for scalar ordinary equations (Example 1 of section 3) and symplectification procedure  for rank $2$ distributions (a particular case of Example 3 of section 3) are nice. In both cases the curves of flags are generated by osculation from certain curves in projective spaces. However, this is not in general the case for systems of ODEs of mixed order and distributions of rank greater than $2$. See \cite{flag1} for the proof of non-existence of nice (in the above sense) normalization conditions for curves of flags appearing in ODEs of mixed order $(3,2)$ up to a 0-equivalence (in the sense of Example 2 of section 3 below). Similarly, it can be shown that a nice normalization conditions does not exist in the case of rank $3$ distributions in $\mathbb R^7$, having a 6 dimensional square. The questions on the classification of flag symbols, for which  a nice normalization condition can be chosen, remains open.} $\Box$
\end{remark}

%In the case , when $\mathcal F_\gamma$ are curves a normalization condition is a subspace   $\mathcal N$ complementary to the subspace $\mathfrak u^F(\delta)^{+}+[\delta, \mathfrak g^0_+(m)]$ in $\mathfrak g^0_+(m)$ , where $\mathfrak u(\delta)^+= u(\delta)\cap \mathfrak \mathfrak g^0_+(m)$. Such complementary space defines ``normalization conditions'' for the structure equation of the moving frame.
%Obviously, $\mathfrak {gl}(V)_k=0$ for all $k\geq \mu$. So $\mathfrak u_k=0$ for $k\geq \mu$.
%\subsection{Quasi-principal bundles associated with  Structures $\bigl(\Delta,\{\mathcal Y^\gamma\}_{\gamma\in\mathcal S}\bigr)$ with constant flag symbol}
%Now, based on \cite{flag2} we describe
\subsection{Prolongation of quasi-principal frame bundles}
\label{quasiprosec}
Now we return to general quasi-principle bundles.
In \cite{tan} for any principle bundle of type $(\mathfrak m, \mg^0)$ Tanaka described the prolongation procedure for the construction of the canonical frame (the structure of an absolute parallelism) by means of the so-called \emph {universal algebraic prolongation of the pair $ (\mathfrak m, \mg^0)$}. One of the main goals of the present paper is to generalize this prolongation procedure to quasi-principle bundles of type $(\mathfrak m, \mg^0)$.

Now let us define the algebraic prolongation of the pair $(\mathfrak m, \mg^0)$ and formulate our main result.
%a necessary terminology in order to formulate the main Tanaka result of \cite{tan} on canonical frames for structure of type %$(\mathfrak m, \mg^0)$ and our extension of this result to quasi-principal bundles of type $(\mathfrak m, \mg^0)$.
First note that  the subspace
$\mathfrak m\oplus \mg^0$ is endowed with the natural structure of a graded Lie algebra. For this we only need to define
brackets $[f,v]$
for $f\in \mg^0$ and $v\in \mathfrak m$, because $\mathfrak m$ and $\mg^0$ are already Lie algebras.
Set
\begin{equation}
\label{br1}
[f,v]:= f(v).
\end{equation}
Since $\mg^0$ is a subalgebra of the algebra of the derivations of $\mathfrak m$ preserving the grading, such operation indeed defines the structure of the graded Lie algebra on $\mathfrak m\oplus \mg^0$.
%Now recall the notion of Tanaka universal algebraic prolongation of the graded algebra $\mathfrak m\oplus\mg_0$.
Informally speaking, the Tanaka universal algebraic prolongation of the pair $(\mathfrak m, \mg^0)$
%(or of the graded algebra $\mathfrak m\oplus\mg_0$$
is the maximal (nondegenerate) graded Lie algebra, containing the graded Lie algebra
$\displaystyle{\bigoplus_{i\leq 0}\mg^i}$ as its non-positive part. More precisely, Tanaka introduces the following
%the algebra
%\begin{definition}
% The algebraic prolongation of the algebra
%$\mathfrak m\oplus\mg^0=\displaystyle{\bigoplus_{i\leq 0}\mg^i}$ is
\begin{definition}
\label{univdef}
The graded Lie algebra $\mathfrak u(\mathfrak m, \mg^0)= \displaystyle{\bigoplus_{i\in\mathbb Z}\mathfrak u^i(\mathfrak m,\mg^0)}$, satisfying the following three conditions:
\begin{enumerate}
\item $\mathfrak u^i(\mathfrak m,\mg^0)=\mg^i$ for all $i\leq 0$;
\item if $X\in \mathfrak u^i(\mathfrak m,\mg^0)$ with $i>0$ satisfies $[X, \mg^{-1}]=0$, then $X=0$;
\item $\mathfrak u(\mathfrak m, \mg^0)$ is the maximal graded Lie algebra, satisfying Properties 1 and 2.
\end{enumerate}
%\end{definition}
%This graded Lie algebra $\mathfrak u(\mathfrak m, \mg^0)$
is called  the \emph{universal algebraic prolongation}  of the graded Lie algebra $\mathfrak m\oplus \mg^0$ or of the pair
$(\mathfrak m,\mg^0)$.
\end{definition}
The algebra $\mathfrak u(\mathfrak m, \mg^0)$  can be explicitly realized as follows.
Define the \emph{$k$-th algebraic prolongation $\mg^k$ of the Lie algebra $\mathfrak m\oplus \mg^0$} by induction for any $k\geq 0$. Assume that spaces $\mg^l\subset \displaystyle{\bigoplus_{i<0}{\rm Hom}(\mg^i,\mg^{i+l})}$ are defined  for
all $0\leq l<k$. Set

\begin{equation}
\label{br2}
[f,v]=-[v, f]=f(v) \quad \forall f\in \mg^l, 0\leq l<k, \text{ and }v\in\mathfrak m.
\end{equation}
 Then let
\begin{equation}
\label{mgk}
\mg^k\stackrel{\text{def}}{=}\left\{f\in \bigoplus_{i<0}{\rm Hom}(\mg^i,\mg^{i+k}): f ([v_1,v_2])=[f (v_1),v_2]+[v_1, f( v_2)]\,\,\forall\, v_1, v_2 \in \mathfrak m\right\}.
\end{equation}
Directly from this definition and the fact that $\mathfrak m$ is fundamental (that is, it is generated by $\mg^{-1}$) it follows that if
$f\in \mg^k$ satisfies $f|_{\mg^{-1}}=0$, then $f=0$.
The space $\bigoplus_{i\in Z} \mg^i$ can be naturally endowed with the structure of a graded Lie algebra.
The brackets of two elements from $\mathfrak m$ are as in $\mathfrak m$. The brackets of an element with non-negative weight and an element from $\mathfrak m$ are already defined by  \eqref{br2}.
It only remains to define  the brackets $[f_1,f_2]$ for $f_1\in\mg^k$, $f_2\in \mg^l$ with $k,l\geq 0$.
%, because brackets of all other pair of elements are already defined.
The definition is inductive with respect to $k$ and $l$: if $k=l=0$ then the bracket $[f_1,f_2]$ is as in $\mg^0$. Assume that $[f_1,f_2]$
is defined for all $f_1\in\mg^k$, $f_2\in \mg^l$ such that a pair  $(k,l)$ belongs to  the set
\begin{equation*}
\{(k,l):0 \leq k\leq \bar k, 0\leq l\leq \bar l\}\backslash \{(\bar k,\bar l)\}.
\end{equation*}
Then define $[f_1, f_2]$  for $f_1\in\mg^{\bar k}$, $f_2\in \mg^{\bar l}$ to be the element of $\displaystyle{\bigoplus_{i<0}{\rm Hom}(\mg^i,\mg^{i+\bar k+\bar l})}$ given by
\begin{equation}
\label{posbr}
[f_1,f_2]v\stackrel{\text{def}}{=} [f_1(v), f_2]+[f_1,f_2(v)]\quad \forall v\in\mathfrak m.
\end{equation}
It is easy to see that $[f_1,f_2]\in \mg^{k+l}$ and that $\bigoplus_{i\in Z} \mg^i$ with bracket product defined as above is a graded Lie algebra. As a matter of fact (\cite{tan} \S 5) this graded Lie algebra satisfies Properties 1-3 of Definition \ref{univdef} . That is it is a realization of the universal algebraic prolongation $\mg(\mathfrak m, \mg^0)$ of the algebra $\mathfrak m\oplus\mg^0$.
In particular $\mathfrak u^i(\mathfrak m,\mg^0)\cong\mg^i$.

It turns out (\cite{tan} \S 6, \cite{yam} \S 2) that  $\mathfrak u(\mathfrak m, \mg^0)$ is closely related to the Lie algebra of infinitesimal symmetries of the so-called \emph{flat principle bundle of type $(\mathfrak m, \mg^0)$} . To define this object  let, as before, $D_{\mathfrak m}$ be the flat distribution of the constant type $\mathfrak m$, i.e. the
 left invariant distribution on $M(\mathfrak{m})$ such that $D_{\mathfrak m}(e)=\mg^{-1}$. Let $G^0$ be the simply-connect Lie group with the Lie algebra $G^0$. Denote by $L_x$ the left translation on $M(\mathfrak m)$ by an element $x$. Finally, let $P^0(\mathfrak m, \mathfrak g^0)$
be the set of all pairs $(x,\vf)$,
where  $x\in M(\mathfrak m)$ and  $\vf:\mathfrak{m}\to\mathfrak m(x)$ is an isomorphism of the graded Lie algebras $\mathfrak {m}$ and $\mathfrak m(x)$ such that
$(L_{x^{-1}})_*\vf\in G^0$. The bundle $P^0(\mathfrak m, \mathfrak g^0)$ is called \emph{the flat principal bundle of constant type $(\mathfrak m, \mg^0)$}.
If $\dim \mathfrak u(\mathfrak m, \mg^0)$ is finite (which is equivalent to the existence of $l>0$ such that $\mathfrak u^l(\mathfrak m,\mg^0)=0$), then the algebra of infinitesimal symmetries is  isomorphic to $\mathfrak u(\mathfrak m, \mg^0)$. The analogous formulation in the case  when
$\mathfrak u(\mathfrak m, \mg^0)$ is infinite dimensional may be found in \cite{tan} \S 6.
%then the formal algebra of this algebra of infinitesimal symmetries is isomorphic to the completion of $\mg(\mathfrak m, \mg^0)$ (see %\cite{tan} for detail).

One of the main results of Tanaka theory \cite{tan} can be formulated as follows:
\begin{theorem}
\label{tantheor}
Assume that the universal algebraic prolongation $\mathfrak u(\mathfrak m, \mg^0)$ of the pair $(\mathfrak m, \mg^0)$ is finite dimensional and $l\geq 0$ is the maximal integer such that $\mathfrak u^l(\mathfrak m, \mg^0)\neq 0$.
To any principle bundle $P^0$ of type $(\mathfrak m, \mg^0)$ one can assign, in a canonical way,
a sequence of bundles $\{P^i\}_{i=0}^l$ such that  $P^i$ is an affine bundle over $P^{i-1}$ with fibers being affine spaces over the linear space $\mathfrak u^i(\mathfrak m,\mg^0)$ $(\cong \mg^i)$ for any $i=1,\dots l$ and such that $P^l$ is endowed with the canonical frame.
\end{theorem}
 %In general $P^i$ is not a frame bundle. This is the case only for $\mathfrak m=\mg^{-1}$; that is, for $G$-structures. But if $\dim \mg(\mathfrak m, \mg^0)$ is finite or, equivalently, if there exists $l\geq 0$ such that $\mg^{l+1}(\mathfrak m,\mg^0)=0$,
%but $\mg^i(\mathfrak m,\mg^0)\neq 0$ for any $0\leq i\leq l$
%then the bundle $P^{l+\mu}$ is an $e$-structure over $P^{l+\mu-1}$, i.e. $P^{l+\mu-1}$ is endowed with a canonical frame (a structure of absolute parallelism).
%Note that all $P^i$ with $i\geq l$ are identified one with each other by the canonical projections (which are diffeomorphisms in that %case).
%Hence,
%already
%\emph{ $P^{l}$  is endowed with a canonical frame}.

Once a canonical frame is constructed, the equivalence problem for the principle bundles of type $(\mathfrak m, \mg^0)$ is in essence solved. Moreover,  $\dim \mathfrak u(\mathfrak m, \mg^0)$ gives the sharp upper bound for the dimension of the algebra of infinitesimal symmetries of such
structures.

The main results of the present paper is the following
\begin{theorem}
\label{maintheor}
Assume that the universal algebraic prolongation $\mathfrak u(\mathfrak m, \mg^0)$ of the pair $(\mathfrak m, \mg^0)$ is finite dimensional and $l\geq 0$ is the maximal integer such that $\mathfrak u^l(\mathfrak m, \mg^0)\neq 0$.
To any quasi-principle bundle $P^0$ of type $(\mathfrak m, \mg^0)$ one can assign, in a canonical way,
a sequence of bundles $\{P^i\}_{i=0}^l$ such that  $P^i$ is an affine bundle over $P^{i-1}$ with fibers being affine spaces over the linear space of dimension equal to $\dim \,\mathfrak u^i(\mathfrak m,\mg^0)$ $(=\dim\mg^i)$ for any $i=1,\dots l$ and such that $P^l$ is endowed with the canonical frame.
\end{theorem}
Note that in both theorems above the construction of the sequence of bundles $P^i$ depends on the choice of the so-called
%identifying spaces and
normalization conditions on each step of the prolongation procedure, while in Theorem \ref{maintheor} it also depends on the choice of the so-called identifying spaces on each step of the prolongation procedure. The precise meaning of these notions
will be given in the course of the proof of Theorem \ref{maintheor} in sections \ref{firstprolongsec} and \ref{highprolongsec}.
%and more precise formulation of this theorem will be given there as well (see Theorem \ref{maintheor1}).
This proof
%of Theorem \ref{maintheor} is given in sections \ref{firstprolongsec} and \ref{highprolongsec} below and it
is a modification of the proof of Theorem \ref{tantheor} given in the second author's lecture notes \cite{zeltan}.

As a direct consequence of Theorems \ref{quasidelta} and \ref{maintheor} we get the following

\begin{theorem}
\label{cor1}
Given a flag structure $\bigl(\Delta,\{\mathcal Y^\gamma\}_{\gamma\in\mathcal S}\bigr)$ with the constant Tanaka symbol $\mathfrak m$  (of $\Delta$) and with  the constant flag symbol $\delta$  such that  the universal algebraic prolongation $\mathfrak u\bigl(\mathfrak m, \mathfrak u^F(\delta)\bigr)$ of the pair $(\mathfrak m, \mathfrak u^F(\delta))$ is finite dimensional, one can assign to it a bundle over $\mathcal S$ of the dimension  $\dim\,\mathfrak u\bigl(\mathfrak m, \mathfrak u^F(\delta)\bigr)$ endowed with the canonical frame.
\end{theorem}

Thus, the construction of the canonical frames for a flag structure $\bigl(\Delta,\{\mathcal Y^\gamma\}_{\gamma\in\mathcal S}\bigr)$ with the constant Tanaka symbol $\mathfrak m$  (of $\Delta$) and with  the constant flag symbol $\delta$ is reduced first to the calculation of the algebra $\mathfrak u^F(\delta)$ and then to the calculation of  the algebra $\mathfrak u\bigl(\mathfrak m, \mathfrak u^F(\delta)\bigr)$. Both task are the problems of linear algebra.

The last theorem can be applied directly to the structures of subsections \ref{difeqprelim} and \ref{dprelim}, as discussed in more detail in Examples 1-3 of the next section. To apply our scheme to the structures discussed in subsection \ref{sRprelim} we need to make two straightforward modifications,
%of the general theory
described in two remarks at the end of this section.

In any case for all structures mentioned in subsections \ref{difeqprelim}-\ref{dprelim} ( and discussed in more detail in Examples 1-4 of section \ref{motsec} below) the ambient distribution of the corresponding flag structures are either $T\mathcal S$ or a contact distribution on $\mathcal S$ (i.e. the symbol $\mathfrak m$ is either the commutative algebra or the Heisenberg algebra of an appropriate dimension) and the group naturally acting on $\Delta$ is either  the General Linear group or the symplectic group or the conformal symplectic group. All possible symbols $\delta$ of curves of flags with respect to these three groups were listed in \cite[section 7]{flag2} and their universal algebraic prolongation $\mathfrak u^F(\delta)$ were calculated in \cite[section 8]{flag2}, using the representation theory of $\mathfrak{sl}_2$. In section 3 we describe the $\mathfrak u\bigl(\mathfrak m, \mathfrak u^F(\delta)\bigr)$ for several particular symbols $\delta$ or refer to our previous papers, where they were calculated. The calculations of $\mathfrak u\bigl(\mathfrak m, \mathfrak u^F(\delta)\bigr)$ for any flag symbol $\delta$ with respect to the three aforementioned groups will be given in
\cite{mix2} and \cite{jacsymb}.

\begin{remark}{\bf (the case when additional structures are given on the distribution $\Delta$)}
\label{addstrrem}
{\rm Assume that an additional structure is given on the distribution $\Delta$ such that the presence of this structure allows one to reduce the bundle $P^0(\mathfrak m)$ to a principle bundle $\widetilde P^0$ with the structure group $\widetilde G^0\subset G^0(\mathfrak m)$. For example, for the flag structures discussed briefly in subsection \ref{sRprelim} and, in more detail, in Example 4 of the next section, we have that $\Delta=T\mathcal S$ and the symplectic structure is given on $\mathcal S$. In this case $\mathfrak m=\mg^{-1}$ and $P^0(\mathfrak m)$ coincides with the bundle $\mathcal F(\mathcal S)$  of all frames on $\mathcal S$. Fixing a symplectic structure on $\mathfrak m$  we can reduce the bundle $\mathcal F(\mathcal S)$ to the bundle with the structure group $Sp(\mathfrak m)$ of the so-called symplectic frames such that the fiber over a point $\gamma\in \mathcal S$ consist of all isomorphisms from $\mathfrak m$ to $T_\gamma \mathcal S$, preserving the corresponding symplectic structures.  Returning to the general situation, let $\tilde \mg^0$ be the Lie algebra of $\widetilde G^0$. Then exactly the same theory will work if we will replace everywhere starting from Definition \ref{quasidef} the bundle $P^0(\mathfrak m)$ by $\widetilde P^0$, the Lie algebra $\mg^0(\mathfrak m)$ by $\widetilde \mg$, and the group $ G^0(\mathfrak m)$ by $\widetilde G^0$. $\Box$}
\end{remark}

\begin{remark} {\bf (the case when $\mathcal Y^\gamma$ are parameterized curves)}
\label{pararem}
 {\rm Assume that in the flag structure $\bigl(\Delta, \{\mathcal Y^\gamma\}_{\gamma \in\mathcal S}\bigr)$ each submanifolds $\mathcal Y^\gamma$ of flags is a curve with a distinguished parametrization, up to a translation, or shortly a parameterized curve. In order to formulate theorems similar to Theorems \ref{quasidelta} and \ref{cor1} we need only to modify the definition of a symbol of a curve of flags and its universal algebraic prolongation to the case of parameterized curves, as was done
%There are only two modifications here compared to the unparameterized case. These modifications are based on the modifications needed for the construction of the canonical moving frames for parameterized curves of flags compared to unparametrized curves, given
in \cite[subsection 4.5]{flag2}.   The symbol of a parameterized  curve is an orbit of an element (and not a line)  of $\gr \mg^{-1}$ from  ${\rm gr}\mg^0(\mathfrak m)$ defined up to the adjoint action of $G_+^0(\mathfrak m)$ on the space of degree $-1$ elements. Fix a representative $\delta$
%\in \mathfrak g_{-1}$
of this orbit. Then the $0$-degree algebraic prolongation $\mathfrak u_0^{F, \text {par}}(\delta)$ of the symbol $\delta$ of a parametrized curve should be defined as follows:
%  . Let $\delta\in\mathfrak g_{-1}$.
%Set $\mathfrak u_{-1}=\mathbb K\delta$ and
\begin{equation}
\label{prolong0par}
\mathfrak u_0 ^{F, \text {par}}(\delta):=\{X\in \bigl({\rm gr} \mg^{0}(\mathfrak m)\bigr)_{0}:[X,\delta]=0\}.
\end{equation}
where, as before, $\bigl({\rm gr} \mg^{0}(\mathfrak m)\bigr)_{0}$ denotes the space of all degree $0$ endomorphisms of the graded space ${\rm gr} \mg^{-1}$ (corresponding to the filtration \eqref{filtg}) belonging to ${\rm gr} \mg^{0}(\mathfrak m)$. The spaces $\mathfrak u_k ^{F, \text {par}}(\delta)$ for $k>0$ are defined recursively, using \eqref{kprolong} with $\mathfrak u_{k-1}^F(\delta)$ replaced by $\mathfrak u_{k-1}^{F, \text {par}}(\delta)$ and the universal algebraic prolongation $\mathfrak u ^{F, \text {par}}$ of the symbol $\delta$ of a parametrized curve of flag is defined as $$\mathfrak u^{F,\text{par}}(\delta)=\displaystyle{\bigoplus_{k\geq -1}\mathfrak u_k^{F,\text{par}}(\delta)}.$$
Finally, in order to formulate the results analogous to Theorems \ref{quasidelta} and \ref{cor1} for a flag structure
$\bigl(\Delta,\{\mathcal Y^\gamma\}_{\gamma\in\mathcal S}\bigr)$ with the constant Tanaka symbol $\mathfrak m$  (of $\Delta$) and with  $\mathcal Y^\gamma$ being parametrized curves with the constant symbol $\delta$, we need just to replace  $\mathfrak u^{F}(\delta)$ by $\mathfrak u^{F,\text{par}}(\delta)$ in these theorems.} $\Box$

\end{remark}

\section{Applications: flag structures via linearization}
\label{motsec}

\setcounter{equation}{0}
\setcounter{theorem}{0}
\setcounter{lemma}{0}
\setcounter{proposition}{0}
\setcounter{definition}{0}
\setcounter{cor}{0}
\setcounter{remark}{0}
\setcounter{example}{0}
In this section we describe how the linearization procedure works in general and show how the developed theory can be applied
to equivalence problems for various types of differential equation, bracket-generating distributions, sub-Riemannian and more general geometric structures mentioned in the Introduction. Before going to these examples, let us consider the following general situation.
Let $\widetilde{\mathcal S}$ be a smooth
manifold endowed with a tuple of distributions $(\widetilde \Delta, \mathcal C, \{\mathcal V_i\}_{i\in \mathbb Z})$ satisfying the following properties:

\begin{enumerate}

\item [{\bf (P1)}]

$\mathcal C$ and $\{\mathcal V_i\}_{i\in \mathbb Z}$ are subdistributions of
$\widetilde {\Delta}$ ;

\item [{\bf(P2)}]

Distribution $\mathcal C$ is integrable and satisfies
$$ [\mathcal C,\widetilde  \Delta]\subset\widetilde \Delta.$$
In other words, any section of the distribution $\mathcal C$ is an infinitesimal symmetry of the distribution $\widetilde \Delta$.

\item [{\bf (P3)}]
The distribution $\mathcal V_i$ is a subdistribution of $\mathcal V_{i-1}$ for every $i\in \mathbb Z$, $\mathcal V_i+\mathcal C=\mathcal D$ for $i\leq r$
, $\mathcal V_i=\mathcal V_s$ for $i\geq s$ for some integer $r$ and $s$, and

\begin{equation}
\label{compat}
[\mathcal C, \mathcal V_i]\subseteq \mathcal V_{i-1}+\mathcal C
\end{equation}
\item [{\bf (P4)}]
For every integers $i$  spaces
$\mathcal V_i(x)\cap \mathcal C(x)$ have the same dimensions at all points $x\in\widetilde{\mathcal S}$.
\end{enumerate}
%Here $r$ and $s$ are some integers, $s\leq r$.

We are interested in the (local) equivalence problem for such tuples of distributions $(\widetilde \Delta, \mathcal C, \{\mathcal V_i\}_{i\in\mathbb Z})$ with respect to the action of the group of diffeomorphisms of $\widetilde M$.  Note, that the treatment here is slightly more general than in the subsection \ref{df}, where the case of double fibration is considered, because we do not assume that subdistributions $V_i$ are integrable.
 %These tuples of distributions appear naturally in the geometric study of differential equations and variational problems, see  motivating examples in %subsection \ref{motsec} below.
 %Note that in practice, among all distributions $\mathcal V_i$ we will have only finitely many distributions different from a zero section of $\widetilde %{\Delta}$ or $\widetilde {\Delta}$ itself.

In many cases the tuple of distributions  $(\widetilde \Delta, \mathcal C, \{\mathcal V_i\}_{i\in \mathbb Z})$ will satisfy also the following strengthening of property (P3):
\medskip

{\bf (P3$\,'$)} There exists $i_0\in \mathbb Z$ such that
\begin{eqnarray}
&~&\forall i\leq i_0 \quad [\mathcal C, \mathcal V_i]= \mathcal V_{i-1}+\mathcal C, \label{compats1}\\
&~&\forall i>i_0\quad \mathcal V_i(x)=\left\{v\in \mathcal V_{i-1}(x): \begin{array}{c}\text {there exists a vector field } \hat v \text{ tangent to }  \mathcal V_{i-1}\\ \text {such that }
\hat v(x)=v \text{ and } [C, \hat v](x)\subset \mathcal V_{i-1}(x)\end{array}\right\}, \label{compats2}
\end{eqnarray}
\medskip

 If the property (P3$\,'$) holds, then all distributions $\mathcal V_i$ with $i\geq i_0$ and all distributions $\mathcal V_i+\mathcal C$ with $i<i_0$ are determined by the pair of distributions $(\mathcal C, \mathcal V_{i_0})$, using \eqref{compats1} and \eqref{compats2} inductively. We will say in this case that the tuple of distributions $(\widetilde \Delta, \mathcal C, \{\mathcal V_i\}_{i\in \mathbb Z})$ is \emph{generated modulo $\mathcal C$ by the triple of distributions  $(\widetilde \Delta,\mathcal C, \mathcal V_{i_0})$.}

 Finally let us discuss the following strengthening of property (P3$\,'$):
\medskip

 {\bf (P3$\,''$)} There exists $i_0\in \mathbb Z$ such that
\begin{equation}
\forall i\leq i_0 \quad [\mathcal C, \mathcal V_i]= \mathcal V_{i-1}, \label{compats3}
\end{equation}
 and also \eqref{compats2} holds.
\medskip

 If property (P3$\,''$) holds, then the whole tuple of distributions $\{\mathcal V_i\}_{i\in\mathbb Z}$
  %with $i\geq i_0$ and all distributions $\mathcal V_i+\mathcal C$ with $i<i_0$
  is determined by the pair of distributions $(\mathcal C, \mathcal V_{i_0})$ using \eqref{compats3} and \eqref{compats2} inductively.
 We  will say in this case that the tuple of distributions $(\widetilde \Delta, \mathcal C, \{\mathcal V_i\}_{i\in \mathbb Z})$ is \emph{ generated by the triple of distributions  $(\widetilde \Delta,\mathcal C, \mathcal V_{i_0})$.} In practice, one starts with a triple of distributions $(\widetilde \Delta,\mathcal C, \mathcal V)$ such that $\mathcal C$ and $\mathcal V$ are subdistributions of $\widetilde \Delta$ and $\mathcal  C$ satisfies the property (P2) above and then, setting $\mathcal V=\mathcal V_{i_0}$ for some integer $i_0$, one can generate the tuple $\{\mathcal V_i\}_{i\in\mathbb Z}$ inductively using \eqref{compats3} and \eqref{compats2}.

%one starts with the tuple of distributions $(widetilde \Delta,\mathcal C,\mathcal V)$ satisfying properties (P1) and (P2) above and then one constructs %....

%\subsection{Linearization procedure and distributions with distinguished submanifolds of flags}
 Let  $\Fol(\mathcal C)$ be a foliation of $\widetilde{\mathcal S}$ by maximal integral submanifolds of the integrable distribution $\mathcal C$.
 The main idea to treat the equivalence problem under consideration is to pass to the quotient manifold
 by the foliation $\Fol(\mathcal C)$.
 %be a foliation of $\mathcal S$ by maximal integral submanifolds of $\mathcal G$. Then we can define the \emph{linearization of the distribution $\mathcal V$ along the foliation $\Fol(\mathcal G)$} in the following way.
 Indeed, locally we can assume that there exists a
quotient manifold $$\mathcal S=\widetilde{\mathcal S}/\Fol(\mathcal C),$$
whose points are leaves of $\Fol(\mathcal C)$.
Let $\Phi\colon
\widetilde{\mathcal S}\to\mathcal S%}/\Fol(\mathcal C)
$ be the canonical projection to the quotient manifold. Then by the property (2) above the push-forward of the distribution $\widetilde{\Delta}$ by $\Phi$ defines the distribution
$$\Delta:=\Phi_*(\widetilde\Delta)$$
on the quotient manifold $\mathcal S$.
%/\Fol(\mathcal C)$.

%Now for definiteness assume that distributions $\mathcal C$, $\mathcal V$, and $\mathcal V\cap \mathcal C$ have ranks $m$, $k$ and $r$, respectively.
Fix a leaf $\gamma$ of  $\Fol(\mathcal C)$. Let

\begin{equation}
\label{Ji}
\mathcal Y_i^\gamma(x):=\Phi_*(\mathcal V _i(x))\}, \quad x\in \gamma.
\end{equation}
Then we can define the map $\mathcal Y^\gamma$ from $\gamma$ into the appropriate variety of flag of the space $\widetilde \Delta(\gamma)$ as follows:
%$\mathrm{Gr}_{k-r}(\mathcal D(\gamma)\bigr)$ of $(k-r)$-dimensional subspaces of $\mathcal D(\gamma)$ or, under additional regularity assumptions, an %$m$-dimensional submanifold of  ${\rm Gr}_{k-r}(T_\Gamma\bigl(\mathcal S/\Fol(\mathcal G)\bigr))$ as follows:

\begin{equation}
\label{Jflag}
\mathcal Y^\gamma(x)=\{\mathcal Y^\gamma_i(x))\}_{i\in \mathbb Z}, \quad x\in \gamma;\quad \forall i\in\mathbb Z\,\quad\mathcal Y^\gamma_i(x)\subseteq  \mathcal Y^\gamma_{i-1}(x)
\end{equation}
%, where $\mathrm {pr}\colon
%\mathcal S\to \mathcal S/\Fol(\mathcal G)$ is a natural projection.
The map $\mathcal Y^\gamma$ or its image in the appropriate  flag variety is called the \emph{linearization of the sequence of distributions $\{\mathcal V_i\}_{i\in Z}$ along the foliation $\Fol (\mathcal C)$ at the leaf $\gamma$}. This image  will be denote by $\mathcal Y^\gamma$ as well.

%or
%the \emph {linearization of the distribution $\mathcal V$ along the leaf $\Gamma$ (of $\Fol (\mathcal G)$}). In the cases under consideration $l=1$ so that the linearizations are  curves in projective spaces.

%\subsection{Geometry of curves of flags}
By means of the linearization procedure we obtain from the original tuple of distributions $(\widetilde \Delta, \mathcal C, \{\mathcal V_i\}_{i\in \mathbb Z})$ the distributions $\Delta$ on the manifold $\mathcal S$ with the fixed submanifolds $\mathcal Y^\gamma$ in an appropriate flag variety on each fiber $\Delta(\gamma)$ of $\Delta$, $\gamma\in \mathcal S$, i.e. the flag structure
%For shortness, we will denote such structures by
 $\bigl(\Delta, \{\mathcal Y^\gamma\}_{\gamma \in\mathcal S}\bigr)$. Note also that $\widetilde{\mathcal S}\rightarrow \mathcal S$ is exactly the tautological bundle over $\mathcal S$ in the sense of Remark \ref{nice}.
 %and we will call them \emph{flag structures}.

It is clear from our constructions  that if \emph{two tuples of distributions
%$(\widetilde \Delta, \mathcal C, \{\mathcal V_i\}_{i\in \mathbb Z})$ and $(\widetilde \Delta^1, \mathcal C^1, \{\mathcal V_i^1\}_{i\in \mathbb Z})$,
satisfying properties (P1)-(P4), are equivalent, then the corresponding flag structures obtained by the linearization procedure are equivalent.} The converse is not true in general, but it is obviously true if every distribution $\mathcal V_i\neq 0$ contains the distribution $\mathcal C$. Namely, we have

\begin{proposition}
\label{flagprop}
Assume that two tuples of distributions $(\widetilde \Delta, \mathcal C, \{\mathcal V_i\}_{i\in \mathbb Z})$ and $(\widetilde \Delta^1, \mathcal C^1, \{\mathcal V_i^1\}_{i\in \mathbb Z})$ satisfy properties (P1)-(P4) and, in addition,  $\mathcal C$ is a subdistribution of $\mathcal V_i$ and
$\mathcal C^1$ is a subdistribution of $\mathcal V_i^1$ for any $i\in\mathbb Z$. Then these two structures are equivalent if and only if the corresponding flag structures obtained by the linearization procedure are equivalent.
\end{proposition}

Note that the application of the linearization procedure to the tuple $(\widetilde \Delta, \mathcal C, \{\mathcal V_i\}_{i\in \mathbb Z})$ and to the tuple $(\widetilde \Delta, \mathcal C, \{\mathcal V_i+C\}_{i\in \mathbb Z})$ leads to the same flag structure.

\begin{definition}
\label{reduc}
We say that the tuple $(\widetilde \Delta, \mathcal C, \{\mathcal V_i\}_{i\in \mathbb Z})$, satisfying properties (P1)-(P4), is recoverable  by the linearization procedure, if each subdistribution $\mathcal V_i$ is intrinsically recovered from the tuple $(\widetilde \Delta, \mathcal C, \{\mathcal V_i+\mathcal C\}_{i\in \mathbb Z})$.
\end{definition}
One example of such intrinsic recovery is when  $\mathcal V_i$ is a unique integrable subdistribution of maximal rank of $\mathcal V_i+\mathcal C$. Another example is when $\mathcal V_i$ is the Cauchy characteristic subdistribution of $\mathcal V_{i-1}+\mathcal C$, i.e. the maximal subdistribution of $\mathcal V_{i-1}+\mathcal C$ such that $[\mathcal C,\mathcal V_i]\subseteq \mathcal V_{i-1}+\mathcal C$. As an immediate consequence of Proposition \ref{flagprop} we get the following

\begin{cor}
Two tuples of distributions
 %$(\widetilde \Delta, \mathcal C, \{\mathcal V_i\}_{i\in \mathbb Z})$ and $(\widetilde \Delta^1, \mathcal C^1, \{\mathcal V_i^1\}_{i\in \mathbb Z}$
 satisfying properties (P1)-(P4) and recoverable by the linearization procedure are equivalent if and only if the corresponding flag structures obtained by the linearization procedure are equivalent.
\end{cor}

In most of the applications $\mathcal C$ will be a distribution of rank 1 (a line distribution). In this case, after the linearization procedure, one gets the distinguished curves of flags on each fiber of $\Delta$. From relation \eqref{compat} of property (P2) it follows that the submanifolds of flags $\mathcal Y^\gamma$ are not arbitrary but they are the \emph {compatible with respect to differentiation} as described in subsection \ref{quasiflag}
following the terminology of \cite[Section 3]{flag2}.

\begin{remark}
\label{refinrem}
{\rm Finally note that if our original structure $(\widetilde \Delta, \mathcal C, \{\mathcal V_i\}_{i\in \mathbb Z})$ satisfies the property (P3$\,'$) for some integer $i_0$,
then the  submanifold of flags $\mathcal Y^\gamma=\{\mathcal Y_i^\gamma(x)\}_{i\in \mathbb Z}$ is completely determined by the submanifold
$\mathcal Y_{i_0}^\gamma$ in the corresponding Grassmannian. Moreover,  $\mathcal Y^\gamma=\{\mathcal Y_i^\gamma(x)\}_{i\in \mathbb Z}$ is the so-called maximal refinement of $\mathcal Y_{i_0}^\gamma$ in the sense of \cite[section 6]{flag2}.}$\Box$
\end{remark}
%\subsection{Motivationg examples} Below we give several motivating examples:

%The main idea of using the linearization procedure in the equivalence problem for the structures given by the pair of distribution $(\mathcal %V,\mathcal G)$ on $\mathcal S$ with respect to the action of group of diffeomorphisms of $M$ is that it allows to construct the invariants of such %structures from invariants of submanifold in an appropriate flag vatiety with respect to the natural action of the General Linear Group on this %Grassmannian.

As a direct consequence of Theorem \ref{cor1}
%Finally returning to our original tuples of distributions $(\widetilde \Delta, \mathcal C, \{\mathcal V_i\}_{i\in \mathbb Z})$
we have the following

\begin{theorem}
\label{main2}
Assume that tuples of distributions $(\widetilde \Delta, \mathcal C, \{\mathcal V_i\}_{i\in \mathbb Z})$ satisfy conditions (P1)-(P4)
%it is recoverable by the linearization procedure
and the resulting
%(after linearization)
flag structures $\bigl(\Delta,\{\mathcal Y^\gamma\}_{\gamma\in\mathcal S}\bigr)$ satisfy conditions (F1)-(F4) with the constant Tanaka symbol $\mathfrak m$  (of $\Delta$) and with  the constant flag symbol $\delta$. Assume also that the universal algebraic prolongation $\mathfrak u\bigl(\mathfrak m, \mathfrak u^F(\delta)\bigr)$ of the pair $(\mathfrak m, \mathfrak u^F(\delta))$ is finite dimensional. Then to any such  tuple $(\widetilde \Delta, \mathcal C, \{\mathcal V_i\}_{i\in \mathbb Z})$ one can assign a bundle over $\widetilde {\mathcal S}$ of the dimension  $\dim\,\mathfrak u\bigl(\mathfrak m, \mathfrak u^F(\delta)\bigr)$ endowed with the canonical frame. If, in addition, the tuples of distributions $(\widetilde \Delta, \mathcal C, \{\mathcal V_i\}_{i\in \mathbb Z})$ are recoverable by the linearization procedure, two of them  are equivalent if and only if the canonical frames assigned to them are equivalent.
\end{theorem}

%\textcolor{blue} {Will be explained in more detail}
%{Boris, as a matter of fact, the possibility to lift the frame from the bundle over $M$ from the previous Corollary to the bundle over $\widetilde M$ %should be justified. As I understand it is related to Proposition 4 in your text on the symmetries of trivial systems ODEs but it considers certainly %more general situation than in
%this Proposition.}

This theorem is of a quite general nature and, as the examples below show, it treats in a unified way many previously known  results  and  generalizes them significantly. Using this theorem one can obtain the main results of \cite{chern50} and \cite{doub2} on contact geometry of scalar ODEs of order $3$ and of order greater than $3$ ,respectively, (see Example 1 below) and generalize them to various  equivalence problems for systems of differential equations, including the systems with different highest order of derivatives for different unknown functions, called the systems of ODEs of mixed order (see Example 2 below). Using Theorem  \ref {main2} one can obtain the main result of \cite{doubzel2} on canonical frames for rank 2 distributions on manifolds of arbitrary dimension. Our approach here  also gives much more conceptual  point of view on the construction of the preprint \cite{doubzel3} on rank $3$ distributions, or, more precisely, the part of these results addressing the existence of the canonical frames. It generalized the results of \cite{doubzel2, doubzel3} to distributions of arbitrary rank (see Example 3). The modification of this Theorem, corresponding to the flag structure with $\mathfrak F^\gamma$ being parameterized curves of flags gives the unified construction of canonical frames for sub-Riemannian, sub-Finslerian, and more general additional structures on distributions.

In the context of Theorem \ref{main2} the following natural question arises: For what pairs $(\delta, \mathfrak m)$ is the canonical frame of Theorem \ref{main2} can be chosen as the Cartan connection over $\widetilde {\mathcal S}$? The answer
to this question, even in the case of the commutative or Heisenberg algebra $\mathfrak m$,
%and nice $\delta$,
is known yet for several particular cases only.
%\begin{renark}
%\label{nonrec}
%{\rm Even if the tuple of distributions $(\widetilde \Delta, \mathcal C, \{\mathcal V_i\}_{i\in \mathbb Z})$ is not recoverable by linearization
%\textcolor{blue}{More details on these cases will be added.}

%More detailed analysis of the algebras  $\mathfrak u\bigl(\mathfrak m, \mathfrak u^F(\delta)\bigr)$ for all possible flag symbols $\delta$ in the cases when $\Delta=\mathcal S$ , ${\rm rank} C=1$ (corresponding to system of equations of mixed order) will be given in \cite{mix1, mix2} and in the case when $\Delta$ is contact,  ${\rm rank} C=1$
%(corresponding to the symplectification procedure for the equivalence of vector distributions) will be given in \cite{jacsymb}.

%\subsection{Motivating examples}
%\label{motsec}
%Now we give several examples.
% in which the developed theory can be applied. In particular, example 3 below was the origin for all this developments.
\medskip

{\bf Example 1. Scalar ordinary differential equations up to a contact transformation.}
Consider scalar ordinary differential
equations of the form
% solved with respect
%to the highest derivative:
\begin{equation}
\label{ODE}
y^{(n)} = f(x,y,y',\dots,y^{(n-1)}).
\end{equation}
Each such equation can be considered as a hypersurface $\E$ in the jet space
$\JT{n}$, called the  \emph{equation manifold}.
%We assume that our
Equations \eqref{ODE} are  solved with respect
to the highest derivative, so the restriction of the natural projection
$\pi_0\colon \JT{n}\to\JT{n-1}$ to the hypersurface $\E$ is a diffeomorphism.
Two such equations $\E$ and $\E'$ are said to be \emph{contact equivalent}, if
there exists a contact transformation $\Phi\colon \JT{1}\to\JT{1}$ with a
prolongation $\Phi^{n}$ to $\JT{n}$ such that $\Phi^{n}(\E)=\E'$.

Note that by Lie-Backlund theorem any diffeomorphism of $\JT{n}$ preserving the Cartan distribution (i.e. the rank 2 distribution defined by the contact forms $\theta_i=dy_i -y_{i+1}dx,\quad i=0,\dots,n-1$, in the standard coordinates $(x,y_0,\ldots, y_{n})$) is the prolongation to $\JT{n}$ of some contact transformation of $\JT{1}$. Therefore two equations
 $\E$ and $\E'$  are contact equivalent if and only if there is a diffeomorphism $\Psi$ of $\JT{n}$
which preserves the Cartan distribution  on
$\JT{n}$ and such that $\Psi(\E)=\E'$.

The equivalence problem for ordinary differential equations up to contact transformations is a classical subject going back to the works of
Lie, Tresse, Elie Cartan~\cite{cartan24}, Chern~\cite{chern50},
%M.~Fels~\cite{fels}
and others. The existence of the normal Cartan connection (a special type of frames) associated
with any system of ODEs was proved in \cite{chern50} for the equations of third order and in ~\cite{doub2} for $n>3$, The latter is based on the Tanaka-Morimoto theory of normal Cartan connection
for filtered structures on manifold ~\cite{tan,tan2,mori}.

The reason we start with this example is because
the problem of contact equivalence of equations \eqref{ODE} can be reformulated  as the equivalence problem for a very particular tuples of distributions satisfying properties (P1)-(P4) and the linearization procedure in this case gives a shorter way to obtained important contact invariant of ODE's, the so-called generalized Wilczynski invariants \cite{doub3}, without constructing the whole normal Cartan connection  mentioned above. It also gives an alternative way for the construction of the canonical frame in this problem.

The Cartan distribution on
$\JT{n}$
defines the line distribution $\mathcal X$ on $\E$. This distribution is obtained by the intersection of the Cartan distribution with the tangent space to $\E$ at every point of $\E$. Note that the corresponding foliation $\Fol(S)$ is nothing but the foliation of the prolongations of the solutions of our ODE to $\JT{n}$.
%If the hypersurface $\E$ has the form \eqref{ODE} in coordinates $(x,y_0,\ldots, y_{N+1})$ on $\JT{N+1}$, then
In the coordinates $(x,y_0,\ldots, y_{n-1})$ on~$\E$:
\begin{equation}
\label{Scoord}
\mathcal X=\left\langle \frac{\partial}{\partial x} + \sum_{i=0}^{n-2}y_{i+1} {\partial\over\partial y_i}+F(x,y_0,y_1,\dots,y_{n-1}){\partial\over\partial y_{n-1}}\right\rangle.
\end{equation}
%The hypersurface $\E$ will play the role of the ambient manifold $\mathcal S$. The whole tangent bundle $T\E$ will play the role of the distribution %$\mathcal D$.

Further, let $\pi_i\colon\JT{n}\mapsto\JT{n-1-i}$ be the canonical projection, $0\leq i\leq n-1$. For any $\varepsilon\in\E$ we can define the filtration $\{V_i(\varepsilon)\}_{i=0}^{n-1}$ of $T_\varepsilon \E$ as follows:
\begin{equation}
\label{Vicoord}
V_i(\varepsilon)=\ker d_\varepsilon\pi_i \cap T_\varepsilon \E.
\end{equation}
Then $V^i$ is a rank $i$ distribution on $\E$. In the coordinates $(x,y_0,\ldots, y_{n-1})$ on $\E$ we have
\begin{equation}
\label{Vicoord1}
V_i=\left\langle \frac{\partial}{\partial y_{n-i}},\dots, \frac{\partial}{\partial y_{n-1}} \right\rangle.
\end{equation}

The tuple of distributions $(\widetilde \Delta, \mathcal C, \{\mathcal V_i\}_{i\in \mathbb Z})$, where $\widetilde \Delta=T\mathcal E$, $\mathcal C=\mathcal X$, $\mathcal V_i=0$  for $i\geq 0$,
$\mathcal V_i=V_{-i}$ for $-n+2\leq i\leq -1$, $\mathcal V_{-n+1}=
V_{n-1}+\mathcal X$, and $\mathcal V_i=
T\mathcal E$ for $i<-n+1$,
%\begin{equation}
%\label{ViODE}
%\mathcal V_i=\begin{cases}0 & i\geq 0\\
%V_{-i} & -n+2\leq i\leq -1\\
%V_{-n+1}+\mathcal X, & i=-n+1\\
%T\mathcal E, & i<-n+1
%\end{cases}.
%\end{equation}
satisfies properties (P1)-(P4) above.
We say that this tuple of distributions is \emph{associated with the ordinary differential equation \eqref{ODE} with respect to the contact equivalence}
Below we justify this terminology.
Note that  relations \eqref{Scoord} and \eqref{Vicoord1} imply that it satisfies the property (P3$\,'$) for any $i_0$ satisfying
$-n+2\leq i\leq -1$ and even the property (P3$\,''$) for $i_0=n-2$.
The latter implies that two ordinary differential equations are contact equivalent if and only if the tuples of distributions associated with these equations are equivalent.

Moreover, the constructed tuple $(\widetilde \Delta, \mathcal C, \{\mathcal V_i\}_{i\in \mathbb Z})$ is recoverable by the linearization procedure. For this we only need to check that the distributions $\mathcal V_i$ with $-n+2\leq i\leq -1$ can be intrinsically recovered from the tuple $(\widetilde \Delta, \mathcal C, \{\mathcal V_i\oplus \mathcal C\}_{i\in \mathbb Z})$. The latter follows from the fact that  $\mathcal V_i$
is the Cauchy characteristic subdistribution of $\mathcal V_{i-1}+\mathcal C$, which is a part of the proof of the Lie-Backlund theorem in this case.

Let $\mathrm{Sol}$ denote the quotient manifold $\mathcal \E/\Fol(\mathcal X)$, i.e. the manifold of the prolonged solutions of the equation manifold $\E$.
Consider the linearization $\mathcal Y^\gamma=\{\mathcal Y^\gamma_i\}_{i\in Z}$ of the sequence of distributions $\{\mathcal V_i\}$ along the foliation $\Fol (\mathcal X)$ at a leaf $\gamma$ (i.e. along a prolongation of equation \eqref{ODE}).
%Let $\mathrm{Sol}$ denote the quotient manifold $\mathcal \E/\Fol(\mathcal S)$, i.e. the manifold of solutions of the equation $\E$. Fix a point $\Gamma %\in \rm {Sol}$. In other words, $\Gamma$ is a leaf of $\Fol(\mathcal S)$ or a solution of the equation $\E$. Consider the linearization $\mathrm %{Lin}^i_
%\Gamma$ of the distribution $V^i$ along $\Gamma$.
%It is a curve in $\mathrm {Gr}_i\bigl(T_\Gamma \mathrm {Sol}\bigr)$. In particular, $\mathrm {Lin}^1_\Gamma$
%is a curve in the projective space $\mathbb P(T_\Gamma\rm{Sol})$.
By  our constructions, $\dim\, \mathcal Y_i^\gamma(\varepsilon)=-i$ if $-n+1\leq i\leq -1$, $\dim\, \mathcal Y_i^\gamma(\varepsilon)=0$ if $i\geq 0$, and $\dim\, \mathcal Y_i^\gamma(\varepsilon)=n-1$ if $i<-n+1$.
In particular, $\mathcal Y_\gamma^{-1}$ is a curve in the projective space $\mathbb PT_\gamma \mathrm{Sol}$.

Moreover, relations \eqref{Scoord},\eqref{Vicoord1},  and Remark \ref{refinrem} imply again that
%\begin{lemma}
%\label{osclem}
for any $\varepsilon\in \gamma$   the space $\mathcal Y^\gamma_{i}(\varepsilon)$ of $T_\gamma\rm{Sol}$ is exactly the $i$-th osculating space of the curve $\mathcal Y_{-1}^\gamma$ at
$\varepsilon$.
%\end{lemma}
In other words, the whole curve of flags $\mathcal Y^\gamma$ is completely determined by the curve $\mathcal Y^\gamma_{-1}$ in a projective space.
Note also that a curve in a projective space corresponds to a linear differential equation, up to transformations of independent and dependent variables  preserving the linearity. For the curve $\mathcal Y^\gamma_{-1}$ this linear equation is exactly the classical linearization (or the equation in variations) of the original differential equation \eqref{ODE} along the (prolonged) solution $\gamma$. This is why we use the word linearization here.

The symbol $\delta^1_n$ of the curve of complete flags $\mathcal Y^\gamma$ (with respect to $GL(T_\gamma{\rm Sol})$) at any point is a line of degree $-1$ endomorphisms of the corresponding graded spaces, generated by an endomorphism which has the matrix equal to a Jordan nilpotent block in some basis.
The  flat curve with this symbol is the curve of osculating subspaces of a rational normal curve in the projective space $\mathbb P T_\gamma{\rm Sol}$. Recall that a rational normal curve in $\mathbb P W$ is a curve represented as $t\mapsto [1:t:t^2:\ldots t^n]$ in some homogeneous coordinates.
It is well-known that the algebra of infinitesimal symmetries of a rational normal curve and, therefore, the universal algebraic prolongation $\mathfrak u^F(\delta^1_n)$  is equal to the image of the irreducible
embedding of $\gl_2(\mathbb R)$ into $\gl(T_\gamma{\rm Sol})$ (see also \cite[Theorem 8.3]{flag2}, where the universal algebraic prolongation was computed for all symbols of curves of flags with respect to the General Linear group).

Further, recall that in this case the symbol of the distribution $\Delta$ is isomorphic to the $n$-dimensional commutative Lie algebra, i.e. to $\mathbb R^n$. By direct computations,
the Tanaka universal algebraic prolongation $\mathfrak u\bigl(\mathbb R^n, \mathfrak u^F(\delta^1_n)\bigr)$ of the pair $(\mathbb R^n, \mathfrak u^F(\delta^1_n))$ is isomorphic to $\mathfrak{sp}_4(\mathbb R)$, if $n=3$, and to the semidirect sum of $\gl_2(\mathbb R)\leftthreetimes\mathbb R^n$ of  $\gl_2(\mathbb R)$ and $\mathbb R^n$, if $n>3$ (were $\gl_2(\mathbb R)$ acts irreducibly on $\mathbb R^n$), as expected from \cite{chern50} in the first case and from \cite{doub2} in the second case. Note also that in the second case already the first algebraic prolongation $\mathfrak u^1(\mathbb R^n, \mathfrak u^F(\delta^1_n))$ is equal to zero, i.e. the canonical bundle lives already on the quasi-principle bundle $P$ assigned to the flag structure $\bigl(\Delta, \{\mathcal Y^\gamma\}_{\gamma \in\mathcal S}\bigr)$ by Theorem \ref{quasidelta}. Note also that the local contact geometry of second-order ODE's is trivial,
since any two such equations are locally contact equivalent
%under the pseudogroup of all
%contact transformations
and the corresponding algebra $\mathfrak u^1(\mathbb R^2, \mathfrak u^F(\delta^1_2))$ is infinite dimensional  with the $k$th algebraic prolongation equal to the space if $\mathbb R^2$-values homogeneous polynomials of degree $k+1$.

%Finally, from %subsection \ref{algproglv}
%Theorem \ref{prolongN} below (see also the sentence after it) it follows that the universal algebraic prolongation of this symbol is isomorphic to $\gl_2$ (or $\sll_2$ in the case $G=SL(W)$), which also follows from the well-known fact that the algebra of infinitesimal symmetries of a rational normal curve is isomorphic to $\gl_2$ (or $\sll_2$ in the case $G=SL(W)$). In this case
Further, it follows from \cite[subsection 4.1]{flag1} that the symbol $\delta^1_n$ is nice in the sense of Remark \ref{nice}.
%in the considered case  one can choose a nice normalization condition on the level of curves of flags so that the quasi-principal bundle $P$, corresponding to this normalization condition, is a principal bundle.
%One can show that the normalization condition can be chosen to be  invariant with respect to the natural action of the subgroup $U_+(\mathfrak m)$ (which is isomorphic in this case to the group the upper-triangular matrices) on $C^1_+$ and
In this way the classical fundamental system of invariants of curves in projective spaces, the Wilczynski invariants, can be constructed (see \cite{flag2} or \cite{doub3} for the detail) for each curve $\mathcal Y^\gamma$ and they produce
%The fundamental invariants of curves in projective spaces were described by Wilczynski~\cite{wil}.
%%The algebra of all invariants admits a basis of so-called fundamental invariants $W_3, \dots, W_{N+1}$ of order $3,\dots,N+1$ respectively.
%The Wilczynski invariants of the curve $\mathcal Y_{-1}^\gamma$, taken for every prolonged solution $\gamma$, define
the contact invariants of the original ODE, called the \emph{generalized Wilczinski invariants} (\cite{doub3}).

Finally, it can be shown that the normalization condition for Theorem \ref{cor1} can be chosen such that the resulting canonical frame  will be a Cartan connection over the equation manifold $\E$ modeled by the corresponding  $Sp_4(\mathbb R)$ parabolic geometry if $n=3$ and  by a homogeneous space corresponding to a pair of Lie algebras  $\gl_2(\mathbb R)\leftthreetimes\mathbb R^n$ and $\mathfrak t_2(\mathbb R)$ if $n>3$, where $\mathfrak t_2(\mathbb R)$ is the algebra of the upper triangle $2\times 2$-matrices.

%%$\E$ under the group of contact transformations (see~\cite{dou}). We call these invariants the \emph{generalized Wilczinski invariants of $\E$} and %denote them also by $W_i$, $i=3,\dots,N+1$.
%%\textcolor{blue}{Will be completed}

\medskip

{\bf Example 2. Natural equivalence problems for systems of two differential equations of mixed order.}
%\textcolor{blue}{Will be added}
Given two natural numbers $k$ and $l$ and an integer $s$ satisfying $0\leq s\leq l$ consider the following system of two differential equations:

\begin{equation}
\label{ODE2}
\left\{\begin{array}{l}y^{(k)}(x) = f_1\bigl(x,y(x),y'(x),\dots,y^{(k-1)}(x),z(x),z'(x),\ldots,z^{(\min\{l-1,l-s\})}(x)\bigr)\\
z^{(l)}(x) = f_2\bigl(x,y(x),y'(x),\dots,y^{(k-1)}(x),z(x),z'(x),\ldots,z^{(l-1)}(x)\bigr)
\end{array}\right.,
%\quad k\leq l
\end{equation}

Such system will be called a \emph{system of differential equations of mixed order $(k,l)$ with shift $s$}. Obviously, systems having a shift $s$ have a shift $s-1$ as well.  System \eqref{ODE2} defines a submanifold $\mathcal E_0$ in the mixed jet space $J^{(k,l)}(\mathbb R, \mathbb R^2)$: in the standard coordinates
$(x, y_0, y_1,\ldots, y_k, z_0, z_1,\ldots, z_l)$ in $J^{(k,l)}(\mathbb R, \mathbb R^2)$ the submanifold $\mathcal E_0$ has a form:
\begin{equation}
\label{ODE2jet}
\left\{\begin{array}{l}y_k = f_1(x,y_0,y_1,\dots,y_{k-1},z,z_1,\ldots,z_{\min\{l-1,l-s\}})\\
z_l = f_2(x,y_0,y_1,\dots,y_{k-1},z,z_1,\ldots,z_{l-1})
\end{array}\right.
\end{equation}

Differentiate both parts of the first equation of the system \eqref{ODE2} with respect to the independent variable $x$ and replace $z^{(l)}$, if it appears  in the obtained equation, by the right hand side of the second equation of \eqref{ODE2}. With this new equation, written in the first place,  we obtain a new system of three equations. We call this system of equation the \emph{first prolongation of the system  \eqref{ODE2} with respect to the first equation}. This new system  defines a submanifold $\mathcal E_1$ in the mixed jet space $J^{(k+1,l)}(\mathbb R, \mathbb R^2)$: to obtain this submanifold in the standard coordinates
$(x, y_0, y_1,\ldots, y_{k+1}, z_0, z_1,\ldots, z_l)$ in $J^{(k+1,l)}(\mathbb R, \mathbb R^2)$ we  replace $y^{(i)}$ by $y_i$ and $z^{(j)}$ by $z_j$ in the new system obtain. The submanifold $\mathcal E_1$  will be called  the \emph{first prolongation of the submanifold $\mathcal E_0$ with respect to the first equation of the system \eqref{ODE2}}.

Similarly,  differentiating the first equation of the new system once more  and replacing $z^{(l)}$, if it appears in the obtained equation, by the right hand side of the second equation of \eqref{ODE2}, we obtain, together with already existing equation, the system of $4$ equations, called  the \emph{second prolongation of the system  \eqref{ODE2} with respect to the first equation} and the submanifold $\mathcal E_2$ in the mixed jet space $J^{(k+2,l)}(\mathbb R, \mathbb R^2)$, called the \emph{second prolongation of the submanifold $\mathcal E_0$ with respect to the first equation of the system \eqref{ODE2}}. In the same way, by induction for any $s\geq 0$ we construct the system of $s+2$ equations and the submanifold $\mathcal E_s$ in the mixed jet space $J^{(k+s,l)}(\mathbb R, \mathbb R^2)$, which will be called the \emph{prolongation of order $s$ of the system \eqref{ODE2} and of the submanifold $\mathcal E_0$ with respect to the first equation of the system \eqref{ODE2}}.

Further, the Cartan distribution is defined on any mixed jet space $J^{(\bar k,\bar l)}(\mathbb R, \mathbb R^2)$. In standard coordinates
$(x, y_0,\ldots y_{\bar k},z_0,\ldots, z_{\bar l})$ in $J^{(\bar k,\bar l)}(\mathbb R, \mathbb R^2)$ the Cartan distribution is defined by the following system of Pfaffian equations:
\begin{align*}
& dy_i - y_{i+1}dx, i= 0,\dots,\bar k-1;\\
& dz_j - z_{j+1}dx, j= 0,\dots,\bar l-1.
\end{align*}

\begin{definition}
Consider two systems of differential equations of mixed order $(k,l)$ with shift $s$. Assume that $\mathcal E_0$ and $\widetilde{\mathcal E}_0$ are the corresponding  submanifolds in the mixed jet space $J^{(k,l)}(\mathbb R, \mathbb R^2)$ and $\mathcal E_s$ and $\widetilde{\mathcal E}_s$ are the prolongations of order $s$ of the submanifold $\mathcal E_0$  and $\widetilde{\mathcal E}_0$ with respect to the first equation of the corresponding systems. We say that our systems (or submanifolds $\mathcal E_0$ and $\widetilde{\mathcal E}_0$) are $s$-equivalent (with respect to the first equation) if
\begin{enumerate}
\item
in the case $0\leq s\leq l-k$ there exists a diffeomorphism $\Phi$ of the mixed jet space $J^{(0, l-k-s)}(\mathbb R, \mathbb R^2)$ preserving the Cartan distribution on  it such that the prolongation of $\Phi$ to the mixed jet space $J^{(k+s, l)}(\mathbb R, \mathbb R^2)$ sends $\mathcal E_s$ onto $\widetilde{\mathcal E}_s$.

\item
in the case $l-k<s\leq l$ there exists a diffeomorphism $\Phi$ of the mixed jet space $J^{(k+s-l, 0)}(\mathbb R, \mathbb R^2)$ preserving the Cartan distribution on it
%$J^{(k+s-l,0)}(\mathbb R, \mathbb R^2)$
such that the prolongation of $\Phi$ to the mixed jet space $J^{(k+s, l)}(\mathbb R, \mathbb R^2)$ sends $\mathcal E_s$ onto $\widetilde{\mathcal E}_s$.
\end{enumerate}
\end{definition}

Note that given $s_1$ and $s_2$ with $0\leq s_1<s_2\leq l$  the $s_1$-equivalence and the $s_2$-equivalence are in general two different equivalence problems on the set of systems of mixed type $(k,l)$ with shift $s_2$ (which is the common set of objects for which both equivalence problems are defined). Symmetries of the trivial system of the mixed order $(k,l)$, namely of the system $y^{(k)}=0,\, z^{(l)}=0$, with respect to the $s$-equivalence  are calculated in the recent preprint \cite{mix1}.
In particular, it was shown that in the case $(k,l)=(2,3)$ the groups of symmetries for $s=0$ and $s=1$ are both $15$-dimensional but not isomorphic one to another.
By symmetries with respect to the $s$-equivalence of a system of mixed order $(k,l)$ corresponding to a submanifold $\mathcal E_0$ of  the mixed jet space $J^{(k,l)}(\mathbb R, \mathbb R^2)$ we mean any diffeomorphism $\Phi$ of either the mixed jet space $J^{(0, l-k-s)}(\mathbb R, \mathbb R^2)$ or the mixed jet space $J^{(k+s-l, 0)}(\mathbb R, \mathbb R^2)$ (depending on the sign of $l-k-s$) such that the prolongation of $\Phi$ to the mixed jet space $J^{(k+s, l)}(\mathbb R, \mathbb R^2)$ preserves the prolongations  $\mathcal E_s$ of order $s$ of $\mathcal E_0$  with respect to the first equation of the  systems. Note also that if $s=l-k\geq 0$, the $s$-equivalence is nothing but the equivalence of the $s$-prolongation $\mathcal E_s$ of $\mathcal E_0$ with respect to the first equation of the system up to point transformations.

Now let us show that the introduced $s$-equivalence is a particular case of the equivalence problems introduced in the beginning of this section.
First, system \eqref{ODE2} defines  a rank 1 distribution $\mathcal X_{k,l,s}$ on $\mathcal E_s$: its leaves are prolongations of the solutions $\bigl(y(x), z(x)\bigr)$ to the mixed jet space $J^{(k+s,l)}(\mathbb R, \mathbb R^2)$. It can be equivalently defined
%via the Cartan distribution on $J^{k+s,l}(\mathbb R, \mathbb R^2)$, which is defined by the following system of Pfaffian equations:
as the rank 1 distribution obtained by the intersection of the Cartan distribution on $J^{(k+s,l)}(\mathbb R, \mathbb R^2)$ with the tangent space to $\E_s$ at every point of $\E_s$.

Let $\pi_{s,i}\colon J^{(k+s,l)}(\R,\R^2)\to J^{(k+s-1-i,l-1-i)}(\R,\R^2)$ be the canonical projections , and let
\begin{equation}
\label{Vs}
V_{s,i}(\varepsilon)=\ker d_\varepsilon\pi_{s,i}\cap T_\varepsilon \E_s.
\end{equation}
Here $0\leq i\leq \min\{k+s-1,l-1\}$.
In the coordinates $(x,y_0,\ldots, y_{k-1},z_0,\ldots, z_{l-1})$ on $\E_s$ we have $V_{s,0}=0$ and

\begin{equation}
\label{Vscoord}
V_{s,i}=\begin{cases}\left\langle \bigl\{\frac{\partial}{\partial z_j}\bigr\}_{j=l-i}^{l-1} \right\rangle & 1\leq i\leq s\\
\left\langle \bigl\{\frac{\partial}{\partial y_j}\bigr\}_{j=k-i+s}^{k-1}, \bigl\{\frac{\partial}{\partial z_j}\bigr\}_{j=l-i}^{l-1} \right\rangle &
s<i\leq \min\{k+s-1,l-1\}
\end{cases}.
\end{equation}
We also assume that $V_{s,i}=0$ for $i\leq 0$, while for $i> \min\{k+s-1,l-1\}$ we define $V_{s,i}$ inductively by $V_{s,i}=[\mathcal X_{k,l,s}, V_{s,i-1}]$.
Then the tuple of distributions $(\widetilde \Delta, \mathcal C, \{\mathcal V_i\}_{i\in \mathbb Z})$ with $\widetilde \Delta=T\mathcal E_s$, $\mathcal C=\mathcal X_{k,l,s}$, and $\mathcal V_i=V_{s, -i}$ satisfies properties (P1)-(P4) above.
We say that this tuple of distributions is \emph{associated with the system of  differential equation of mixed order \eqref{ODE} with respect to the $s$-equivalence}.
Using \eqref{Vscoord}, it is not hard to show  that  it satisfies the property (P3$\,'$) for any $i_0$ satisfying
$s< i_0\leq \min\{k+s-1,l-1\}$ and even the property (P3$\,''$) for $i_0=\min\{k+s-1,l-1\}$.
%(\textcolor{red}{should be checked yet!}).
The latter implies that two systems of mixed order  $(k,l)$ with shift $s$ are $s$-equivalent if and only if the tuples of distributions associated with these equations are equivalent.

%In the
%. The pull-back of the contact distribution from $J^{k-1,l-1}(\R,\R^2)$ to $\E$ is a 3-dimensional vector distribution $D$ on $\E$ defined by 1-forms:
%\begin{align*}
%& dy_i - y_{i+1}dx, i= 0,\dots,k-2;\\
%& dz_j - z_{j+1}dx, j= 0,\dots,l-2.
%\end{align*}
%or, in terms of vector fields by:
%\begin{align*}
%X &= \dd{x} + y_1\dd{y} + \dots + y_{k-1}\dd{y_{k-2}}+ z_1\dd{z} +
%\dots + z_{l-1}\dd{y_{l-2}},\\
%Y &= \dd{y_{k-1}},\\
%Z &= \dd{z_{l-1}}.
%\end{align*}

%Now let us ``prolong'' the system \eqref{ODE2} by adding to it one more equation obtained by differention of both part of the first equation w.r.t. to independent variable $x$  and replacing any appearance of $y^k(x

%\textcolor{blue}{I would like to discuss what are the symbols of $\mathcal Y^\gamma$ in this case and to refer the paper \cite{mix1} for the prolongations}

From \eqref{Vscoord} it follows easily that the curves of flags $\mathcal Y^\gamma$, which are the linearizations of the sequence of distributions $\{\mathcal V_i\}_{i\in \mathbb Z}$ along the foliation ${\rm Fol}(\mathcal X_{k,l,s})$ along leaves $\gamma$, have the same symbol at any point.
%Denote by $\delta^2_{k,s,l}$.
Let us describe this symbol. For this let us introduce some notation, following \cite[subsection 7.1]{flag2}.
Let $\delta_1$ and $\delta_2$ be degree $-1$ endomorphisms of the graded spaces $V_1$ and $V_2$, respectively. The direct sum $V_1\oplus V_2$ is equipped with  the natural grading such that its $i$th component is the direct sum of $i$th components of $V_1$ and $V_2$.
The direct sum $\delta_1\oplus\delta_2$ is the degree $-1$ endomorphism of $V_1\oplus V_2$ such that the restriction of it to $V_i$ is equal to $\delta_i$ for each $i=1,2$.
A degree $-1$ endomorphism $\delta$ of a graded space $V$ is called \emph{indecomposable} if
it cannot be represented as a direct sum of two degree $-1$ endomorphisms acting on nonzero graded spaces.
%there is no splitting of $V$ into two graded spaces with the gradings inherited from the grading of $V$ such that the restrictions of $\delta$ to both this spaces are of degree $-1$ (with respect to those grading). Now we give natural models of indecomposable degree $-1$ endomorphisms.
Further, given two integers $r\leq p<0$ let $V_{rp}=\displaystyle{\bigoplus_{i=r}^pE_i}$, where $\dim E_i=1$ for every $i$, $r\leq i\leq p$, and let $\delta_{rp}$ be the degree $-1$ endomorphism of $V_{rp}$ which sends $E_i$ onto $E_{i-1}$ for every $i$, $r<i\leq p$, and sends $E_r$ to $0$.
For example in this notation the symbol $\delta^1_n$ appearing in the previous example is equal to $\mathbb R \delta_{-n, -1}$. As was shown in \cite[subsection 7.1]{flag2} the endomorphisms $\delta_{rs}$ are the only, up to a conjugation, indecomposable degree $-1$ endomorphism of a graded space  of curve of flags with respect to the General Linear group. In other words, any degree $-1$ endomorphism of a graded space is  a direct sum of endomorphisms of type $\delta_{rp}$. In a similar way one define the direct sum for symbols and the notion of indecomposable symbol with respect to the General Linear group.

Using \eqref{Vscoord}, it is easy to see that the symbol in the current example is isomorphic to the line generated by $$\delta_{-l, -1}\oplus\delta_{-s-k,-s-1}.$$ The universal algebraic prolongations of all possible symbols of curves of flags with respect to the General Linear group was
calculated in \cite[section8.2]{flag2} using the representation theory of $\mathfrak{sl}_2$. The Tanaka universal algebraic prolongation $\mathfrak u\Bigl(\mathbb R^{k+l}, u^F(\mathbb R(\delta_{-l, -1}\oplus\delta_{-s-k,-s-1})\bigr)\Bigr )$ of the pair $(\mathbb R^{k+l}, u^F\bigl(\mathbb R(\delta_{-l, -1}\oplus\delta_{-s-k,-s-1})\bigr)$ is calculated in \cite{mix1}, using the fact that it is isomorphic to the algebra of infinitesimal symmetries of the trivial system of the mixed order $(k,l)$, namely of the system $y^{k}=0,\,z^{(l)}=0$, with respect to the $s$-equivalence. In \cite[section 4.1]{flag1} it was shown that the symbol of curves of flags in the case $k=2$, $l=3$, and $s=0$ are not nice. We expect that this is the case also for all symbols with $k<l$ and $s>0$.

One can consider more general systems of $m$  differential equations of mixed order and $m-1$ shifts (or, shortly, a given multishift) such that this multishift has some natural restrictions. Then one can define a natural equivalence relation for such systems according to this multishift. Passing to the corresponding flag structures, one obtains that the curves of flags have the symbol $\delta$ which is a direct sum of $m$ indecomposable symbols of the type.
%$\delta_{rp}$.
Such equivalence problems and the corresponding Tanaka universal prolongations from Theorem \ref{main2} will be described in \cite{mix2}.

%Let $W$ be a linear space with the basiss
%\begin{equation}
%\label{tuple1}
%%\bigl\{\{e_i\}_{r\leq i\leq s}, \{f_i\}_{-s-\varepsilon\leq i\leq -r-\varepsilon}\bigr\}
%\{e_{-l+1},\ldots, e_-1,f_{-s-k+1},\ldots,f_{-s-1}\}.
%\end{equation}
%Define the grading on $V_{s;l}^{\mathfrak{sp}}$ such that the $i$th component equal to the span of all vectors with index $i$ appearing in the tuple \eqref{tuple1}. hen denote by $\delta_
%%_{rs}^{\varepsilon}
%{s;l}^{\mathfrak{sp}}$ a degree $-1$ endomorphism of
%$V_
%%_{rs}^{\varepsilon}
%{s;l}^{\mathfrak{sp}}$ from the symplectic algebra such that $\delta_{s;l}^{\mathfrak{sp}}(e_i)=e_{i-1}$ for $s-l+1\leq i\leq s$, $\delta_{s;l}^{\mathfrak{sp}}(e_{s-l})=0$, $\delta_{s;l}^{\mathfrak{sp}}(f_i)=f_{i-1}$ for $-s+1\leq i\leq l-s$, and
%$\delta_{s;l}^{\mathfrak{sp}}(f_{-s})=0$.
\medskip

{\bf Example 3. Geometry of distribution via abnormal extremals.}
Consider the problem of equivalence of bracket generating distributions of a given rank %on a manifold $M$
with respect to the action of the group of diffeomorphisms of the ambient manifold. The classical approach to this problem is provided by Tanaka theory \cite{aleks, tan, zeltan} described shortly in subsection \ref{quasiprosec} above. Assume that $D$ is a distribution with constant Tanaka symbol $\mathfrak t$ . Consider the corresponding bundle $P^0(\mathfrak t)$ (as defined in subsection \ref{prelim} above).  Then the construction of a canonical frame for $D$ is given by Theorem \ref{tantheor} (with $\mathfrak m=\mathfrak t$ and $\mg^0=\mg^0(\mathfrak t)$).
 %is to assign to a distribution  at a given point a special graded nilpotent algebra, called the Tanaka symbol of the distribution at this point (see the subsection for the precise definition) and to describe the construction
%of the canonical frame for all distributions with given constant Tanaka symbol
%by the construction of such frame for the flat distribution.
%in terms of
%natural algebraic operations on this symbol in the category of graded Lie algebras (the so-called
%universal algebraic prolongation, see also subsection \ref{quasiprosec} below).

However, in order to apply the Tanaka machinery
to all bracket-generating distributions of the given rank $l$ on a manifold of the given dimension $n$, one has to classify all $n$-dimensional graded nilpotent Lie algebra with $l$ generators and also one has to generalize the Tanaka prolongation procedure to distributions with non-constant symbol, because the set of all possible symbols may contain moduli (may depend on continuous parameters).
%For example,  for rank 2 distributions (with five dimensional cube) continuous parameters for symbols appear starting from dimension 8.
Note that the classification
of all symbols (graded nilpotent Lie algebras) is a quite
nontrivial problem  already for $n=7$ (see \cite{kuz}) and it looks completely
hopeless for arbitrary dimensions.

In a series of papers \cite{doubzel1, doubzel2} and the preprint \cite{doubzel3} we proposed an alternative approach which allows to avoid a classification of the Tanaka symbols. It is based on the ideas of the geometric control
theory and leads, after a symplectification of the problem,  to the equivalence problems for a particular class of structures discussed in the beginning of this section. By symplectification we mean the
lifting of the original
distribution to the cotangent bundle.

In more detail, let $D$ be a distribution on a manifold $M$. First, under some natural generic assumptions we distinguish a characteristic $1$-
foliation
(the foliation of abnormal extremals) on a special odd-dimensional submanifold
of the cotangent bundle associated with $D$
%any rank 3 distribution $D$

For this first introduce some notations.
Taking Lie brackets of vector fields tangent to a distribution $D$ (i.e. sections of $D$)   one can define a filtration $D^{-1}\subset D^{-2}\subset\ldots$
of the tangent bundle, called a \emph{weak derived flag}  or a \emph{small flag (of $D$)}. More precisely,
set
%$D^{-j}$ be the $j$-th power of the distribution $D$.
$D^{-1}=D$ and define recursively $D^{-j}=D^{-j}+[D,D^{-j+1}]$.
We
assume that all $D^{j}$ are subbundles of the tangent bundle. Denote by
$(D^j)^{\perp}\subset T^*M$ the annihilator of $D^j$, namely
\[
(D^j)^{\perp}=\{(p,q)\in T_q^*M\mid p\cdot v=0\quad\forall\,v\in D^j(q)\}.
\]
Let $\pi\colon T^*M\mapsto M$ be the canonical projection. For any $\lambda\in
T^*M$, $\lambda=(p,q)$, $q\in M$, $p\in T_q^*M$, let
$\varsigma(\lambda)(\cdot)=p(\pi_*\cdot)$ be the canonical Liouville form
and $\hat\sigma=d\varsigma$ be the standard symplectic structure on $T^*M$.

The crucial notion in the symplectification procedure  of
distributions is the notion of \emph{an abnormal extremal}. An \emph{unparametrized} Lipshitzian curve in $D^\perp$ is called
\emph{abnormal extremal of a distribution $D$} if the tangent line to it at
almost every point belongs to the kernel of the restriction
$\hat\sigma|_{D^\perp}$ of $\hat\sigma$ to $D^\perp$ at this point. As explained in the Introduction the term
``abnormal extremals'' comes from Optimal Control Theory: abnormal extremals of
$D$ are exactly Pontryagin extremals with zero Lagrange multiplier near the
functional for any variational problem with constrains, given by the
distribution $D$.

Now let us describe the submanifold of $T^*M$ foliated by the abnormal extremals.
First set
\begin{equation}
\label{tildeWD}
\widetilde W_D:=\{\lambda\in D^\perp:{\rm ker}\bigl(\hat\sigma|_{D^\perp}(\lambda)\bigr)\neq 0\}
\end{equation}
Then
\begin{enumerate}
\item
If ${\rm rank}\, D$ is odd, then $\widetilde W_D=D^\perp$, because a skew-symmetric form in an odd dimensional vector space has nontrivial kernel;
\item
If ${\rm rank}\, D=2$, then it is easy to show \cite[Proposition 2.2]{zelrigid} that $\widetilde W_D=(D^{-2})^{\perp}$;
\item
More generally, if ${\rm rank}\, D=2k$ then  the intersection of $\widetilde W_D$ with the fiber $D^\perp(q)$ of $D^{\perp}$ over a point $q\in M$
is a zero level set of a certain homogeneous polynomial of degree $k$ on $D^\perp(q)$. This polynomial at a point $\lambda=(p, q)\in D^\perp$ is the  Pfaffian of the so-called Goh matrix at $\lambda$: if $X_1,\ldots X_k$ is a local basis of the distribution $D$, then the Goh matrix at $\lambda$ (w.r.t. this basis) is the matrix $\bigl(p\cdot[X_i,X_j](q)\bigr)_{i,j=1}^{2k}$.
%In particular, $\widetilde W_D$ is a codimension $1$ submanifold of $D^\perp$
\end{enumerate}

In the sequel, for simplicity, we will assume that ${\rm rank}\, D$ is odd or ${\rm rank}\, D=2$.
In both cases $\widetilde W_D$ is an odd dimension submanifold.
Therefore ${\rm ker}\bigl(\hat\sigma|_{\widetilde W_D}(\lambda)\bigr)$ is nontrivial.
Define a subset $W_D$ of $\widetilde W_D$ as follows:

\begin{equation}
\label{tildeWD}
W_D:=\{\lambda\in \widetilde W_D:{\rm ker}\bigl(\hat\sigma|_{\widetilde W_D}(\lambda)\bigr) \text{ is one-dimensional}\}.
\end{equation}
By direct calculation one can show that
\begin{enumerate}
\item If ${\rm rank}\, D=2$, then $W_D=(D^{-2})^\perp\backslash (D^{-3})^\perp$ (\cite[section 2]{zelrigid});
\item If ${\rm rank}\, D=3$, then $W_D=D^\perp\backslash (D^{-2})^\perp$ (\cite{doubzel3}).
\end{enumerate}
Consequently, for any bracket generating rank $2$ or rank $3$ distribution on a manifold $M$ with
$\dim M\geq 4$ the set $W_D$ is an open and dense subset of $\widetilde W_D$.
In the sequel we will assume that the set $W_D$ is an open and dense subset of $\widetilde W_D$. Note that this is a generic assumption for distributions of odd rank greater than $3$. See also Remark \ref{WDempt} below addressing the case when this assumptions does not hold.

%If $rank D$ is odd and greater than $3$ then the set $W_D$ is an open and dense subset of $\widetilde W_D$ for generic $D$.
%Since $\dim D=3$, the submanifold $D^\perp$ has odd codimension in $T^*M$, and
By constructions, the kernels of $\hat\sigma|_{W_D}$ form the \emph{characteristic rank $1$ distribution} $\widehat{\mathcal A}$ on $W_D$.
The integral curves  of this distribution are abnormal extremals of the distribution $D$.
Note that in general, these are not all abnormal extremals of $D$, however for our purposes it is enough to work with these abnormal extremals only.
% the
%abnormal extremals of the distribution $D$ lying in the complement to
%$(D^2)^\perp$.

Further, it is more convenient to work with the projectivization of  $\mathbb PT^*M$ rather than with $T^*M$.
Here $\mathbb PT^*M$ is the fiber bundle over $M$ with the fibers that are the projectivizations of the fibers of $T^*M$.
The canonical projection $\Pi\colon T^*M\rightarrow \mathbb PT^*M$ sends the characteristic distribution
$\widehat{\mathcal A}$ on $W_D$ to the line distribution $\mathcal A$ on
$\mathbb P W_D$ ($:=\Pi(W_D)$), which will be also called the \emph{characteristic distribution}
of the latter manifold.
%The manifold $\mathbb P W_D$ and the distribution $\mathcal A$ play the role of $\widetilde {\mathcal S}$ and $\mathcal C$, respectively, from the general linearization procedure of subsection \ref{stat}.

Further note that the corank 1 distribution on $T^*M
%\backslash S_0
$ annihilating the tautological Liouville form $\varsigma$ on $T^*M$ induces a contact distribution on $\mathbb P T^*M$, which in turns induces the even-contact (quasi-contact) distribution $\widetilde\Delta$
on $\mathbb P(D^2)^\perp\backslash \mathbb P(D^3)^\perp$. The characteristic line distribution $\mathcal A$ is exactly the Cauchy characteristic distribution of $\widetilde\Delta$, i.e. it is the maximal subdistribution of $\widetilde\Delta$ such that
\begin{equation}
\label{Cauchy}
[\mathcal A,\widetilde \Delta]\subset \widetilde \Delta.
\end{equation}
%It is more natural to work on the projectivization of the contangent bundle
%instead of the tangent bundle itself. We define the same objects on the
%projectivization of $D^\perp\backslash(D^2)^\perp$. As homotheties of the
%fibers of $D^\perp$ preserve the characteristic line distribution, the
%projectivization induces the \emph{characteristic line distribution} on
%$\PDD$, which will be denoted by $\C$. It defines
%\emph{the characteristic $1$-foliation}, and its leaves are called the
%\emph{abnormal extremals of the distribution $D$} on
%$\PDD$.
%
%The distribution $\C$ can be defined equivalently in the following way. Take
%the corank $1$ distribution on $D^\perp\backslash(D^2)^\perp$, given by the
%Pfaffian equation $\varsigma|_{D^\perp}=0$ and push it forward under
%projectivization to $\PD^\perp$. In this way we obtain a corank 1 distribution
%on $\PD^\perp$, which will be denoted by $\widetilde\Delta$. The distribution
%$\widetilde\Delta$ defines a quasi-contact structure on the even dimensional
%manifold $\PDD$ and $\C$ is exactly the characteristic distribution of this
%quasi-contact structure.
%Moreover, the symplectic form $\hat\sigma$ induces the
%antisymmetric form on each subspace of a distribution $\widetilde\Delta$,
%defined up to a multiplication by a constant. This antisymmetric form will be
%denoted by~$\tilde\sigma$.

Further, let $\tilde \pi:\mathbb P T^*M\mapsto M$ be the canonical projection. Let $\mathcal J$ be the pullback of the original distribution $D$ to $\mathbb PW_D$ by the canonical projection $\tilde\pi$:
\begin{equation}
\label{pullJ}
\mathcal J(\lambda)=\{v\in T_\lambda \mathbb P W_D: \tilde\pi_* v\in D(\tilde\pi(\lambda))\}
\end{equation}
and $V$  be the tangent space to the fibers of $\mathcal W_D$
\begin{equation}
V(\lambda)=\{v\in T_\lambda \mathbb P W_D: \tilde\pi_* v=0\}.
\end{equation}
%, i.e the tangent space to the fibers of \tb{$\mathcal H_D$}.
%\pause
Note that $V+\mathcal A\subset \mathcal J$. We work with the distributions $\mathcal A$, $V$, and $\mathcal J$ instead of the original distribution $D$.

Now define a sequence of subspaces ${\mathcal
J}^{i}(\lambda)$, $\lambda\in \mathbb P W_D$, by the following recursive formulas for $i<0$:
\begin{equation}
\label{Jidef}
{\mathcal J}_{i-1}(\lambda):=[\mathcal A, \mathcal J_{i}](\lambda), \quad \mathcal J_{-1}(\lambda)= \mathcal J(\lambda)
\end{equation}

Directly from the fact that $\mathcal A$ is a line distribution it follows that
\begin{equation}
\label{J-2-J-1}
\dim {\mathcal J}_{-2}(\lambda)-{\mathcal J}_{-1}(\lambda)\leq {\rm rank} D-1.
\end{equation}

From \eqref{Cauchy} it follows that ${\mathcal
J}_i\subset \widetilde\Delta$ for all
$i<0$.
%Simple counting of dimensions implies that $rank \Delta=2n+1$ so that  ${\rm dim}\, {\mathcal
%J}^{(i)}(\lambda)\leq 2n+1$.
Note that the symplectic form $\hat\sigma$ induces the
antisymmetric form $\widetilde \sigma$ on each subspace of a distribution $\widetilde\Delta$,
defined up to a multiplication by a constant, and
$\mathcal A$ is exactly the distribution of kernels of this form.
%: for this take the restriction to $\widetilde\Delta(\lambda)$ of the differential of any $1$-form annihilating $\widetilde \Delta$.

Given a subspace $\Lambda$ of $\widetilde\Delta(\lambda)$ denote by $\Lambda^\angle$ the skew-symmetric complement of $W$ with respect to this form.
It is easy to show that the spaces $\mathcal J(\lambda)$ are \emph{coisotropic} with respect to the form $\widetilde\sigma$, i.e. they contain their skew symmetric complement: $\mathcal J(\lambda)^\angle\subseteq \mathcal J(\lambda)$. Moreover, if $\rm rank\, D=2$ then in fact $\mathcal J(\lambda)^\angle= \mathcal J(\lambda)$. Using the operation of skew-symmetric complement we can define the subspaces $\mathcal J_i(\lambda)$ for $i\geq 0$ as follows

\begin{equation}
\label{Jskewcomp}
\mathcal J_i(\lambda)=\begin{cases}
\bigl(\mathcal J_{-i-1}(\lambda)\bigr)^\angle\cap V(\lambda), & {\rm rank}\, D=2,\\
\bigl(\mathcal J_{-i-2}(\lambda)\bigr)^\angle\cap V(\lambda), & {\rm rank}\, D \text{ is odd }.
\end{cases}
\end{equation}

Note also that if ${\rm rank}\, D$ is odd,  then $\mathcal J(\lambda)^\angle=V(\lambda)+\mathcal A(\lambda)$, which implies that in this case

\begin{equation}
\label{J0V}
\mathcal J_0(\lambda)= V(\lambda).
\end{equation}
Similarly, if ${\rm rank}\, D=2$ and $\dim D^2(q)=3$, then
$\mathcal J_{-2}(\lambda)^\angle=V(\lambda)+\mathcal A(\lambda)$ for any $\lambda\in \mathbb P W_D$ with $\pi(\lambda)=q$, so formula \eqref{J0V}
%$J_0(\lambda)= V(\lambda)$
holds in this case as well. Besides, it is easy to see that $[\mathcal A, \mathcal J_i]\subseteq \mathcal J_{i-1}+\mathcal A$ also for
nonnegative $i$.

Further, for a generic point $q\in M$ there is a neighborhood $U$  and an open and dense subset $\widehat U$ of $\pi^{-1}(U)\cap \mathbb P W_D$ such that for any $i\in \Z$
$\dim \mathcal J_i(\lambda)$ is the same for all $\lambda \in\widehat U$.  Then the tuple of distributions $(\widetilde \Delta, \mathcal C, \{\mathcal V_i\}_{i\in \mathbb Z})$ on $\widehat U$ with $\mathcal C=\mathcal A$ and $\mathcal V_i=\mathcal J_i$ satisfies properties (P1)-(P4) above. We will say that this tuple is \emph{associated with the distribution $D$} by the symplectification. From our constructions and formula \eqref{J0V} it follows immediately that two distributions are  equivalent if and only if the corresponding tuples of distributions associated by symplectification are equivalent.

Moreover, in most of the situations $V$ is an integral  subdistribution of $V+\mathcal A$ of maximal rank. The latter condition implies that the tuple
$(\widetilde \Delta, \mathcal A, \{\mathcal J_i\}_{i\in \mathbb Z})$ is recoverable by the linearization procedure.
The linearization $\mathcal Y^\gamma$ of the sequence of distributions $\{\mathcal J_i\}_{i\in Z}$ along the foliation $\Fol (\mathcal A)$ of abnormal extremals  at a leaf (an abnormal extremal)  $\gamma$ is called the \emph{Jacobi curve of the abnormal extremal $\gamma$.} From \eqref{Jidef}, \eqref{Jskewcomp}, and Remark \ref{compgradrem} it follows that the curves of flags satisfy conditions (F1)-(F3) of subsection \ref{quasiflag}.
%Note that by Remark \ref{liftD} $(\mathcal J^{(0)})^\angle=\mathcal J^{(0)}$
% Then set
%\begin{equation}
%\label{contr} {\mathcal J}^{(-i)}(\lambda)= \bigl({\mathcal J}^{(-i)}(\lambda)\bigr)^\angle.
%\end{equation}
%\{v\in T_\lambda
%\bigl((D^2)^\perp\bigr): \sigma (v, w)=0\,\, \forall w\in {\mathcal
%J}^{(i)}(\lambda)\},
%The sequence of subspaces $\{\mathcal J^{(i)}(\lambda)\}_{i\in \mathbb Z}$ defines the filtration of $\Delta(\lambda)$.
%\subsection{Lifting of the distribution to the cotangent bundle}
%We can consider the characteristic distribution $\C$ as a dynamical
%system naturally associated with any $(3,6,\dots)$-distribution and
%study the dynamics of fibers of the natural projection
%$\pi\colon\PDD\to M$.

Let $\delta$ be a symbol of a curve of symplectic flags or, shortly, a symplectic symbol. A point $\lambda\in \mathbb P W_D$ is called \emph{$\delta$- regular} if there is a neighborhood $\widehat U$ of $\lambda$ in $\mathbb P W_D$ such that for any $\tilde\lambda\in \widehat U$ if $\tilde \gamma$ is the abnormal extremal passing through  $\widetilde \lambda$, then $\delta$ is the symbol of the curve $\mathcal Y^{\tilde\gamma}$ at $\tilde\lambda$. Note that by constructions the line distribution $\mathcal A$ depends algebraically on the fibers. It turns out (see \cite{jacsymb} for detail) that from this, the fact that the set of all symplectic symbols of Jacobi curves is discrete and from the classification of these symbols given in \cite[subsection 7.2]{flag2} it follows that for distributions of rank $2$ or of odd rank (and for distributions of any rank if we work over $\mathbb C$) for a generic point $q\in M$ there is a neighborhood $U$ of $q$ in $M$ and a symplectic symbol $\delta$ such that any point from a generic subset of $\tilde \pi^{-1}(U)\cap \mathbb PW_D$ is $\delta$-regular. Moreover for a generic point  $\tilde q\in U$ the set of all $\delta$-regular points in  $\tilde \pi^{-1}(\tilde q)\cap \mathbb PW_D$ is Zariski open.
% following statement holds:  if $\gamma$ is the abnormal extremal passing through $\lambda$, then the symbol  of the Jacobi curve $\mathcal Y^\gamma$  %at $\lambda$ is isomorphic to the same symbol $\delta$.
We call the symbol $\delta$ the \emph{Jacobi symbol of the distribution $D$ at the point $q$}

Now the problem is \emph{to construct the canonical frames uniformly for all distributions with given Jacobi symbol $\delta$}.
If we apply the linearization procedure to the tuple $(\widetilde \Delta, \mathcal A, \{\mathcal J_i\}_{i\in \mathbb Z})$ on the set $\tilde \pi^{-1}(U)\cap \mathbb PW_D$ then the resulting flag structure $\bigl(\Delta, \{\mathcal Y^\gamma\}_{\gamma \in\mathcal S}\bigr)$ has the constant flag symbol $\delta$ at generic points of the curve $\mathcal Y^\gamma$ of a generic abnormal extremal $\gamma$. Despite this property is weaker than the constancy of the flag symbol, the conclusion of Theorem \ref{cor1} still holds true if we restrict ourselves to the points of the curves $\mathcal Y^\gamma$, where the flag symbol is isomorphic to $\delta$. Recall also that the symbol of $\Delta$ in this case is isomorphic to the Heisenberg algebra $\mathfrak n$ of the appropriate dimension with the grading $\mg^{-1}\oplus\mg^{-2}$, where $\mg^{-2}$ is the center. Therefore, by Theorem \ref{main2} the construction of the canonical frame for distributions with given Jacobi symbol $\delta$ is reduced to the calculation of the algebra $\mathfrak {u} \bigl(\mathfrak n, \mathfrak u^F(\delta)\bigr)$. A natural generic subclass of distributions are distributions of the so-called \emph{maximal class} \cite{doubzel1}-\cite{doubzel3}.
A distribution is called of maximal class
if all curves $\mathcal Y_{-1}^\gamma$ do not lie in a proper subspace of $\Delta(\gamma)$ for generic abnormal extremal $\gamma$. Obviously this property is encoded in the Jacobi symbol.

\begin{remark}
\label{WDempt}
 {\rm The scheme described above will work also for distributions for which the set $W_D$ is empty. Indeed, in this case instead of $W_D$ one can consider a subset of $\widetilde W_D$ for which $\dim\,{\rm ker}\bigl(\hat\sigma|_{\widetilde W_D}(\lambda)$ attains its infimum. Clearly it is an open set. Restricting on this set we still have an integrable distribution $\mathcal A$ of this kernels (of rank greater than $1$) and we can proceed with the linearization procedure as well, getting as $\mathcal Y^\gamma$ submanifolds of dimension $>1$ instead of course. Although, in this case the assumption of constancy of symbol is not automatical, Theorem \ref{main2} still can be applied under this constancy of symbol assumption.}$\Box$
\end{remark}

{\bf (a) The case of rank $2$ distributions.}
Let us describe the algebra $\mathfrak {u} \bigl(\mathfrak n, \mathfrak u^F(\delta)\bigr)$ for rank $2$ distributions of the maximal class.
Using \eqref{J-2-J-1} it is easy to show that in the case of rank $2$ distributions the condition of the maximality of class is equivalent to the fact that the flag $\mathcal Y^\gamma(\lambda)$ is a complete flag for a generic point $\lambda$ on the curve $\gamma$ \cite{doubzel2}. If we let $\dim M=n$, $n\geq 4$, the latter is also equivalent to the fact that $\dim \mathcal Y^\gamma_{n-5}(\lambda)=1$ for a generic point $\lambda$ on the curve $\gamma$. Moreover, the whole curve  $\mathcal Y^\gamma$  at generic points can be completely recovered  by osculations by the curve $\mathcal Y^\gamma_{n-5}$ in the corresponding projective space (of the dimension $2n-7$)

From this prospective, the equivalence problem for rank $2$ distributions is similar  to the contact equivalence of scalar ordinary differential equations. The only difference is that for distributions there is an underlined  (conformal) symplectic structure on $\Delta(\gamma)$.
In particular, the curves $\mathcal Y^\gamma_{n-5}$ are not arbitrary curves  in the projective space of $\Delta(\gamma)$, but they satisfy the following property: the curve of complete flags obtained from them by the  osculation must consist of symplectic flags. Such curves in a projective space are also called self-dual %in %$\Delta(\gamma)$
\cite{var}.
%Equivalently, they satisfy the following property the dual curve in $\Delta(\gamma)^*$ to the corresponding osculating curve $\mathfrak \mathfrak %J^\gamma_{3-n}$ of codimension $1$ subspaces of $\Delta(\gamma)$ is equivalent to the curve $\mathcal Y^\gamma_{n-5}$ up to isomorphism between $\Delta(\gamma)$ and $\Delta(\gamma)^*$.

The important point is that for a given $n$ there exists the unique Jacobi symbol of rank $2$ distributions of maximal class. To describe it given
a natural  $m$ let
\begin{equation}
\label{Lm}
\mathcal L_m=\displaystyle{\bigoplus_{-m-1\leq i\leq m-2}E_i}
\end{equation}
be a graded spaces endowed with a symplectic form $\omega$, defined up to a multiplication by a nonzero constant, such that $\dim E_i=1$ for every admissible $i$ and $E_i$ is skew orthogonal to $E_j$ for all $i+j\neq -3$.    Let $\tau_m$ be a degree $-1$ endomorphism of $\mathcal L_m$ from the symplectic algebra which sends $E_i$ onto $E_{i-1}$ for every admissible $i$, except  $i=-m-1$, and  $\tau_m(E_{-m-1})=0$. We also assume that $\omega(\tau_m (e), e)\geq 0$ for all $e\in E_{-1}$. Then, by our constructions, for a given $n$ the unique Jacobi symbol of rank $2$ distributions of maximal class is isomorphic to the line generated by $\tau_{n-3}$.

Disregarding the underlying conformal symplectic structure on $\mg^{-1}$,  and up to a shift in the chosen weight of the grading, this Jacobi symbol is the same as the symbol of a scalar ordinary differential equation of order $2n-6$, i.e. $\delta^1_{2n-6}$ in the notations of Example 1. Note that in the notation of Example 2 and again disregarding the underlying conformal symplectic structure it is exactly $\mathbb R\delta_{2-n, n-5}$.
%in projective spaces
%\cite{var}.

It is not clear yet if the assumption of maximality of class is restrictive.
We checked by direct computations that for $n\leq 8$ all bracket generating rank $2$ distributions with small growth vector
$(2,3,5,\ldots)$ are of maximal class.
Actually we do not have any example of
bracket generating rank $2$ distributions with small growth vector
$(2,3,5,\ldots)$  which are not of maximal class.
Comparing this to the set of Tanaka symbols, for rank $2$ distributions with five dimensional cube
if $n=6$ there are $3$ non-isomorphic  Tanaka symbol, if $n=7$ there are  $8$ non-isomorphic Tanaka symbols, and the
continuous parameters in the set of all Tanaka symbols appear starting from dimension $8$.
\emph{Since by above, all such distributions with given $n$, at least for $n\leq 8$, are of maximal class and therefore have the same Jacobi symbols,
it already shows that starting with the Jacobi symbol
instead of Tanaka symbols as a basic characteristic of rank $2$ distributions, we get much more uniform construction of the canonical frames.}

Further, similar to the case of scalar ordinary differential equations (of order $2n-6$),
the corresponding flat curve is a curve of complete flags,
consisting of all osculating subspaces of the rational normal curve
in a $(2n-7)$-dimensional projective space and the universal algebraic prolongation  $\mathfrak u^F(\tau_{n-3})$ of the Jacobi symbol of $\tau_{n-3}$ is  equal to the image of the irreducible
embedding of $\gl(2,\mathbb R)$ into $\mathfrak {csp}(\mg^{-1})$ (after identifying the graded  symplectic spaces $\mg^{-1}$ and $\mathcal L_{n-3}$).
%, where
%$\dim {\rm gr} V=2n-6$.  t
The Tanaka universal prolongation $\mathfrak u\bigl(\mathfrak n, \mathfrak u(\tau_{n-3})\bigr)$ of the pair $\bigl(\mathfrak n, \mathfrak u(\tau_{n-3})\bigr)$ is equal to the split real form of the exceptional simple Lie algebra $G_2$ for
$n=5$, as expected from the classical Cartan work \cite{cartan}, and to the semidirect sum $\gl(2,\mathbb R)\leftthreetimes\mathfrak n$
%of $\gl(2,\mathbb %R)$\mathfrak u(\tau_{n-3})$
%itself
for $n>5$ , i.e. already the first algebraic prolongation $\mathfrak u\bigl(\mathfrak n, \mathfrak u(\tau_{n-3})\bigr)$ vanishes in this case, as expected from \cite{doubzel1, doubzel2}.
 %Here by $ \mathfrak n\oplus\mathfrak u(\tau_{n-3})$
 Here in the semidirect sum of $\gl_2(\mathbb R)$ and $\mathfrak n=\mg^{-1}\oplus \mg^{-2}$ the algebra  $\gl_2(\mathbb R)$ acts irreducibly on $\mg^{-1}$. Note that in the case $n=4$ the local equivalence problem for generic bracket generating rank $2$ distribution is trivial: any two generic germ of rank $2$ distribution in $\mathbb R^4$ are equivalent (the Engel normal form). In this case  the algebra $ \mathfrak n\oplus\mathfrak u(\tau_1)$ is infinite dimensional and its completion is isomorphic to the algebra of formal Taylor series of contact transformations of the jet space $J^1(\mathbb R,\mathbb R)$.

Further, %it follows from in \cite[subsection 4.1]{flag1} that the symbol $\delta^1_n$
similarly to the case of ordinary differential equations the Jacobi symbol $\tau_{n-1}$ is nice in the sense of Remark \ref{nice}.
%in the considered case  one can choose a nice normalization condition on the level of curves of flags so that the quasi-principal bundle $P$, corresponding to this normalization condition, is a principal bundle.
%One can show that the normalization condition can be chosen to be  invariant with respect to the natural action of the subgroup $U_+(\mathfrak m)$ (which is isomorphic in this case to the group the upper-triangular matrices) on $C^1_+$ and
In this way, as in Example 1, the Wilczynski invariants can be computed for each curve $\mathcal Y^\gamma$ and they produce
%The fundamental invariants of curves in projective spaces were described by Wilczynski~\cite{wil}.
%%The algebra of all invariants admits a basis of so-called fundamental invariants $W_3, \dots, W_{N+1}$ of order $3,\dots,N+1$ respectively.
%The Wilczynski invariants of the curve $\mathcal Y_{-1}^\gamma$, taken for every prolonged solution $\gamma$, define
the invariants of the original rank $2$ distribution, that by analogy with \cite{doub3} can be called  the \emph{generalized Wilczinski invariants of the rank $2$ distribution}. Note that the self-duality of the curves in a projective space generating the Jacobi curves implies that some Wilczinski invariants, namely the Wilczinski invariants of odd order, vanish  automatically. The minimal order of possibly nonzero Wilczinski invariants is equal to $4$. As shown in \cite{zelcart}, in the case $n=5$ the Wilczinski invariants of order $4$ (which is the only invariant of the Jacobi curves in this case) coincides with the so-called fundamental tensor of
rank $2$ distributions in 5 dimensional manifold discovered by E. Cartan \cite{cartan}.

Finally, it can be shown that the normalization condition for Theorem \ref{cor1} can be chosen such that the resulting canonical frame
%on the bundle $P$
will be a Cartan connection over $W_D$ ($=(D^{-2})^\perp\backslash (D^{-3})^\perp$) modeled by the corresponding  $G_2$ parabolic geometry in the case of $n=5$ and  by a homogeneous space corresponding to a pair of Lie algebras  $\gl_2(\mathbb R)\leftthreetimes\mathfrak n$ and $\mathfrak t_2(\mathbb R)$ if $n>5$, where $\mathfrak t_2(\mathbb R)$ is the algebra of the upper triangle $2\times 2$-matrices.

\medskip

{\bf (b) The case of rank 3 distributions.}
Using \eqref{J-2-J-1} and disregarding for a moment the underlying conformal symplectic structure on $\mg^{-1}$ it is easy to show that the Jacobi symbol
of a rank $3$ distribution of maximal class with $6$-dimensional square is a direct sum of two indecomposable symbols with respect to the General Linear group (see \cite[section 2]{doubzel3} where it is formulated in different terms using Young diagrams instead of Jacobi symbols). From this prospective, the equivalence problem for rank $3$ distributions is similar  to the $s$-equivalence of systems of two ordinary differential equation of mixed order for some shift $s$.

In the case of rank 3 distributions with Jacobi symbol $\delta$  the structure of
the Lie algebras $\mathfrak u^F(\delta)$ and $\mathfrak u\bigl(\mathfrak n, \mathfrak u^F(\delta)\bigr)$ are much more complicated.
In contrast to the case of rank 2 distribution, here the presence of the additional conformal symplectic structure is already important on the level of algebraic prolongation of the flag symbol. First a Jacobi symbol $\delta$ being decomposable with respect the General Linear group is symplectically indecomposable in the sense of \cite[subsection 7.2]{flag2} (it is a indecomposable symbol of the type (D1) there).
Second the universal algebraic universal prolongation $\mathfrak u^F(\delta)$ with respect to conformal symplectic group is different (strictly smaller) than such prolongation with respect to the General Linear group. Both  algebras $\mathfrak u^F(\delta)$ and $\mathfrak u\bigl(\mathfrak n, \mathfrak u^F(\delta)\bigr)$ for all possible Jacobi symbols of rank $3$ distributions were explicitly described \cite[section 5]{doubzel3} using the language of Algebraic Geometry , i.e.  in terms of certain polynomials vanishing on certain tangential variety of a rational normal curve in a projective space and its secants.

The algebraic prolongation $\mathfrak u^F(\delta)$ of any symbol $\delta$ of a curve of symplectic flags  (with respect to conformal symplectic group) was described in \cite[subsection 8.3]{flag2} using the representation theory of $\mathfrak{sl}_2$. The construction of canonical frames for distributions of arbitrary rank with a given Jacobi symbol $\delta$ and in particular the algebra $\mathfrak u\bigl(\mathfrak n, \mathfrak u^F(\delta)\bigr)$ will be studied in \cite{jacsymb}.

%and for an arbitrary symbol of a curve of symplectic flags in \cite{subsection 8.3]{flag2}.
%. In the case of rank 3 distributions with Jacobi symbol encoded by a non-rectangular Young diagram the structure of
%the Lie algebras $\mathfrak u(\delta)$, $\mathfrak G(\delta)$ and its Tanaka universal prolongation are much
%more complicated. The latter can be described using the language of Algebraic Geometry , in terms of certain polynomials vanishing on certain tangential %variety of a rational normal curve in a projective space and its secants \cite{doubzel3}.

\medskip

{\bf Example 4. Geometry of sub-Riemannian, sub-Finslerian, and other structures on manifolds via normal extremals.}
As in subsection \ref{sRprelim}, let $\mathcal U$ be a submanifold of $TM$ transversal to the fibers and consider the time minimal problem associated with $\mathcal U$ .Let $H: T^*M\mapsto \mathbb R$ be the maximized Hamiltonian as in \eqref{maxH}. Assume that  it is well defined and smooth in an open domain $O\subset T^*M$ and that the corresponding $c$ level set  $\mathcal H_c$ (defined by \ref{Hc}) for
some $c>0$ (and therefore for any $c>0$ by homogeneity of $H$ on
each fiber of $T^*M$)  is nonempty and consists of
regular points of $H$ (for more general setting see \cite{jac1} or Remark \ref{nonmonotrem} below).

Further let $\mathcal H_c(q)=\mathcal H_c\cap T_q^*M$. $\mathcal
H_c(q)$ is a codimension 1 submanifold of $T^*_qM$.  For any
$\lambda\in \mathcal H_c$ denote $\Pi(\lambda)=T_\lambda
\bigr(\mathcal H_c(\pi(\lambda))\bigr)$, where $\pi:T^*M\mapsto M$
is the canonical projection. Actually $\Pi(\lambda)$ is the vertical
subspace of $T_\lambda \mathcal H_c$,
\begin{equation}
\label{pilam} \Pi(\lambda)=\{\xi\in T_\lambda \mathcal
H_c:\pi_*(\xi)=0\}.
\end{equation}
Now define a sequence of subspaces $\Pi^{i}(\lambda)$, $\lambda\in \mathcal H_c$, by the following recursive formulas for $i<0$:
\begin{equation}
\label{Pidef}
{\Pi}_{i-1}(\lambda):=[\overrightarrow H, \Pi_{i}](\lambda), \quad \mathcal J_{-1}(\lambda)= \mathcal J(\lambda),
\end{equation}
where  $\overrightarrow H$ is the Hamiltonian vector field $\overrightarrow H$, corresponding to the Hamiltonian $H$.
%Directly from the fact that $\m$ is a line distribution it follows that
%\begin{equation}
%\label{J-2-J-1}
%\dim {\mathcal J}_{-2}(\lambda)-{\mathcal J}_{-1}(\lambda)\leq {\rm rank} D-1.
%\end{equation}

%From \eqref{Cauchy} it follows that ${\mathcal
%J}_i\subset \widetilde\Delta$ for all
%$i<0$.
%Simple counting of dimensions implies that $rank \Delta=2n+1$ so that  ${\rm dim}\, {\mathcal
%J}^{(i)}(\lambda)\leq 2n+1$.
Note that the symplectic form $\hat\sigma$ induces the
$2$-form $\widetilde \sigma$ on  $\mathcal H_c$ and
$\mathbb R \overrightarrow H$ is exactly the line distribution of kernels of this form.
%: for this take the restriction to $\widetilde\Delta(\lambda)$ of the differential of any $1$-form annihilating $\widetilde \Delta$.
Given a subspace $\Lambda$ of $\widetilde\Delta(\lambda)$ denote by $\Lambda^\angle$ the skew-symmetric complement of $W$ with respect to this form.
By constructions
%the spaces $\mathcal J(\lambda)$ are \emph{coisotropic} with respect to the form $\widetilde\sigma$, i.e. they contain their skew symmetric complement: $\mathcal J(\lambda)^\angle\subseteq \mathcal J(\lambda)$. Moreover, if $\rm rank\, D=2$ then in fact
$\Pi(\lambda)^\angle= \Pi(\lambda)$. Using the operation of skew-symmetric complement we can define the subspaces $\Pi_i(\lambda)$ for $i\geq 0$ as follows:

\begin{equation}
\label{Piskewcomp}
\Pi_i(\lambda)=
%\bigl(\mathcal J_{-i-1}(\lambda)\bigr)^\angle\cap V(\lambda), & {\rm rank}\, D=2,\\
\bigl(\Pi_{-i-2}(\lambda)\bigr)^\angle\cap \Pi(\lambda).
%\end{cases}
\end{equation}

Further, for a generic point $q\in M$ there is a neighborhood $U$  and an open and dense subset $\widehat  U$ of $\pi^{-1}(U)\cap \mathcal H_c$ such that for any $i\in \Z$
$\dim \Pi_i(\lambda)$ is the same for all $\lambda \in\widehat U$.  Then the tuple of distributions $(\widetilde \Delta, \mathcal C, \{\mathcal V_i\}_{i\in \mathbb Z})$ on $\mathcal U$ with $\widetilde \Delta=T\mathcal H_c$, $\mathcal C=\mathbb R \overrightarrow H$, and $\mathcal V_i= \Pi_i$ satisfies properties (P1)-(P4). As a matter of fact, in contrast to our previous examples, the vector field $\overrightarrow H$ is distinguished on $\mathcal C$ so we are interested not in the equivalence problem of the tuple $(\widetilde \Delta, \mathcal C, \{\mathcal V_i\}_{i\in \mathbb Z})$ but in the equivalence problem of the tuple $(\widetilde \Delta, \overrightarrow H , \{\mathcal V_i\}_{i\in \mathbb Z})$
We will say that the tuple is \emph{associated with the geometric structure $\mathcal U$ by the symplectification}. Since the vertical distribution $\Pi$ is one of the elements of this tuple,  two geometric structures on $M$ are  equivalent if and only if the tuples associated with them by symplectification are equivalent. Besides, in most of the situations the tuple $(\widetilde \Delta, \overrightarrow H , \{\mathcal V_i\}_{i\in \mathbb Z})$ is recoverable from the corresponding flag structure (if one also take into account the distinguished parametrization on the curves of flags).

%Moreover, in most of the situations $V$ is an integral  subdistribution of $V+\mathcal A$ of maximal rank. The latter condition implies that the tuple
%$(\widetilde \Delta, \mathcal A, \{\mathcal J_i\}_{i\in \mathbb Z})$ is recoverable by the linearization procedure.
The linearization $\mathcal Y^\gamma$ of the sequence of distributions $\{\mathcal J_i\}_{i\in Z}$ along the foliation $\Fol (\mathcal \mathbb R  \overrightarrow H)$ of normal extremals  at a leaf (a normal extremal)  $\gamma$ is called the \emph{Jacobi curve of the normal extremal $\gamma$.} Since $\gamma$ is parameterized curve so that $\dot\gamma(t)=\overrightarrow H\bigl(\gamma(t)\bigr)$, the Jacobi curve  $\mathcal Y^\gamma$ is parameterized as well. From \eqref{Pidef}, \eqref{Piskewcomp}, and Remark \ref{compgradrem} it follows that the curves of flags satisfy conditions (F1)-(F3) of subsection \ref{quasiflag}.

%\textcolor{blue}{i want to include here some prolongation results, at least for the curve satisfying condition (G)}

\begin{remark}
\label{nonmonotrem}
{\rm The same scheme works in more general situation, when the maximized Hamiltonian is not defined (for example, for sub-pseudo-Riemannian structures, defined by a distribution $D$ and pseudo-Euclidean norms on each space $D(q)$). Assume that for some open subset $O\subset T^*M$ there exists a smooth map $u:O \mapsto \mathcal V$ such that for any $\lambda=(p,q)\in O$ the point $u(\lambda)$ is a critical point of a function $h_\lambda: \mathcal V_q\mapsto \mathbb R$, where $h_\lambda(v)\stackrel{def}{=}p(v)$. Define $\widetilde H(\lambda)=p\bigl(u(\lambda)\bigr)$. The function $\widetilde H$ is called a \emph {critical Hamiltonian} associated with the geometric structure $\mathcal U$ and one can make the same constructions as above with any critical Hamiltonian.}
\end{remark}

Further, similarly to Example 3, given a symbol $\delta$ of a parameterized curves of symplectic flags or, shortly, a parameterized symplectic symbol, one can define the notion of $\delta$-regular curve. In the case when $\mathcal U$ is a sub-Riemannian  or a sub-pseudo-Riemannian structure the set $\mathcal H_c(q)=\mathcal H_c\cap T^*_qM$ is a level set of a quadratic form and the corresponding Hamiltonian vector fields depends algebraically on the fibers. In these cases, similarly to Example 3,
for a generic point $q\in M$ there is a neighborhood $U$ of $q$ in $M$ and a parameterized symplectic symbol $\delta$ such that any point from a generic subset of $\tilde \pi^{-1}(U)\cap \mathbb PW_D$ is $\delta$-regular and this symbol $\delta$ is called
the Jacobi symbol of the structure $\mathcal U$ at $q$. For more general geometric structures, it is possible that  $\delta_1$-regular  and $\delta_2$-regular points belong to the same fiber $\mathcal H_c(q)$ for two different parameterized symplectic symbols $\delta_1$ and $\delta_2$.
In this case we can restrict ourselves to an open subset $\widehat U(\delta)$ of $\mathcal H_c$ where all points
are $\delta$-regular for some parameterized symplectic symbol $\delta$ and to proceed with the linearization procedure on this subset to get the canonical frame  for the original structure from the canonical frame for the resulting flag structure $(T\mathcal S, \{\mathcal Y^\gamma\}_{\gamma \in\mathcal S}\bigr)$ with parameterized flag symbol $\delta$. Here $\mathcal S$ is the space of normal extremals in $\widehat U$.

By analogy with Example 3 we say that a geometric structures $\mathcal U$ is said to be of \emph{maximal class}
if all curves $\mathcal Y_{-1}^\gamma$ do not lie in a proper subspace of $\Delta(\gamma)$ for generic normal extremal $\gamma$.
It was  proved in \cite{agrnorm} (although using a different but equivalent terminology) that any sub-Riemannian structure on a bracket-generating  distribution is of maximal class.

Further, let $\tau_m$ denote the degree $-1$ endomorphism of a $2m$-dimensional graded symplectic space as defined in Example 3, case a, after the formula \eqref{Lm}
(i.e. the endomorphism generating the symbol of rank 2 distribution of maximal class on $\mathbb R^{m+3}$).

From \cite{li2} it follows that for any sub-Riemannian structure on a bracket-generating  distribution with Jacobi symbol $\delta$  and, more generally, for any geometric structure $\mathcal U$ with the maximized Hamiltonian being well defined  and smooth in the set $\widehat U$ of $\delta$-regular points  and such that $\mathcal U$ is of maximal class the parameterized flag symbol $\delta$  is a direct sum of endomorphisms of type $\tau_m$.

More generally, fix two functions $N_+, N_-:\mathbb N\rightarrow \mathbb N\cup\{0\}$ with finite support and assume that the parameterized symplectic symbol $\delta$ is the direct sum of endomorphisms of type $\tau_m$ and $-\tau_m$, where $\tau_m$ appears $N_+(m)$ times and
$-\tau_{m}$ appears $N_-(m)$ times in this sum for each $m\in \mathbb N$. These symbols correspond to curves in a Lagrangian Grassmannian satisfying condition (G) in the terminology of the previous papers of the second author with C. Li (\cite{li1,li2}) and they may appear, for example, after symplectificatiom/linearization of sub-(pseudo)-Riemannian structures.
Then from the results of  \cite{li1,li2} or more general results of \cite[subsection 8.3.6]{flag2} it follows that the non-negative part $\mathfrak u_+^{F, \text{par}}(\delta)$ of $\mathfrak u^{F, \text{par}}(\delta)$ is equal to $\displaystyle{\bigoplus_{m\in\mathbb N} \mathfrak{so}(N_+(m), N_-(m))}$ and it is actually equal to the zero component $\mathfrak u_0^{F, \text{par}}(\delta)$ of $\mathfrak u^{F, \text{par}}(\delta)$.

Moreover, $\delta$ is a nice symbol so that applying Theorem \ref{quasidelta} to the corresponding flag structure we obtain a principal bundle $P(\delta)$ over the space of normal extremals $\mathcal S$  in $\widehat U(\delta)$ with the Lie algebra of the structure group isomorphic to  $\displaystyle{\bigoplus_{m\in\mathbb N} \mathfrak{so}(N_+(m), N_-(m))}\oplus \mathbb R\delta$. Moreover, this bundle $P$ induces the principle bundle $P_1(\delta)$ over $\widehat U(\delta)$ with the structure group ${\displaystyle{\prod_{m\in\mathbb N}} O\bigl(N_+(m), N_-(m)\bigr)}$ (note that the bundles  $P(\delta)$ and $P_1(\delta)$ coincide as sets). In particular, it gives the canonical (pseudo-) Riemannian metric on $\widehat U(\delta)$, which immediate implies that the first algebraic prolongation of the pair $\bigl(\mathfrak m, \mathfrak u^{F, \text{par}}(\delta)\bigr)$ is equal to $0$, as in a (pseudo-) Riemannian case (here $\mathfrak m$ is a commutative Lie algebra of the appropriate dimension).  In other words, \emph{the canonical frame of Theorem \ref{cor1} (or Theorem \ref{main2}) applied to the flag structure (or the tuple of distributions) associated with the geometric structure $\mathcal U $  can be constructed already on the bundle $P(\delta)$.}

 Note that, as already mentioned in \cite{li1, li2}, this type of constructions gives not only a canonical (pseudo-) Riemannian metric on $\widehat U(\delta)$ but a canonical splitting of each  tangent spaces to any point  of $\widehat U(\delta)$ such that each space of the splitting is endowed with the canonical (pseudo-) Euclidean structure.

Finally, note that not any parameterized symplectic symbol  is the direct sum of endomorphisms of type $\tau_{m_1}$ and $-\tau_{m_2}$ (or of type (D2) in the terminology of \cite[section 7.2]{flag2}), because there is another type of symplectically indecomposable degree $-1$ endomorphisms (type (D1) in the same paper), which can be used in this direct sum. Similarly to the Jacobi symbols of rank 3 distributions these symplectically indecomposable endomorphisms are sums of $2$ indecomposable endomorphisms with respect to the General Linear group. The algebras $\mathfrak u^{F, \text{par}}(\delta)$ and $\mathfrak u\bigl(\mathfrak m, \mathfrak u^{F, \text{par}}(\delta)\bigr)$ for arbitrary parameterized symplectic symbol will be described elsewhere.

%There are two types of indecomposable degree $-1$ of endomorphisms of symplectic graded space: endomorphisms of type $\tau_m$ (or (D1) in the terminology

%From this and Remark \ref{parprolongrem} it follows that in the parameterized case $\mathfrak u_+(\mathfrak m)=\displaystyle{\bigoplus_{m\in\mathbb N} %\mathfrak{so}(N_+(m), N_-(m))}$

%\textcolor{blue}{Here I want to include some prolongation results in the case when symbol of Jacobi curves are direct sums of $\tau_m$ }

%Thus, at any generic point $\lambda\in\PDD$ we
%have a flag of subbundles:
%\[
%0\subset \C = \mJ^{(-m-1)} \subset \mJ^{(-m)}\subset \dots
%\subset\mJ^{(-1)}=\C+V\subset\mJ^{(0)}=\pi^*D\subset\]
%\[\mJ^{(1)}\subset\dots\subset\mJ^{(m)}=\widetilde\Delta,
%\]
%where the dimension gap between two neighbors in this sequence is either $2$ or
%$1$.

%Let $D$ be a bracket-generating distribution on a manifold $M$.
%\textcolor{blue}{Will be added}

\section{Proof of Theorem \ref{maintheor}: First prolongation of quasi-principle bundle}
\label{firstprolongsec}
\setcounter{equation}{0}
\setcounter{theorem}{0}
\setcounter{lemma}{0}
\setcounter{proposition}{0}
\setcounter{definition}{0}
\setcounter{cor}{0}
\setcounter{remark}{0}
\setcounter{example}{0}

Let $P^0$ be a quasi-principle bundle of type $(\mathfrak m, \mg^0)$.
Let $\Pi_0:P^0\to \mathcal S$ be the canonical projection. The filtration $\{\Delta^i\}_{i<0}$ of $T\mathcal S$ induces a filtration $\{\Delta^i_0\}_{i\leq 0}$ of $T P^0$ as follows:
\begin{equation}
\label{P0filt}
\begin{split}
~&\Delta^0_0=\ker (\Pi_0)_*,\\
~& \Delta^i_0(\vf)=\Bigl\{v\in T_\vf P^0: (\Pi_0)_*v\in \Delta^i\bigl(\Pi_0(\vf)\bigr)\Bigr\} ,\quad \forall i<0, \quad \vf\in P^0
\end{split}
\end{equation}
We also set $\Delta^i_0=0$ for all $i>0$.
Note that $\Delta^0_0(\vf)$ is the tangent space at $\vf$ to the corresponding fiber of $P^0$ and therefore can be identified with the subspace $L_{\vf}^0$ of $\mg^0(\mathfrak m)$ via the map $\Omega(\vf): T_\vf\bigl(P^0\bigl(\Pi_0(\vf)\bigr)\bigr)\mapsto L_\vf^0$, defined by \eqref{omega}.
 %Denote by $I_\lambda:\mg^0\to D^0_0(\lambda)$ the identifying isomorphism.

Besides, all spaces $L_\vf^0$ can be canonically identified with one vector space. For this take a subspace $\mathcal M_0$ of the space $\mg^0(\mathfrak m)\subset \gl(\mg^{-1})$ such that the corresponding graded space $\gr \mathcal S_0$ is complementary to $\gr L_{\vf}^0$ in
$\gl (\gr \mg^{-1})$, i.e.
\begin{equation}
\label{splitM0}
\gl (\gr \mg^{-1})=\gr L_{\vf}^0\oplus \gr \mathcal M_0.
\end{equation}
Recall that by condition (2) of Definition \ref{quasidef} the space $\gr L_{\vf}^0$ does not depend on $\vf$, so the choice of $\mathcal M_0$ as above is indeed possible. Therefore
\begin{equation}
\label{splitM}
\gl (\mg^{-1})=L_{\vf}^0\oplus \mathcal M_0.
\end{equation}
 for any $\vf\in P^0$. Let
 \begin{equation}
\label{LO}
L^0:=\gl (\mg^{-1})/\mathcal M_0,
\end{equation}

 The splitting \eqref{splitM} defines the identification ${\rm Id}_\vf^0$ between the space $L_\vf^0$ and the factor-space
%\begin{equation}
%\label{LO}
%$L^0:=
$\gl (\mg^{-1})/\mathcal M_0$,
%\end{equation}
${\rm Id}_\vf^0:L_\vf^0\mapsto L^0$. The space $L^0$ has the natural filtration
 induced by the filtration on $\gl (\mg^{-1})$ . The identification ${\rm Id}_\vf^0$ preserves the filtrations on the spaces $L_\vf^0$ and $L^0$. Note that by condition (3) of Definition \ref{quasidef} we have the following identifications:
 \begin{equation}
 \label{id3}
 \mg^0\cong \gr L_\vf^0\cong \gr L^0.
 \end{equation}
The space $\mathcal M_0$ is called the \emph{identifying space for the zero prolongation}.

Now
fix a point $\vf\in P^0$ and let  $\pi_0^i:\Delta_0^i(\vf)/ \Delta_0^{i+2}(\vf)\to \Delta_0^i(\vf)/ \Delta_0^{i+1}(\vf)$ be the canonical projection to the factor space.
Note that $\Pi_{0_*}$ induces an isomorphism between  the space $\Delta_0^i(\vf)/ \Delta_0^{i+1}(\vf)$ and the space $\Delta^i\bigr(\Pi_0(\vf)\bigl)/ \Delta^{i+1}\bigl(\Pi_0(\vf)\bigr)$ for any $i<0$. We denote this isomorphism by $\Pi_0^i$.
% as on the map from
%$\Delta_0^i(\lambda)/ \Delta_0^{i+2}(\lambda)$ to $\Delta^i\bigl(\Pi_0(\lambda)\bigr)/ \Delta^{i+1}\bigl(\Pi_0(\lambda)\bigr)$.
The fiber of the bundle $P^0$ over a point $\gamma\in \mathcal S$ is a subset of the set of all maps

\begin{equation*}\vf\in \bigoplus_{i<0} \text {Hom}\bigl(\mg^i,
\Delta^i(\gamma)/ \Delta^{i+1}(\gamma)\bigr),
\end{equation*}
 which are isomorphisms of the graded Lie algebras $\mathfrak{m}=\displaystyle{\bigoplus_{i<0} \mg^i}$ and
$\displaystyle{\bigoplus_{i<0} \Delta^i(\gamma)/\Delta^{i+1}(\gamma)}$.  Let $\widehat P^1$ be the bundle over $P^0$ with the fiber $\widehat P^1(\vf)$ over $\vf\in P^0$ consisting
%We are going to construct a new bundle $P^1$
%over the bundle $P^0$ such that the fiber of $P^1$ over a point $\vf\in P^0$ will be a certain subset of
%the set
of all maps
\begin{equation*}\hat\vf\in \bigoplus_{i< 0} \text{Hom}\bigl(\mg^i,
\Delta_0^i(\vf)/ \Delta_0^{i+2}(\vf)\bigr)\oplus \text{Hom}\bigl(L%_\vf
^0,
\Delta_0^0(\vf)\bigr)
\end{equation*}
 such that
\begin{equation}
\label{hat}
\begin{split}
~&\vf|_{\mg^i}=\Pi_0^i\circ\pi_0^i\circ\hat\vf|_{\mg^i}, \quad \forall i<0,\\
~&\hat\vf|_{L_\vf^0}=\bigl(\Omega(\phi)|_{L_\vf^0})\bigr)^{-1}\circ \bigl({\rm Id}_\vf^0\bigr)^{-1}.
\end{split}
\end{equation}
The bundle $\widehat P^1$ is an affine bundle as shown below. Our goal in this section is to distinguish in a canonical way an affine subbundle of $\widehat P^1$ of minimal possible dimension.

For this fix again a point $\vf\in P^0$. For any $i<0$ choose a subspace $H^i\subset \Delta_0^i(\vf)/ \Delta_0^{i+2}(\vf)$, which is a complement of
$\Delta_0^{i+1}(\vf)/\Delta_0^{i+2}(\vf)$ to $\Delta_0^i(\vf)/ \Delta_0^{i+2}(\vf)$:
\begin{equation}
\label{H1}
\Delta_0^i(\vf)/ \Delta_0^{i+2}(\vf)=\Delta_0^{i+1}(\vf)/ \Delta_0^{i+2}(\vf)\oplus H^i.
\end{equation}
Then the map $\Pi^i_0\circ\pi^i_0|_{H^i}$ defines an isomorphism between $H^i$ and $\Delta^i\bigl(\Pi_0(\vf)\bigr)/ \Delta^{i+1}\bigl(\Pi_0(\vf)\bigr)$.  So, once a tuple of subspaces $\mathcal H=\{H^i\}_{i<0}$ is chosen,
one can define a map
\begin{equation*}
\vf^{\mathcal H}\in \displaystyle{\bigoplus_{i< 0} \text{Hom}\bigl(\mg^i,
\Delta_0^i(\vf)/ \Delta_0^{i+2}(\vf)\bigr)}\oplus  \text{Hom}\bigl(L%_\vf
^0,
\Delta_0^0(\vf)\bigr)
\end{equation*}
 as follows
\begin{equation}
\label{vfH}
\begin{split}
~&\vf^{\mathcal H}|_{\mg^i}=
%\begin {cases}
(\Pi^i_0\circ\pi_0^i|_{H^i})^{-1}\circ\vf|_{\mg^i} \text{ if } i<0\\
~&\vf^{\mathcal H}|_{L_\vf^0}=\bigl(\Omega(\phi)|_{L_\vf^0})\bigr)^{-1}\circ \bigl({\rm Id}_\vf^0\bigr)^{-1}
\end{split}
%\end{cases}
\end{equation}
Clearly $\hat\vf=\vf^{\mathcal H}$ satisfies  \eqref{hat}. Tuples of subspaces $\mathcal H=\{H^i\}_{i<0}$ satisfying \eqref{H1}
play here the same role as horizontal subspaces (an Ehresmann connection) in the prolongation of the usual $G$-structures (see, for example, \cite{stern} or \cite{zeltan}[section 2]. Can we choose a tuple
$\{H^i\}_{i<0}$ in a canonical way? For this, by analogy with the prolongation of $G$-structure, we introduce a \textquotedblleft partial soldering form\textquotedblright of the bundle $P^0$ and the structure function of a tuple $\mathcal H$. The \emph{soldering form} of $P^0$ is a tuple
$\Omega_0=\{\omega_0^i\}_{i<0}$, where $\omega^i_0$ is a $\mg^i$-valued linear form on $\Delta_0^i(\vf)$
%/D_0^{i}(\lambda)$
defined by
\begin{equation}
\label{soldpart}
\omega_0^i(Y)=\vf^{-1}\Bigl(\Bigl((\Pi_0)_*(Y)\Bigr)_i\Bigr),
%\pi_0^i
\end{equation}
where $\Bigl((\Pi_0)_*(Y)\Bigr)_i$ is the equivalence class of $(\Pi_0)_*(Y)$ in $\Delta^i(\gamma)/\Delta^{i+1}(\gamma)$.
Observe that $\Delta_0^{i+1}(\vf)=\ker \omega_0^i$. Thus the form $\omega_0^i$ induces the $\mg^i$-valued form $\bar \omega_0^i$ on $\Delta_0^i(\vf)
/\Delta_0^{i+1}(\vf)$.
The \emph{structure function $C_{\mathcal H}^0$ of the tuple $\mathcal H=\{H^i\}_{i<0}$} is the element of the space
\begin{equation}
\label{A0}
\mathcal A_0=
\left(\bigoplus_{i=-\mu}^{-2} {\rm Hom}(\mg^{-1}\otimes\mg^i,\mg^{i})\right)
\oplus {\rm Hom}(\mg^{-1}\wedge\mg^{-1},\mg^{-1})
\end{equation}
 defined as follows.
Let $\text{pr}_i^{\mathcal H}$ be the projection of $\Delta_0^i(\vf)/ \Delta_0^{i+2}(\vf)$ to $\Delta_0^{i+1}(\vf)/ \Delta_0^{i+2}(\vf)$
%($\sim D^{i+1}(x)/ D^{i+2}(x)$)
parallel to $H^i$ ( or corresponding to the splitting \eqref{H1}). Given vectors $v_1\in \mg^{-1}$ and $v_2\in\mg^{i}$, take two vector fields $Y_1$ and $Y_2$ in a neighborhood of $\lambda$ in $P^0$ such that $Y_1$ is a section of $\Delta_0^{-1}$, $Y_2$ is a section of $\Delta_0^i$, and
\begin{equation}
\label{Y1Y2}
\begin{split}
~&\omega_0^{-1}(Y_1)\equiv v_1,\quad \omega_0^i(Y_2)\equiv v_2,
\\
~& Y_1(\vf)=\vf^{\mathcal H}(v_1),\quad Y_2(\vf)\equiv \vf^{\mathcal H}(v_2)\,\,{\rm mod}\, \Delta_0^{i+2}(\vf).
\end{split}
\end{equation}
Then set
\begin{equation}
\label{structT1}
C_{\mathcal H}^0(v_1,v_2)\stackrel{\text{def}}{=}\bar\omega_0^i\Bigl({\rm pr}_{i-1}^{\mathcal H}\bigl([Y_1,Y_2](\vf)\bigr)\Bigr).
\end{equation}
In the above formula we take the equivalence class of the vector $[Y_1, Y_2](\vf)$ in $\Delta_0^{i-1}(\vf)/\Delta_0^{i+1}(\vf)$ and then apply ${\rm pr}_{i-1}^{\mathcal H}$.
% to it.
It is easy to show (see \cite[section 3]{zeltan}) that $C_{\mathcal H}^0(v_1,v_2)$ does not depend on the choice of vector fields $Y_1$ and $Y_2$, satisfying \eqref{Y1Y2}.
%Indeed, assume that $\widetilde Y_1$ and $\widetilde Y_2$ are another pair of vector fields in a neighborhood of $\lambda$ in $P^0$ such that $\widetilde Y_1$ is a section of $D_0^{-1}$, $\widetilde Y_2$ is a section of $D_0^i$, and they satisfy \eqref{Y1Y2}
%with  $Y_1$, $Y_2$ replaced by $\widetilde Y_1$, $\widetilde Y_1$. Then
%\begin{equation}
%\label{zz}
%\widetilde Y_1=Y_1+Z_1, \quad \widetilde Y_2=Y_2+Z_2,
%\end{equation}
%where $Z_1$ is a section of the distribution $D_0^0$ such that $Z_1(\lambda)=0$ and $Z_2$ is a section of the distribution $D_0^{i+1}$ such that $Z_2(\lambda)\in D_0^{i+2}(\lambda)$.
% It follows that $[Y_1, Z_2](\lambda) \in D_0^{i+1}(\lambda)$ and $[Y_2, Z_1](\lambda)\in D_0^{i+1}(\lambda)$. This together with the fact that $[Z_1,Z_2]$ is a section of $D_0^{i+1}$ imply that
%\begin{equation*}
%[\widetilde Y_1,\widetilde Y_2](\lambda)\equiv [Y_1,Y_2]\,
%\text{ mod }\,\,D_0^{i+1}(\lambda).
% \end{equation*}
%%and by
%From
%%definition
%\eqref{structT1} we see that the structure function is independent
%%$C_{\mathcal H}^0(v_1,v_2)$
%of the choice of vector fields $Y_1$ and $Y_2$.
%
%%\begin{equation}
%%\label{Z}
%%\begin{split},
%
%%\end{equation*}

We now take another tuple $\widetilde {\mathcal H}=\{\widetilde H^i\}_{i<0}$ such that
\begin{equation}
\label{tildeH1}
\Delta_0^i(\vf)/ \Delta_0^{i+2}(\vf)=\Delta_0^{i+1}(\vf)/ \Delta_0^{i+2}(\vf)\oplus \widetilde H^i
\end{equation}
and consider how the structure functions $C_{\mathcal H}^1$ and $C_{\widetilde{\mathcal H}}^1$ are related.
By construction, for any vector $v\in\mg ^i$ the vector $\vf^{\widetilde{\mathcal H}}(v)-\vf^{\mathcal H}(v)$ belongs to
$\Delta_0^{i+1}(\vf)/ \Delta_0^{i+2}(\vf)$.
%($\sim \Delta^{i+1}(x)/ \Delta^{i+2}(x)$).
Let
\begin{equation*}
%\begin{split}
%~&
f_{\mathcal H\widetilde {\mathcal H}}(v)\stackrel{\text{def}}{=} \begin{cases}
%\vf^{-1}
\bar\omega_0^{i+1}\left(\vf^{\widetilde{\mathcal H}}(v)-\vf^{\mathcal H}(v)\right) & \text{ if } v\in \mg^i \text{ with } i<-1\\
%~&
%f_{\mathcal H\widetilde {\mathcal H}}(\lambda)(v)\stackrel{def}{=}
%\vf^{-1}
%I_\lambda^{-1}
{\rm Id}_\vf^0\circ\Omega(\vf)\left(\vf^{\widetilde{\mathcal H}}(v)-\vf^{\mathcal H}(v)\right) & \text{ if } v\in \mg^{-1}.
%\end{split}
\end{cases}
\end{equation*}
Then $f_{\mathcal H\widetilde {\mathcal H}}\in \displaystyle{\bigoplus_{i<-1}{\rm Hom}(\mg^i,\mg^{i+1})}\oplus
{\rm Hom}(\mg^{-1},L%_\vf
^0)$.
%In the opposite direction
Conversely, it is clear that for any $$f\in\displaystyle{\bigoplus_{i<-1}{\rm Hom}(\mg^i,\mg^{i+1})}\oplus{\rm Hom}(\mg^{-1},L
%_\vf
^0)$$ there exists a tuple $\widetilde{\mathcal H}=\{\widetilde H^i\}_{i<0}$, satisfying \eqref{tildeH1}, such that  $f=f_{\mathcal H \widetilde {\mathcal H}}$. In other words, the bundle $\widehat P^1$ is the affine bundle over $P^0$ such that each fiber is an affine space over the linear space $\displaystyle{\bigoplus_{i<-1}{\rm Hom}(\mg^i,\mg^{i+1})}\oplus
{\rm Hom}(\mg^{-1},L
%_\vf
^0)$.

Further, let $\mathcal A_0$ be as in \eqref{A0}. For any $\vf\in P^0$ define
a map
\begin{equation*}
\partial_0:
%\displaystyle{\bigoplus_{i<0}{\rm Hom}(\mg^i,\mg^{i+1})}
\displaystyle{\bigoplus_{i<-1}{\rm Hom}(\mg^i,\mg^{i+1})}\oplus{\rm Hom}(\mg^{-1},L_\vf^0)\rightarrow \mathcal A_0
%\displaystyle{\bigoplus_{i<0}
% {\rm Hom}(\mg^{-1}\wedge\mg^i,\mg^i)}
\end{equation*}
  by
\begin{equation}
\label{spencer0}
\partial_0 f(v_1,v_2)=[f(v_1),v_2]+[v_1,f(v_2)]-f([v_1,v_2]),
\end{equation}
where the brackets $[\, \,, \,]$ are as in the Lie algebra $\mathfrak m\oplus\mg^0(\mathfrak m)$ (see \eqref{br1}).
The map $\partial_0$ coincides with the Spencer  (or antisymmetrization) operator  in the case of $G$-structures (see, for example, \cite{stern}).
Therefore it is called the \emph{generalized Spencer operator for the first prolongation (at the point  $\vf\in P^0$)}.
%Note that in the considered generality the operator $\partial_0$ depends on the point $\vf$: the domain space of $\partial_0$ depends %on $\vf$.
Under the identification ${\rm Id}_\vf^0$ between spaces $L^0_\vf$ and $L^0$ we look at the operator $\partial_0$ as acting
\begin{equation}
\label{from}
\text{ from }
\displaystyle{\bigoplus_{i<-1}{\rm Hom}(\mg^i,\mg^{i+1})}\oplus{\rm Hom}(\mg^{-1},L^0)\text{ to }\mathcal A_0.
\end{equation}

The following formula is a cornerstone of the prolongation procedure (for the proof see \cite[Proposition 3.1]{zeltan}):
%%\begin{proposition}
%%\label{prop1}
%%The following identity holds
\begin{equation}
\label{structrans}
C_{\widetilde {\mathcal H}}^0=C_{\mathcal H}^0+\partial_0f_{\mathcal H \widetilde{\mathcal H}}.
\end{equation}

Further the filtration \eqref{filtg} on the spaces $\mg^{-1}$ induces natural (nonincreasing by inclusion) filtrations $\{\mg^{-i}_j\}_{j=-i\nu}^0$ on each space $\mg^{-i}$ with $i>0$ as follows
\begin{equation}
\label{filtgi}
\mg^{-i}_j={\rm span}\{[v_1,[v_2,[\ldots[v_{i-1}, v_i],\ldots,]: v_k\in \mg^{-1}_{j_k}, -\nu\leq j_k\leq -1, \sum_{k=1}^i j_k\geq j\}
\end{equation}
For $i<0$ let $\gr \mg^{i}=\displaystyle{\bigoplus_{j=i\nu}^{-i}\mg^{i,j}}$ be the corresponding graded spaces, where $\mg^{i,j}=\mg^{i}_j/\mg^i_{j+1}$. Also, let
$$\gr \mathfrak m=\displaystyle{\bigoplus_{j=-\mu}^{-1}}\gr \mg^i=\displaystyle{\bigoplus_{i=-\mu}^{-1}\bigoplus_{j=-i\nu}^{-i}\mg^{i,j}}$$
Then $\gr \mathfrak m$ is a bi-graded vector space. Besides, the structure of a graded Lie algebra on $\mathfrak m$ induces the structure of a bi-graded Lie algebra on $\gr \mathfrak m$ with the Lie brackets $[\cdot,\cdot]_{\gr}$ defined as follows:
If  $v_1 \in \mg^{i_1,j_1}$, $v_2\in\mg^{i_2,j_2}$ , $\tilde v_1$ and $\tilde v_2$ are representative of the equivalence classes $v_1$ and $v_2$ in $\mg^{i_1}_{j_1}$ and $\mg^{i_2}_{j_2}$, respectively, then
\begin{equation}
\label{Liebracketsgr}
[v_1, v_2]_{gr}:=[\tilde v_1,\tilde v_2] \quad {\rm mod}\, \mg^{i_1+i_2}_{j_1+j_2+1},
\end{equation}
where $[\cdot,\cdot]$ are the Lie brackets on $\mathfrak m$. Then for arbitrary $v_1$ and $v_2$ from $\gr \mathfrak m$ the Lie brackets
$[v_1,v_2]_{\gr}$ are defined by bilinearity.

Moreover, the Lie algebras $\mathfrak m$ and $\gr \mathfrak m$ are isomorphic:
any linear isomorphism $I: \mg ^{-1}\rightarrow \gr \mg^{-1}$ can be extended to an isomorphism of Lie algebras $\mathfrak m$ and $\gr \mathfrak m$ by setting:
\begin{equation*}
I\bigl([v_1,[v_2,[,\ldots[v_{i-1}, v_i],\ldots,]\bigr)=[I(v_1),[I(v_2),[\ldots[I(v_{i-1}), I(v_i)]_{\gr},\ldots,]_{\gr}
\end{equation*}
As $I$ one can take $J^{-1}$, where $J:\gr \mg^{-1}\rightarrow \mg^{-1}$ is as in condition (2) of Definition \ref{compatdefin} (with $W=\mg^{-1}$ there). Any $X\in
%\gr L_\vf^0
\gr \mg^0(\mathfrak m)\subset \gl(\gr \mg^{-1})$ can be extended to a derivation of the Lie algebra $\gr \mathfrak m$ as follows: the operator $J\circ X\circ J^{-1}$ belongs to $\mg^0$ and, in particular, can be extended to a derivation of the Lie algebra $\mathfrak m$. Let us denote this extension by $Y$. Then to define the desired extension of $X$ we set $X v:=J^{-1}\circ Y\circ J v$ for any $v\in \gr\mathfrak m$. Besides, as in \eqref{br1}, one extends the structure of Lie algebra from $\gr \mathfrak m$ to
$\gr \mathfrak m\oplus \gr
%L_\vf^0$.
\mg^0(\mathfrak m)$. Moreover, its Lie subalgebra $\gr \mathfrak m\oplus \gr
L_\vf^0$ is isomorphic to the Lie algebra $\mathfrak m\oplus \mg^0$ and the isomorphism is given by
\begin{equation}
\label{isomg0}
(v, X)\mapsto (Jv, J\circ X\circ J^{-1}), \quad v\in \gr \mathfrak m,\, X\in \gr L_\vf^0.
\end{equation}

%Using these filtrations one can also define the non-increasing by inclusion filtration
%$\{\mathfrak m_k\}_{k\in \mathbb Z}$ of the space $\mathfrak m$ such that for any $1\leq i\leq \mu$, $-\nu i\leq j\leq 0$ one has
%$$\mathfrak m_{-\frac{i(i-1)}{2}\nu+j}=\displaystyle{\bigoplus_{l=1}^{i-1}\mg^{-l}\oplus \mg^{-i}_j}.$$
%lexicographic order $\prec$ on $\mathbb Z^2$ one can also define the (non-increasing by inclusion with respect to this order) filtration
%\begin{equation}
%\label{filtm}
%\Bigl\{\displaystyle{\sum_{(i,j)\preceq (k,l)}} \mg_{l}^k\Bigr\}_{1\leq i\leq \mu, 1\leq j\leq i\nu}
%\end{equation}
 %on the space $\mathfrak m$.
%Let $\gr \mathfrak m$ be the corresponding graded space. The Lie brakets on $\mathfrak m$ induce the lie bracket $[\cdot,\cdot]_{\gr}$ on $\gr \mathfrak m$: if $v_1\in \mathfrak m_{k_1}$ and $v_2 \in \mathfrak m_{k_2}$, then
%$$[v_1+\mathfrak m_{k_1+1},v_2+\mathfrak_m]_{\gr}=[v_
 %The corresponding graded space
%$\gr \mathfrak m$ is equal to $\displaystyle{\bigoplus_{i=1}^\mu\bigoplus_{j=1}^{i\nu}\mg^i_j}$.
%Enumerate the spaces of the filtration \eqref{filtm} by a set of consequent integers isomorphic as a totally ordered set to the set %of pairs $\{(k,l): 1\leq i\leq \mu, 1\leq l\leq i\nu\}$ endowed with the lexicographic order.
 Further, if $A$ and $B$  are vector spaces endowed with nonincreasing by inclusion filtrations $\{A_j\}_{j\in \mathbb Z}$ and $\{B_j\}_{j\in Z}$, respectively, then by analogy with \eqref{grgldef} define the filtration $\left\{\bigl({\rm Hom}(A,B)\bigr)_k\right\}_{k\in\mathbb Z}$ on ${\rm Hom}(A,B)$ by
\begin{equation}
\label{Homk}
\bigl({\rm Hom}(A,B)\bigr)_k=\{X\in {\rm Hom}(A,B): X(A_j)\subset B_{j+k} \text{ for any } j\in \mathbb Z\}.
\end{equation}
With this notation we can define the filtration  on the domain space of the generalized Spencer operator $\partial_0$ as follows:
\begin{equation}
\label{domainfilt}
\left\{\displaystyle{\bigoplus_{i<-1}\bigl({\rm Hom}(\mg^i,\mg^{i+1})}\bigr)_k\oplus
\bigl({\rm Hom}(\mg^{-1},L_\vf^0)\bigr)_k\right\}_{k\in Z}.
\end{equation}

To define an appropriate filtration on the target space $\mathcal A_0$ of the operator $\partial_0$, first define the natural nonincreasing filtration of the spaces $\mg^{-1}\otimes \mg^{i}$ and $\mg^{-1}\wedge \mg^{-1}$ as follows:
\begin{eqnarray*}
&~&
\Bigl(\mg^{-1}\otimes \mg^{i}\Bigr)_j={\rm span}\{v_1\otimes v_2: v_1\in \mg^{-1}_{j_1}, v_2\in \mg^{i}_{j_2}, j_1+j_2>j\},\\
&~&
\Bigl(\mg^{-1}\wedge \mg^{i}\Bigr)_j={\rm span}\{v_1\wedge v_2: v_1\in \mg^{-1}_{j_1}, v_2\in \mg^{i}_{j_2}, j_1+j_2>j\}
\end{eqnarray*}

With this filtrations and the notation given by \eqref{Homk}, we can define the following filtration on the target space $\mathcal A_0$ of the operator $\partial_0$:
\begin{equation}
\label{targetfilt}
\left\{\Bigl(\bigoplus_{i=-\mu}^{-2} \bigl({\rm Hom}(\mg^{-1}\otimes\mg^i,\mg^{i})\bigr)_k\Bigr)
\oplus \bigl({\rm Hom}(\mg^{-1}\wedge\mg^{-1},\mg^{-1})\bigr)_k\right\}_{k\in \mathbb Z}
\end{equation}

%The introduced filtration on $\mathfrak m$ defines the natural filtration on the spaces $\gl(\mathfrak m)$, $\mathfrak m\wedge\mathfrak m$, and
%${\rm Hom}(\mathfrak m\wedge \mathfrak m,\mathfrak m)$ which in turn induces the filtrations on the spaces
%$\displaystyle{\bigoplus_{i<-1}{\rm Hom}(\mg^i,\mg^{i+1})}\oplus{\rm Hom}(\mg^{-1},L_\vf)\subset \gl(\mathfrak m)$ and
%$\mathcal A_0 \subset {\rm Hom}(\mathfrak m\wedge \mathfrak m,\mathfrak m)$.

Note that directly from \eqref{spencer0} it follows that $\partial_0$ preserves the filtrations \eqref{domainfilt} and \eqref{targetfilt} of the domain  and target spaces, i.e. it sends the $k$th space of filtration \eqref{domainfilt}
to the $k$th space of filtration \eqref{targetfilt} for any $k\in \mathbb Z$.

Now as before assume that $A$ and $B$ are two filtered vector spaces endowed with non-increasing by inclusion filtrations $\{A_k\}_{k\in\mathbb Z}$ and
$\{B_k\}_{k\in\mathbb Z}$, respectively. Let $\Upsilon:A\mapsto B$ be a linear map preserving the filtration, i.e. such that $\Upsilon(A_k)\subset B_k$ for any  $k\in\mathbb Z$. Then to $\Upsilon$  one can associate the linear map $\gr \Upsilon: \gr A\mapsto \gr B$ of the corresponding  graded spaces such that $\gr \Upsilon(a+A_{k+1})=\Upsilon(a)+B_{k+1}$.

Let us consider the map $\gr \partial_0$ associated with the generalized Spencer operator $\partial_0$.
Note that similarly to \eqref{idgr} we have the following natural identifications for the domain space and the target space of the map  $\gr \partial_0$ (which are the graded spaces corresponding to filtrations \eqref{domainfilt}  and \eqref{targetfilt} of the domain and the target space of the operator $\partial_0$, respectively):
\begin{eqnarray}
~&\label{idpart01}
\,\,\,\,\,\,\,\gr\Bigl(\displaystyle{\bigoplus_{i<-1}}{\rm Hom}(\mg^i,\mg^{i+1}\Bigr)\oplus{\rm Hom}(\mg^{-1},L_\vf^0)\cong
\displaystyle{\bigoplus_{i<-1}}{\rm Hom}(\gr\mg^i,\gr\mg^{i+1})\oplus{\rm Hom}(\gr\mg^{-1},\gr L^0_\vf),\\
~&\label{idpart02}\gr {\mathcal A}_0\cong \displaystyle{\bigoplus_{i<-1}} {\rm Hom}(\gr\mg^{-1}\otimes\gr\mg^i,\gr \mg^{i})
\oplus {\rm Hom}(\gr\mg^{-1}\wedge\gr\mg^{-1},\gr\mg^{-1})
\end{eqnarray}
%Therefore $$\gr\partial_0:\bigoplus_{i<-1}{\rm Hom}(\gr\mg^i,\gr\mg^{i+1})\oplus{\rm Hom}(\gr\mg^{-1},\gr L_\vf)\rightarrow
%\bigoplus_{i<-1} {\rm Hom}(\gr\mg^{-1}\otimes\gr\mg^i,\gr \mg^{i})
%\oplus {\rm Hom}(\gr\mg^{-1}\wedge\gr\mg^{-1},\gr\mg^{-1}).$$
In particular, from condition 2) of Definition \ref{quasidef} it follows that the domain space of $\gr \partial_0$ does not depend on a point $\vf\in P^0$. By the definition of Lie brackets
$[\cdot,\cdot]_{\gr}$ and under identifications \eqref{idpart01}-\eqref{idpart02} we have
\begin{equation}
\label{spencer0gr}
\gr\partial_0 f(v_1,v_2)=[f(v_1),v_2]_{\gr}+[v_1,f(v_2)]_{gr}-f([v_1,v_2]_{gr}),
\end{equation}
\begin{remark}
\label{idspencer}
{\rm Using the identification \eqref{isomg0} we can consider the operator $\gr \partial_0$ as the operator from $\displaystyle{\bigoplus_{i<-1}{\rm Hom}(\mg^i,\mg^{i+1})}$ to $\mathcal A_0$ satisfying the same formula as
in \eqref{spencer0}.}
\end{remark}

Now let us prove the following general lemma:
\begin{lemma}\label{lem_abc}
Let $\Upsilon\colon A\to B$ be a mapping of arbitrary filtered
vector spaces $A,B$ preserving the filtration. Let $\gr \Upsilon
\colon \gr A\to\gr B$ be the associated mapping of the corresponding
graded vector spaces. Then the following three statements hold:
\begin{enumerate}
\item $\gr (\ker \Upsilon)\subset \ker (\gr\Upsilon)$;
\item if $C$ is any subspace in $B$ such that
\begin{equation}
\label{groplus} \gr C \oplus \im \gr \Upsilon = \gr B,
\end{equation}
 then $C + \im
\Upsilon = B$;
\item under the assumptions of the previous items, the space $\gr \Upsilon^{-1}(C)$ does not depend on $C$ and coincides with
$\ker (\gr \Upsilon)$.
\end{enumerate}
\end{lemma}
\begin{proof} {\bf 1)} Suppose that $a\in A_k$ and $\Upsilon(a) = 0$. Then
$\gr \Upsilon (a + A_{k+1}) = \Upsilon(a) + B_{k+1}=0$ and $a +
A_{k+1}\in \gr A$ lies in the kernel of $\gr\Upsilon$.

{\bf 2)} Let $b$ be any element in $B^{(k)}$. Then by assumption the
element $b + B_{k+1}\in \gr B$ uniquely decomposes as $(c +
C_{k+1}) + (\Upsilon(a+A_{k+1}))$ for some elements $c +
C_{k+1}\in\gr C$ and $a+A_{k+1}\in\gr A$. Hence, we see that $(b
- c - \Upsilon(a))$ lies in $B_{k+1}$. Proceeding by induction we
get that $b = c'+\Upsilon(a')$ for some elements $c'\in C$ and
$a'\in A$.

{\bf 3)} Let $a\in\Upsilon^{-1}(C)\cap A_{k}$. Then $\gr\Upsilon(a
+ A_{k+1})$ lies in $\gr C$ and, hence, is equal to $0$. Thus, we
have $\gr \Upsilon^{-1}(C)\subset \ker(\gr\Upsilon)$.

To prove the opposite inclusion $\ker(\gr\Upsilon)\subset\gr
\Upsilon^{-1}(C)$ we actually have to show that for any $a\in
A_{k}$, satisfying $\Upsilon(a)\in B_{k+1}$, there exists $a'\in
A_{k}$ such that  $a-a'\in A_{k+1}$ and $\Upsilon (a')\in C$.
For this let $\Upsilon_{k+1}$ be the restriction of $\Upsilon$ to
$A^{(k-1)}$. Then from \eqref{groplus} it follows that $gr
C^{(k-1)}\oplus \im \gr \Upsilon_k=B^{(k-1)}$. Hence, by the
previous item of the lemma we have $$C^{(k-1)}+\im
\Upsilon_{k-1}=B^{(k-1)}.$$ From this and the assumption that
$\Upsilon(a)\in B^{(k-1)}$ it follows that there exist $c_{k-1}\in
C^{(k-1)}$ and $a_{k-1}\in A^{(k-1)}$ such that
$\Upsilon(a)=c_{k-1}+\Upsilon(a_{k-1})$. Therefore, as required $a'$
one can take $a'=a-a_{k-1}$. Indeed, $a'-a=a_{k-1}\in A^{(k-1)}$ and
$\Upsilon(a')=\Upsilon(a-a_{k-1})=c_{k-1}\in C$. This completes the
proof of the third item of the lemma.
\end{proof}

%Now we are ready to define the notion of normalization condition for the first prolongation.
Now fix a subspace
\begin{equation*}
\mathcal N_0\subset \mathcal A_0
%\displaystyle{\bigoplus_{i<0}
% {\rm Hom}(\mg^{-1}\wedge\mg^i,\mg^i)},
 \end{equation*}
such that
%complementary to  $\text{Im}\, \partial_0$,
\begin{equation}
\label{normsplitgr}
%\displaystyle{\bigoplus_{i<0}
% {\rm Hom}(\mg^{-1}\wedge\mg^i,\mg^i)}
\gr A_0= \im \,\gr \partial_0\oplus\gr\mathcal N_0.
 \end{equation}
  By analogy with $G$-structures and with principle bundles of type $(\mathfrak m, \mg^0)$ the subspace $\mathcal N_0$ is called the \emph{normalization conditions for the first prolongation}. From item (2) of Lemma \ref{lem_abc} it follows that
  \begin{equation}
\label{normsplit}
%\displaystyle{\bigoplus_{i<0}
% {\rm Hom}(\mg^{-1}\wedge\mg^i,\mg^i)}
\mathcal A_0= \im \,\partial_0+\mathcal N_0.
 \end{equation}

 Given $\vf\in P^0$ denote by $P^1(\vf)$ the following space:
  \begin{equation}
  \label{fiberP1}
 P^1(\vf)=\bigl\{\vf^{\mathcal H}: \mathcal H=\{H^i\}_{i<0} \text{ satisfies } \eqref{H1} \text{ and }C_{\mathcal H}^0\in \mathcal N_0\bigr\},
 \end{equation}
 where
 %, as before,
 %$\mathcal H=\{H^i\}_{i<0}$ are tuples of spaces satisfying \eqref{H1} and maps
 $\vf^{\mathcal H}$ is defined by
 \eqref{vfH}. Then from the formulas \eqref{structrans} and  \eqref{normsplit} it follows that $P^1(\vf)$ is not empty.
 Moreover, if $\vf^{\mathcal H}\in P^1(\vf)$ for some tuple of spaces $\mathcal H$, then $\vf^{\widetilde{\mathcal H}}\in P^1(\vf)$ for another tuple of spaces $\widetilde {\mathcal H}$ if and only if
 $$\partial_0 f_{\mathcal H, \widetilde{\mathcal H}}\in \mathcal N_0.$$
 Here $\partial_0$ is acting as in \eqref{from}.
 Therefore, $P^1(\vf)$ is an affine space over the linear space
\begin{equation}
\label{L1}
L^1_\vf:=(\partial_0)^{-1}(\mathcal N_0)\subset \displaystyle{\bigoplus_{i<-1}{\rm Hom}(\mg^i,\mg^{i+1})}\oplus{\rm Hom}(\mg^{-1},L^0).
\end{equation}

From item (3) of Lemma
\ref{lem_abc} it follows that the corresponding graded space $\gr L^1_\vf$ (with respect to filtration \eqref{domainfilt} of the domain space of $\partial_0$) does not depend on the normalization condition $\mathcal N_0$ and coincides with $\ker \gr \partial_0$. Taking into account Remark \ref{idspencer} we get that  under identification \eqref{isomg0} $\ker \gr \partial_0\cong\mg^1$, where $\mg^1$ is the first algebraic prolongation of the Lie algebra $\mathfrak m\oplus \mg^0$, as defined in \eqref{mgk}.
The bundle $P^1$ over $P^0$ with the fiber $P^1(\vf)$ over a point $\vf\in P^0$ is called
the \emph {first (geometric)  prolongation} of the bundle $P^0$.
\medskip

{\bf Conclusion} \emph{Given a subspace $\mathcal N_0\subset \mathcal A_0$ satisfying \eqref{normsplitgr} there exists a unique affine subbundle $P^1$ of the bundle $\widehat P^1$ with the fiber $P^1(\vf)$ over the point $\vf$ that satisfies \eqref{fiberP1}. A fiber $P^1(\vf)$ is an affine space over the linear space $L^1_\vf\subset \displaystyle{\bigoplus_{i<-1}{\rm Hom}(\mg^i,\mg^{i+1})}\oplus
{\rm Hom}(\mg^{-1},L^0)$. Moreover the corresponding graded space ${\gr L_\vf^1}$ (with respect to filtration \eqref{domainfilt}) is equal to the first algebraic prolongation $\mg^1$ of the algebra $\mathfrak m\oplus \mg^0$ under the identification \eqref{isomg0}}.

Finally all spaces $L_\vf^1$ can be canonically identified with one vector space. For this take a subspace $\mathcal M_1$ of the space %$\mg^0(m)\subset \gl(\mg^{-1}$
$\displaystyle{\bigoplus_{i<-1}{\rm Hom}(\mg^i,\mg^{i+1})}\oplus{\rm Hom}(\mg^{-1},L^0)$
such that the corresponding graded space $\gr \mathcal M_1$ is complementary to $\gr L_{\vf}^1$ in
$\displaystyle{\bigoplus_{i<-1}{\rm Hom}(\gr\mg^i,\gr \mg^{i+1})}\oplus{\rm Hom}(\gr\mg^{-1},\gr L^0)$, i.e.
\begin{equation*}
%\label{splitM10}
{\bigoplus_{i<-1}{\rm Hom}(\gr\mg^i,\gr\mg^{i+1})}\oplus{\rm Hom}(\gr\mg^{-1},\gr L^0)=\gr L_{\vf}^1\oplus \gr \mathcal M_1.
\end{equation*}
By above, the space $\gr L_{\vf}^1$ is equal to $\mg^1$, i.e. does not depend on $\vf$, so the choice of $\mathcal M_1$ as above is indeed possible. Therefore
\begin{equation*}
%\label{splitM}
%\gl (\mg^{-1})
{\bigoplus_{i<-1}{\rm Hom}(\mg^i,\mg^{i+1})}\oplus{\rm Hom}(\mg^{-1}, L^0)
=L_{\vf}^1\oplus \mathcal M_1.
\end{equation*}
 for any $\vf\in P^0$. This splitting defines the identification
 ${\rm Id}_\vf^1
 $
 between
 the factor-space
\begin{equation}
\label{L1}
L^1:=\left({\bigoplus_{i<-1}{\rm Hom}(\mg^i,\mg^{i+1})}\oplus{\rm Hom}(\mg^{-1},L^1)\right)\bigl/\mathcal M_1\bigr..
\end{equation}

and the spaces $L_\vf^1$ (which in turn are canonically identified with tangent space to the fiber $P^1(\vf)$ of $P^1$ over $\vf$).
%${\rm Id}_\vf:L_\vf^0\mapsto L^0$.
The space $L^1$ has the natural filtration
 induced by the filtration on $\displaystyle{\bigoplus_{i<-1}{\rm Hom}(\mg^i,\mg^{i+1})}\oplus{\rm Hom}(\mg^{-1},L^1)$. The aforementioned
 identification isomorphism
 %${\rm Id}_\vf^1$
 preserves the filtrations on the spaces $L^1$ and $L_\vf^1$.
 The space $\mathcal M_1$ is called the \emph{identifying space for the first prolongation}.
 %Note that by condition (3) of Definition \ref{quasidef} we have the following identifications:
% \begin{equation}
% \label{id3}
% \mg^0\cong \gr L_\vf^0\cong \gr L^0.
% \end{equation}

\section{Proof of theorem \ref{maintheor}: Higher order prolongation of quasi-principle bundles.}
\label{highprolongsec}
\setcounter{equation}{0}
\setcounter{theorem}{0}
\setcounter{lemma}{0}
\setcounter{proposition}{0}
\setcounter{definition}{0}
\setcounter{cor}{0}
\setcounter{remark}{0}
\setcounter{example}{0}

Now we are going to construct the higher order geometric prolongations of the bundle $P^0$ by induction.  Assume that all  $l$-th order prolongations $P^l$ are constructed for $0\leq l\leq k$. We also set $P^{-1}=\mathcal S$. We will not specify what the bundles $P^l$ are exactly. As in the case of the first prolongation $P^1$,  their construction depends on the choice of the identifying spaces and normalization conditions on each step.
But we will point out those properties of these bundles that we need in order to construct the $(k+1)$-st order prolongation $P^{k+1}$. First of all simultaneously with the bundles $P^l$ special filtered vector spaces $L^l$ are constructed recursively such that

\begin{enumerate}
\item [{\bf (A1)}] $L^0$ is as in \eqref{LO};
\item [{\bf (A2)}] $L^l$ is a factor-space of the space $\displaystyle{\bigoplus_{i<-l}{\rm Hom}(\mg^i,\mg^{i+l})}\oplus\displaystyle{\bigoplus_{i=-l}^{-1}{\rm Hom}(\mg^{i},L^{i+l})}$;
\item [{\bf (A3)}] The filtration on $L^l$ is induced by the natural filtration on $\displaystyle{\bigoplus_{i<-l}{\rm Hom}(\mg^i,\mg^{i+l})}\oplus\displaystyle{\bigoplus_{i=-l}^{-1}{\rm Hom}(\mg^{i},L^{i+l})}$, which is given, similarly to \eqref{domainfilt}, by
    \begin{equation}
    \label{domainfiltl}
\left\{\displaystyle{\bigoplus_{i<-l}}\bigl({\rm Hom}(\mg^i,\mg^{i+l})\bigr)_k\oplus
\displaystyle{\bigoplus_{i=-l}^{-1}}\bigl({\rm Hom}(\mg^{i},L^{i+l})\bigr)_k\right\}_{k\in Z},
\end{equation}
where ${\rm Hom}(A,B)_k$ is as in \eqref{Homk} and the filtration on $\mg^i$ is given by \eqref{filtgi};
\item [{\bf (A4)}] The corresponding graded space $\gr L^l$ is naturally identified with the $l$th algebraic prolongation $\mg^l$ of the Lie algebra $\mathfrak m+\mg^0$.
\end{enumerate}

Before describing the properties of bundles $P^l$ note that the filtration on $T\mathcal S$ induces naturally (by pullback) the filtration on each bundle $P^l$, $0\leq l\leq k$. Indeed, let $\Pi_l:P^l\rightarrow P^{l-1}$ be the canonical projection.
The tangent bundle $T P^l$ is endowed with the filtration $\{\Delta^i_l\}$ as follows: For $l=-1$ it coincides with the initial filtration  $\{\Delta^{i}\}_{i<0}$ and  for $l\geq 0$
we get by induction
\begin{equation}
\label{P0filt}
\begin{split}
~&\Delta^l_l=\ker (\Pi_l)_*\\
~& \Delta^i_l(\vf_l)=\Bigl\{v\in T_\lambda P^l: (\Pi_l)_*v\in \Delta_{l-1}^i\bigl(\Pi_l(\vf_l)\bigr)\Bigr\} ,\quad \forall i<l.
\end{split}
\end{equation}
We also set $\Delta^i_l=0$ for $i>l$.

Below are the main properties of bundles $P^l$, $0\leq l\leq k$:

\begin{enumerate}
%\item [{\bf (B1)}] for $l\geq 1$, the bundle  $P^l$ is an affine bundle over $P^{l-1}$ with fibers being affine space over the vector space $L^l$. In particular,  the dimensions of the fibers are  equal to $\dim \mg^l$;
    %with  the canonical projection $\Pi_l$.
%\item [{\bf (B2)}]
%Let $\Pi_l:P^l\rightarrow P^{l-1}$ be the canonical projection.
%The tangent bundle $T P^l$ is endowed with the filtration $\{D^i_l\}$ as follows: For $l=-1$ it coincides with the initial filtration  $\{D^{i}\}_{i<0}$ and  for $l\geq 0$
%we get by induction
%\begin{equation}
%\label{P0filt}
%\begin{split}
%~&D^l_l=\ker (\Pi_l)_*\\
%~& D^i_l(\vf_l)=\Bigl\{v\in T_\lambda P^l: (\Pi_l)_*v\in D_{l-1}^i\bigl(\Pi_l(\vf_l)\bigr)\Bigr\} ,\quad \forall i<l.
%\end{split}
%\end{equation}
%We also set $D^i_l=0$ for $i>l$.
%The subspaces $D^l_l(\vf_l)$, as the tangent spaces to the fibers of $P^l$ , are canonically identified with $L^l$. Denote by ${\rm Id}_{\vf_l}^l:L^l\rightarrow D^l_l(\vf_l)$ the identifying isomorphism.
\item [{\bf (B1)}] The fiber of $P^l$, $0\leq l\leq k$, over a point $\vf_{l-1}\in P^{l-1}$ will be a certain affine subspace of
the space of all maps belonging to the space
\begin{equation*}
%\vf_l\in
\bigoplus_{i\leq -1} \text{Hom}\bigl(\mg^i,
\Delta_{l-1}^i(\vf_{l-1})/ \Delta_{l-1}^{i+l+1}(\vf_{l-1})\bigr)\oplus\bigoplus_{i=0}^{l-1} \text{Hom}\bigl(L^i,
\Delta_{l-1}^i(\vf_{l-1})
%/ \Delta_{l-1}^{i+l+1}(\lambda_{l-1})
\bigr).
\end{equation*}
  %If $l>0$ and $\vf_l
  %=(\lambda_{l-1}, \vf_l)
  %\in P^l$, then $\vf_l|_{L^{l-1}}$ coincides with the identification of $L^{l-1}$ with $D^{l-1}_{l-1}\bigl(\Pi_l(\vf_l)\bigr)$ and
Moreover, for each $i$, $0\leq i<l$ the restrictions $\vf_l|_{L^i}$ are the same  for all points $\vf_l$ from the same fiber of $P^l$;

\item [{\bf (B2)}] For $0<l\leq k$ if $\vf_l\in P^l$ and $\vf_{l-1}=\Pi_l(\vf_l)$, then
    the points $\vf_{l-1}$ and $\vf_l$, considered as maps, are related as follows: if
    \begin{equation}
    \label{pli}
    \pi_l^i:\Delta_l^i(\vf_{l})/ \Delta_l^{i+l+2}(\vf_{l})\to \Delta_l^i(\vf_{l})/ \Delta_l^{i+l+1}(\vf_{l})
     \end{equation}
     denotes the canonical projection to a factor space and
     \begin{equation}
     \label{Pli}
     \Pi_l^i: \Delta^i_l(\vf_l)/\Delta^{i+l+1}_l(\vf_l)\rightarrow \Delta^i_{l-1}\bigl(\Pi_l(\vf_l)\bigr)/\Delta^{i+l+1}_{l-1}\bigl(\Pi_l(\vf_l)\bigr)
     \end{equation}
are the canonical maps induced by $(\Pi_l)_*$, then
\begin{equation}
\label{relvf}
\begin{split}
~&\text{for } i<0 \quad \vf_{l-1}|_{\mg^i}=\Pi_{l-1}^i\circ \pi_{l-1}^i\circ \vf_{l}|_{\mg^i},\\
~&\text{for } 0\leq i<l\quad  \vf_{l-1}|_{L^i}=\Pi_{l-1}^i\circ \vf_{l}|_{L^i}.
\end{split}
\end{equation}
Note that the maps $\Pi_l^i$ are isomorphisms for $i<0$ and the maps $\pi_l^i$ are identities for $i\geq 0$ (recall that $\Delta_l^i=0$ for $i>l$);

\item [{\bf (B3)}] %Finally,
%denote by $\mg^i(\mathfrak m)$, $i>0$ the $i$th algebraic prolongation of the Lie algebra $\mathfrak m+\mg^0(\mathfrak m)$. Also by $\mg^i(\mathfrak m)$ with $i<0$ we will mean the space $\mg^i$. Then
For $l\geq 1$ the tangent spaces  (=$\Delta^l_l(\vf_l)$) to the fiber $P^l(\vf_{l-1})$, where $\vf_{l-1}=\Pi_l(\vf_l)$,  are canonically identified with certain subspaces $L^l_{\vf_{l-1}}$ of the space
    %$\displaystyle{\bigoplus_{i<0}}{\rm Hom}\bigl(\mg^i(\mathfrak m), \mg^{i+l}(\mathfrak m)\bigr)$
    $\displaystyle{\bigoplus_{i<-l}{\rm Hom}(\mg^i,\mg^{i+l})}\oplus\displaystyle{\bigoplus_{i=-l}^{-1}{\rm Hom}(\mg^{i},L^{i+l})}$, which in turn are canonically identified with the space $L^l$. The obtained in this way canonical isomorphism between $L^l$ and $\Delta^l_l(\vf_l)$ will be denoted ${\rm Id}_{\vf_l}^l$. Finally, $\vf_l|_{L^{l-1}}$ coincides with  ${\rm Id}_{\vf_l}^l$.

    %In particular, taking into account item (B2) above, the spaces $L^l$ and $L^l_{\vf_{l-1}}$ are canonically identified (by ${\rm Id}_{\vf_l}^l$).
 \end{enumerate}

Note also that  for $l\geq 1$, the bundle  $P^l$ is an affine bundle over $P^{l-1}$ with fibers being affine space over the vector space $L^l$. In particular,  the dimensions of the fibers are  equal to $\dim \mg^l$.

Now we are ready to construct the $(k+1)$-st order Tanaka geometric prolongation. Fix a point $\vf_k \in P^k$. Then
\begin{equation*}
\vf_k\in
% \displaystyle{\bigoplus_{i< k} \text{Hom}\bigl(\mg^i,
%D_{k-1}^i(\lambda_{k-1})/ D_{k-1}^{i+k+1}(\lambda_{k-1})\bigr)}.
\bigoplus_{i< -1} \text{Hom}\bigl(\mg^i,
\Delta_{k-1}^i(\vf_{k-1})/ \Delta_{k-1}^{i+k+1}(\vf_{k-1})\bigr)\oplus\bigoplus_{i=0}^{k-1} \text{Hom}\bigl(L^i,
\Delta_{k-1}^i(\vf_{k-1})
%/ D_{l-1}^{i+l+1}(\lambda_{l-1})
\bigr),
\end{equation*}
where $\vf_{k-1}=\Pi_k(\vf_k)$.
 Let $\mathcal H_k=\{H_k^i\}_{i<k}$ be the tuple of spaces such that $H_k^i=\vf_k(\mg^i)$ for $i<0$ and $H_k^i=\vf_k(L^i)$ for $0\leq i<k$. Take a tuple  $\mathcal H_{k+1}=\{H_{k+1}^i\}_{i<k}$ of linear spaces such that
\begin{enumerate}
\item for $i<0$ the space
$H_{k+1}^i$ is a complement
of $\Delta_k^{i+k+1}(\vf_k)/\Delta_k^{i+k+2}(\vf_k)$ in $(\Pi_k^i\circ\pi_k^i)^{-1} (H_k^i)\subset \Delta_{k}^i(\vf_{k})/ \Delta_{k}^{i+k+2}(\vf_{k})$,
\begin{equation}
\label{Hk-}
(\Pi_k^i\circ\pi_k^i)^{-1} (H_k^i)=\Delta_k^{i+k+1}(\vf_k)/\Delta_k^{i+k+2}(\vf_k)\oplus H_{k+1}^i;
\end{equation}

\item for $0\leq i<k$ the space $H_{k+1}^i$ is a complement of $\Delta_k^k(\vf_k)$ in  $(\Pi_k^i)^{-1} (H_k^i)$,
\begin{equation}
\label{Hk+}
(\Pi_k^i)^{-1} (H_k^i)=\Delta_k^{k}(\vf_k)\oplus H_{k+1}^i.
\end{equation}
\end{enumerate}
Here the maps $\pi_k^i$ and $\Pi_k^i$ are defined as in \eqref{pli} and \eqref{Pli} with $l=k$.

Since $\Delta_k^{i+k+1}(\vf_k)/\Delta_k^{i+k+2}(\vf_k)=\ker \pi_k^i$ and $\Pi_k^i$ is an isomorphism for $i<0$,  the map $\Pi_k^i\circ\pi^i_k|_{H_{k+1}^i}$ defines an isomorphism between $H_{k+1}^i$ and $H_k^i$ for $i<0$.
Additionally,  by \eqref{Hk+} the map $(\Pi_k)_*|_{H_{k+1}^i}$ defines an isomorphism between $H_{k+1}^i$ and $H_k^i$ for $0\leq i<k$.
 So, once a tuple of subspaces $\mathcal H_{k+1}=\{H_{k+1}^i\}_{i<k}$, satisfying \eqref{Hk-} and \eqref{Hk+}, is chosen,
one can define a map
\begin{equation*}
\vf^{\mathcal H_{k+1}}\in
%\in \bigoplus_{i\leq k} \text{Hom}\bigl(\mg^i,
%\Delta_k^i(\lambda_k)/ \Delta_k^{i+k+2}(\lambda_k)\bigr)
\bigoplus_{i< -1} \text{Hom}\bigl(\mg^i,
\Delta_{k}^i(\vf_{k})/ \Delta_{k}^{i+k+2}(\vf_{k})\bigr)\oplus\bigoplus_{i=0}^{k} \text{Hom}\bigl(L^i,
\Delta_{k}^i(\vf_{k})
%/ \Delta_{l-1}^{i+l+1}(\lambda_{l-1})
\bigr)
\end{equation*}
 as follows

\begin{equation}
\label{vfHk}
\begin{split}
&\vf^{\mathcal H_{k+1}}|_{\mg^i}=
%\begin{cases}
%~&\vf^{\mathcal H_{k+1}}|_{\mg^i}=
(\Pi_k^i\circ\pi^i_k|_{H_{k+1}^i})^{-1}\circ\vf_k|_{\mg^i}, \quad \text { if } i<0,\\
&\vf^{\mathcal H_{k+1}}|_{L^i}=
\bigl((\Pi_k)_*|_{H_{k+1}^i}\bigr)^{-1}\circ\vf_k|_{\mg^i}, \quad \text{ if }0\leq i<k,\\
&\vf^{\mathcal H_{k+1}}|_{L^k}=
{\rm Id}_{\vf_k}^k. %\quad  \text{ if } i=k.
\end{split}
%\end{cases}
\end{equation}
Can we choose a tuple or a subset of tuples
 $\mathcal H_{k+1}$ in a canonical way? To answer this question, by analogy with section \ref{firstprolongsec}, we introduce a \textquotedblleft partial soldering form \textquotedblright of the bundle $P^k$ and the structure function of a tuple $\mathcal H_{k+1}$.
The \emph{soldering form} of $P^k$ is a tuple
$\Omega_k=\{\omega_k^i\}_{i<k}$, where $\omega_k^i$ is a $\mg^i$-valued linear form on $\Delta_k^i(\vf_k)$ for $i<0$ and $L^i$-valued linear form on $\Delta_k^i(\vf_k)$ for $0\leq i<k$
%/\Delta_0^{i}(\lambda)$
defined by
\begin{equation}
\label{soldpart}
\omega_k^i(Y)=\vf_k^{-1}\Bigl(\Bigl((\Pi_k)_*(Y)\Bigr)_i\Bigr)\Bigr).
%\pi_0^i
\end{equation}
Here $\Bigl((\Pi_k)_*(Y)\Bigr)_i$ is the equivalence class of $(\Pi_k)_*(Y)$ in $\Delta_{k-1}^i(\vf_{k-1})/\Delta_{k-1}^{i+k+1}(\vf_{k-1})$.
By construction, it follows immediately that  $\Delta_k^{i+1}(\vf_k)=\ker \omega_k^i$. So, the form $\omega_k^i$ induces the $\mg^i$-valued form $\bar \omega_k^i$ on $\Delta_{k}^i(\vf_k)
/\Delta_k^{i+1}(\vf_k)$.

The \emph{ structure function $C_{\mathcal H_{k+1}}^k$ of a tuple $\mathcal H_{k+1}$} is the element of the space
\begin{equation}
\label{Ak}
\begin{split}
~&{\mathcal A}_k=\left(\bigoplus_{i=-\mu}^{-k-1} {\rm Hom}(\mg^{-1}\otimes\mg^i,\mg^{i+k})\right)\oplus
\left(\bigoplus_{i=-k}^{-2} {\rm Hom}(\mg^{-1}\otimes\mg^i,L^{i+k})\right)
\oplus \\
~&{\rm Hom}(\mg^{-1}\wedge\,\mg^{-1},L^{k-1})\oplus\left( \bigoplus_{i=0}^{k-1} {\rm Hom}(\mg^{-1}\otimes L^i,L^{k-1})\right)
\end{split}
\end{equation}
 defined as follows:
Let  $\pi_l^{i,s}: \Delta_l^i(\vf_l)/\Delta_l^{i+l+2}(\vf_l)\rightarrow \Delta_l^i(\vf_l)/ \Delta_l^{i+l+2-s}(\vf_l)$ be the canonical projection to a factor space, where $-1\leq l\leq k$, $i\leq l$.
%, and $0\leq s\leq k+2$.
Here, as before, we assume that $\Delta_l^i=0$ for $i>l$.
Note that the previously defined $\pi_l^i$ coincides with $\pi_l^{i,1}$.
By construction, one has the following two relations
%of the space $\Delta_{k}^i(\lambda_{k})/\Delta_{k}^{i+k+2}(\lambda_{k})$ for $i<0$:
%\begin{equation}
\begin{eqnarray}
~&
\Delta_{k}^i(\vf_{k})/\Delta_{k}^{i+k+2}(\vf_{k})=\left(\bigoplus_{s=0}^k\pi_k^{i+s,s} (H_{k+1}^{i+s})\right)\oplus
\Delta_{k}^{i+k+1}(\vf_{k})/\Delta_{k}^{i+k+2}(\vf_{k})
\,\, \text{ if } i<0,
\label{longsplit1}
\\
~&\Delta_{k}^i(\vf_{k})=\left(\bigoplus_{s=i}^{k-1}
%\pi_k^{i+s,s} (
H_{k+1}^{i}\right)\oplus \Delta_{k}^{k}(\vf_{k})\quad \text{ if } 0\leq i<k.
\label{longsplit2}
\end{eqnarray}
%\end{equation}
Let $\text{pr}_i^{\mathcal H_{k+1}}$ be the projection of $\Delta_k^i(\vf_k)/ \Delta_k^{i+k+2}(\vf_k)$ to $\Delta_k^{i+k+1}(\vf_k)/ \Delta_k^{i+k+2}(\vf_k)$
%($\sim \Delta^{i+1}(x)/ \Delta^{i+2}(x)$) parallel to $H^i$ ( or
corresponding to the splitting \eqref{longsplit1}
if $i<0$ or the projection of $\Delta_k^i(\vf_k)$ to
%$\Delta_k^{k}(\lambda_k)$
$H_{k+1}^{k-1}$
%($\sim D^{i+1}(x)/ D^{i+2}(x)$) parallel to $H^i$ ( or
corresponding to the splitting \eqref{longsplit2} if $0\leq i<k$.
Given vectors $v_1\in \mg^{-1}$ and $v_2\in\mg^{i}$ take two vector fields $Y_1$ and $Y_2$ in a neighborhood $U_k$ of $\vf_k$ in $P^k$ such that for any $\tilde \vf_k%=(\tilde \lambda_{k-1},\tilde\vf_k)
\in U_k$, where
$$\tilde \vf_k\in
\bigoplus_{i\leq -1} \text{Hom}\Bigl(\mg^i,
\Delta_{k-1}^i\bigl(\Pi_k(\tilde\vf_{k})\bigr)/ \Delta_{k-1}^{i+k+1}\bigl(\Pi_k(\tilde\vf_{k})\bigr)\Bigr)\oplus\bigoplus_{i=0}^{k-1} \text{Hom}\Bigl(L^i,
\Delta_{k-1}^i(\Pi_k(\tilde\vf_{k})\bigr)
%/ D_{l-1}^{i+l+1}(\lambda_{l-1})
\Bigr).$$
%D_{k-1}^i(\lambda_{k-1})/ D_{k-1}^{i+k+1}(\tilde\lambda_{k-1})\bigr)}$,
one has
\begin{equation}
\label{Y1Y2k}
\begin{split}
~&\Pi_{k_{*}}Y_1(\tilde \vf_k)= \tilde \vf_{k}(v_1), \quad \Pi_{k_{*}}Y_2(\tilde\vf_k)\equiv \tilde\vf_{k}(v_2) \quad {\rm mod}\,\, \Delta_{k-1}^{i+k+1}\bigl(\Pi_k(\tilde\vf_{k})\bigr),
\\
~& Y_1(\vf_k)=\vf^{\mathcal H_{k+1}}(v_1),\quad Y_2(\vf_k)\equiv\vf^{\mathcal H_{k+1}}(v_2)\,\,{\rm mod}\, \Delta_k^{i+k+2}(\vf_k).
\end{split}
\end{equation}
Then set
\begin{equation}
\label{structTk}
C_{\mathcal H^{k+1}}^k(v_1,v_2)\stackrel{\text{def}}{=}
\begin{cases}
\bar\omega_k^{i+k}\Bigl({\rm pr}_{i-1}^{\mathcal H_{k+1}}\bigl([Y_1,Y_2]\bigr)\Bigr)
& \text{ if } i<0,\\
%I_{\lambda_k}^{-1}
\omega_k^{k-1}\Bigl({\rm pr}_{i-1}^{\mathcal H_{k+1}}\bigl([Y_1,Y_2]\bigr)\Bigr)& \text{ if } 0\leq i<k.
\end{cases}
\end{equation}
As in the case of the first prolongation, it is not hard to see that $C_{\mathcal H}^k(v_1,v_2)$ does not depend on the choice of vector fields $Y_1$ and $Y_2$, satisfying \eqref{Y1Y2k}.

%Indeed, assume that $\widetilde Y_1$ and $\widetilde Y_2$ is another pair of vector fields in a neighborhood of $\lambda_k$ in $P^k$ such that $\widetilde Y_1$ is a section of $D_k^{-1}$, $\widetilde Y_2$ is a section of $D_k^i$, and they satisfy \eqref{Y1Y2k}
%with  $Y_1$, $Y_2$ replaced by $\widetilde Y_1$, $\widetilde Y_1$.
%Then
%\begin{equation}
%\label{zzk}
%\widetilde Y_1=Y_1+Z_1, \quad \widetilde Y_2=Y_2+Z_2,
%\end{equation}
%where $Z_1$ is a section of the distribution $D_k^k$ such that $Z_1(\lambda_k)=0 $ and $Z_2$ is a section of the distribution $D_k^
%{\min\{i+k+1,k\}}$ such that $Z_2(\lambda_k)\in D_k^
%{\min\{i+k+1,k\}+1}
%(\lambda_k)$.
% Then $[Y_1, Z_2](\lambda_k) \in D_k^
% %{i+k+1}
% {\min\{i+k+1,k\}}
% (\lambda_k)$ and $[Y_2, Z_1](\lambda)\in D_k^
% {\min\{i+k+1,k\}}
%%{i+k+1}
% (\lambda_k)$. This together with the fact that $[Z_1,Z_2]$ is a section of $D_k^
% {\min\{i+k+1,k\}+1}
% %{i+k+2}
% $ implies that
%\begin{equation*}
%[\widetilde Y_1,\widetilde Y_2](\lambda)\equiv [Y_1,Y_2]\,
%\text{ mod }\,D_k^
%{\min\{i+k+1,k\}}
%%{i+k+1}
%(\lambda).
% \end{equation*}
%From \eqref{structTk} it follows that the structure function is independent
%%$C_{\mathcal H}^0(v_1,v_2)$
%of the choice of vector fields $Y_1$ and $Y_2$.

Now take another tuple $\widetilde {\mathcal H}_{k+1}=\{\widetilde H_{k+1}^i\}_{i<k}$ such that
%\begin{equation}
%\label{tildeH1}
%D_0^i(\lambda)/ D_0^{i+2}(\lambda)=D_0^{i+1}(\lambda)/ D_0^{i+2}(\lambda)\oplus \widetilde H^i.
%\end{equation}
\begin{enumerate}
\item for $i<0$ the space
$\widetilde H_{k+1}^i$ is a complement
of $\Delta_k^{i+k+1}(\vf_k)/\Delta_k^{i+k+2}(\vf_k)$ in $(\Pi_k^i\circ\pi_k^i)^{-1} (H_k^i)\subset \Delta_{k}^i(\vf_{k})/ \Delta_{k}^{i+k+2}(\vf_{k})$,
\begin{equation}
\label{Hk-t}
(\Pi_k^i\circ\pi_k^i)^{-1} (H_k^i)=\Delta_k^{i+k+1}(\vf_k)/\Delta_k^{i+k+2}(\vf_k)\oplus \widetilde H_{k+1}^i;
\end{equation}

\item for $0\leq i<k$ the space $\widetilde H_{k+1}^i$ is a complement of $\Delta_k^k(\vf_k)$ in  $(\Pi_k^i)^{-1} (H_k^i)$,
\begin{equation}
\label{Hk+t}
(\Pi_k^i)^{-1} (H_k^i)=\Delta_k^{k}(\vf_k)\oplus \widetilde H_{k+1}^i.
\end{equation}
\end{enumerate}
%\begin{enumerate}
%\item for $i<0$ the space
%$\widetilde H_{k+1}^i$ is a complement
%of $D_k^{i+k+1}(\lambda_k)/D_k^{i+k+2}(\lambda_k)$ in $(\Pi_k^i\circ\pi_k^i)^{-1} (H_k^i)\subset D_{k}^i(\lambda_{k})/ D_{k}^{i+k+2}(\lambda_{k})$,
%\begin{equation}
%\label{Hk-t}
%(\Pi_k^i\circ\pi_k^i)^{-1} (H_k^i)=D_k^{i+k+1}(\lambda_k)/D_k^{i+k+2}(\lambda_k)\oplus \widetilde H_{k+1}^i.
%\end{equation}
%
%\item for $0\leq i<k$ the space $\widetilde H_{k+1}^i$ is a complement of $D_k^k(\lambda_k)$ in  $(\Pi_k^i)^{-1} (H_k^i)$,
%\begin{equation}
%\label{Hk+t}
%(\Pi_k^i)^{-1} (H_k^i)=D_k^{k}(\lambda_k)\oplus \widetilde H_{k+1}^i.
%\end{equation}
%\end{enumerate}
How are the structure functions $C_{{\mathcal H}_{k+1}}^k$ and $C_{\widetilde{\mathcal H}_{k+1}}^k$ related?
By construction, for any vector $v\in\mg ^i$ the vector $\vf^{\widetilde{\mathcal H}_{k+1}}(v)-\vf^{\mathcal H_{k+1}}(v)$ belongs to
$\Delta_k^{i+k+1}(\vf_k)/ \Delta_k^{i+k+2}(\vf_k)$, for $i<0$, and to $\Delta_k^{k}(\vf_k)$,  for $0\leq i<k$.
%($\sim D^{i+1}(x)/ D^{i+2}(x)$).
Let
\begin{equation*}
%\begin{split}
%~&
f_{\mathcal H_{k+1}\widetilde {\mathcal H}_{k+1}}(v)\stackrel{\text{def}}{=} \begin{cases}
%\vf^{-1}
\bar\omega_k^{i+k+1}\left(\vf^{\widetilde{\mathcal H}_{k+1}}(v)-\vf^{\mathcal H_{k+1}}(v)\right) & \text{ if } v\in \mg^i \text{ with } i<-1\\
%~&
%f_{\mathcal H\widetilde {\mathcal H}}(\lambda)(v)\stackrel{def}{=}
%\vf^{-1}
({\rm Id}_{\vf_k}^k)^{-1}\left(\vf^{\widetilde{\mathcal H}_{k+1}}(v)-\vf^{\mathcal H_{k+1}}(v)\right) & \text{ if } v\in \mg^{-1} \text{ or} v\in L^i \text{with } 0\leq i<k.
%\end{split}
\end{cases}
\end{equation*}

Then
\begin{equation*}
f_{\mathcal H_{k+1}\widetilde {\mathcal H}_{k+1}}\in \displaystyle{\bigoplus_{i<-k-1}{\rm Hom}(\mg^i,\mg^{i+k+1})}\oplus \displaystyle{\bigoplus_{-k-1\leq i<0}{\rm Hom}(\mg^i,L^{i+k+1})}\oplus
\displaystyle{\bigoplus_{i=0}^{k-1}{\rm Hom}(L^i,L^k)}.
\end{equation*}
 In the opposite direction, it is clear that for any $$f\in
 %\displaystyle
 %{\bigoplus_{i<0}{\rm Hom}(\mg^i,\mg^{i+k+1})}\oplus
%\displaystyle{\bigoplus_{i=0}^{k-1}{\rm Hom}(\mg^i,\mg^k)}$
\displaystyle{\bigoplus_{i<-k-1}{\rm Hom}(\mg^i,\mg^{i+k+1})}\oplus \displaystyle{\bigoplus_{-k-1\leq i<0}{\rm Hom}(\mg^i,L^{i+k+1})}\oplus
\displaystyle{\bigoplus_{i=0}^{k-1}{\rm Hom}(L^i,L^k)},$$ there exists a tuple $\widetilde{\mathcal H}_{k+1}=\{\widetilde H_{k+1}^i\}_{i<k}$ satisfying \eqref{Hk-t} and \eqref{Hk+t} and such that  $f=f_{\mathcal H_{k+1} \widetilde {\mathcal H}_{k+1}}$. Further, let $\mathcal A_k$ be as in \eqref{Ak} and define
a map
\begin{equation*}
\partial_k:
%\displaystyle{\bigoplus_{i<0}{\rm Hom}(\mg^i,\mg^{i+k+1})}\oplus
%\displaystyle{\bigoplus_{i=0}^{k-1}{\rm Hom}(\mg^i,\mg^k)}\rightarrow
\displaystyle{\bigoplus_{i<-k-1}{\rm Hom}(\mg^i,\mg^{i+k+1})}\oplus \displaystyle{\bigoplus_{-k-1\leq i<0}{\rm Hom}(\mg^i,L^{i+k+1})}\oplus
\displaystyle{\bigoplus_{i=0}^{k-1}{\rm Hom}(L^i,L^k)}\rightarrow\mathcal A_k
%\displaystyle{\bigoplus_{i<0}
% {\rm Hom}(\mg^{-1}\wedge\mg^i,\mg^{i+k})}\oplus\bigoplus_{i=0}^{k-1} {\rm Hom}(\mg^{-1}\wedge\mg^i,\mg^{k})
\end{equation*}
  by
\begin{equation}
\label{Spk}
\partial_k f(v_1,v_2)=\begin{cases}[f(v_1),v_2]+[v_1,f(v_2)]-f([v_1,v_2])&\text{ if } v_1\in \mg^
{-1}, v_2\in \mg^i, i\leq -k-1; \\
\bigl(f(v_1)\bigr)(v_2)- \bigl(f(v_2)\bigr)(v_1)-f([v_1,v_2])&\text{ if } v_1\in \mg^
{-1}, v_2\in \mg^i, -k-1<i< 0 \\
-\bigl(f(v_2)\bigr)(v_1) &\text{ if } v_1\in \mg^
{-1}, v_2\in L^i, 0\leq i
<k-1,
\end{cases}
\end{equation}
Here in the first and in the second line the brackets $[\, \,, \,]$ are as in the Lie algebra $\mathfrak m\oplus \mg^0(\mathfrak m)$ and in the second and the third line in the expressions $\bigl(f(v_1)\bigr)(v_2)$ and  $\bigl(f(v_2)\bigr)(v_1)$ we use the identification between the appropriate spaces $L_{\vf_{l-1}}^l$ and $L^l$  from the property (B3) of the bundles $P^l$. Under this identification, we look on $f(v_1)$ as on an element of the space
$\displaystyle{\bigoplus_{i<-k}{\rm Hom}(\mg^i,\mg^{i+k})}\oplus\displaystyle{\bigoplus_{i=-k}^{-1}{\rm Hom}(\mg^{i},L^{i+k})}$, which gives the appropriate meaning to $\bigl(f(v_1)\bigr)(v_2)$. Similarly, one gives the meaning to the expression $\bigl(f(v_2)\bigr)(v_1)$.

Note that the map $\partial_k$ in general depends on the point $\vf_k\in P^k$.
%follows: the brackets of two elements from $\mathfrak m$ are as in $\mathfrak m$, the brackets of an element with non-negative weight %and an element from $\mathfrak m$ are already defined by  \eqref{br2}, and the brackets of an element from $\mg^k$ and an element %with  non-positive weight are equal to zero. Note that other brackets do not appear in
%\eqref{Spk}. The operator $\partial_k$ is called the \emph{generalized Spencer operator for the $(k+1)$th prolongation}.
%From the definition of brackets it follows that
% \begin{equation}
%\label{Spk+}
%\partial_k f(v_1,v_2)=[v_1,f(v_2)]-f([v_1,v_2]),\quad v_1\in \mg^
%{-1}, v_2\in \mg^i, \text { if } 0\leq i
%<k-1.
%\end{equation}
  %Note also that the bracket of an element from $\mg^k$ and an element with  non-positive weight, used in \eqref{Spk},  are different %in general from the bracket of the same pair
  %in the algebraic universal  prolongation $\mg(\mathfrak m,\mg^0)$.
  %, but
 %such brackets do not appear in \eqref{Spk} for $k=0$.
Also, for $k=0$ this definition coincides with the definition of the generalized Spencer operator for the first prolongation given in the previous section.

%\begin{proposition}
%\label{prop1}
%The reason for introducing the operator $\partial_k$
Similarly to the identity \eqref{structrans} the following identity holds:
\begin{equation}
\label{structransk}
C_{\widetilde {\mathcal H}_{k+1}}^k=C_{\mathcal H_{k+1}}^k+\partial_kf_{\mathcal H_{k+1} \widetilde{\mathcal H}_{k+1}}.
\end{equation}

A verification of this identity for pairs $(v_1,v_2)$, where $v_1\in \mg^{-1}$ and
$v_2\in \mg^i$ with $i<0$, is completely analogous to the proof of Proposition 3.1 in  \cite{zeltan}. For $i\geq 0$ one has to use the inductive assumption that the restrictions $\vf_l|_{\mg^i}$
%with $i\geq 0$
are the same for all $\vf_l$ from the same fiber
(see property (B2) from the list of properties satisfied by $P^l$ in the beginning of this section) and the splitting \eqref{longsplit2}.

Recall that the domain and the target spaces of the map $\partial_k$ have natural filtrations induced by the filtrations on the spaces $\mg^i$, $i<0$, given by
\eqref{filtgi}. Moreover, the map $\partial_k$ preserves these filtrations.

What can we say about the associated map $\gr \partial_k$ of the corresponding graded spaces? Using the identifications similar to \eqref{idpart01} and \eqref{idpart02}, identifications \eqref{isomg0}, and the property (A4) of the spaces $L^l$ above, the domain and the target spaces of the map $\gr \partial_k$ can be identified with the spaces
\begin{equation}
\label{domaingrk}
\displaystyle{\bigoplus_{i<0}{\rm Hom}(\mg^i,\mg^{i+k+1})}\oplus
\displaystyle{\bigoplus_{i=0}^{k-1}{\rm Hom}(\mg^i,\mg^k)}
\end{equation}
and
\begin{equation}
\label{targetgrk}
\left(\bigoplus_{i=-\mu}^{-2} {\rm Hom}(\mg^{-1}\otimes\mg^i,\mg^{i+k})\right)\oplus {\rm Hom}(\mg^{-1}\wedge\,\mg^{-1},\mg^{k-1})\oplus\left( \bigoplus_{i=0}^{k-1} {\rm Hom}(\mg^{-1}\otimes\mg^i,\mg^{k-1})\right),
\end{equation}
respectively. Moreover, using these identifications and \eqref{Spk} one gets that the map $\gr \partial_k$ satisfies

\begin{equation}
\label{Spkgr}
\gr \partial_k f(v_1,v_2)=\begin{cases}[f(v_1),v_2]+[v_1,f(v_2)]-f([v_1,v_2])&\text{ if } v_1\in \mg^
{-1}, v_2\in \mg^i, i<0; \\
[v_1,f(v_2)] &\text{ if } v_1\in \mg^
{-1}, v_2\in \mg^i, 0\leq i
<k-1,
\end{cases}
\end{equation}
where the brackets $[\, \,, \,]$ are as
%follows: the brackets of two elements from $\mathfrak m$ are as in $\mathfrak m$, the brackets of an element with non-negative weight %and an element from $\mathfrak m$ are already defined by  \eqref{br2}, and the brackets of an element from $\mg^k$ and an element %with  non-positive weight are equal to zero. Note that other brackets do not appear in
%\eqref{Spk}. The operator $\partial_k$ is called the \emph{generalized Spencer operator for the $(k+1)$th prolongation}.
%From the definition of brackets it follows that
% \begin{equation}
%\label{Spk+}
%\partial_k f(v_1,v_2)=[v_1,f(v_2)]-f([v_1,v_2]),\quad v_1\in \mg^
%{-1}, v_2\in \mg^i, \text { if } 0\leq i
%<k-1.
%\end{equation}
  %Note also that the bracket of an element from $\mg^k$ and an element with  non-positive weight, used in \eqref{Spk},  are different %in general from the bracket of the same pair
  in the algebraic universal  prolongation $\mathfrak u(\mathfrak m,\mg^0)$ of the pair $(\mathfrak m,\mg^0)$.

 \begin{remark}
 \label{kergrk}
{\rm  Note that
 \begin{equation}
 \label{kerk}
 f\in \ker\, \gr\partial_k \,\Rightarrow\, f|_{\mg^i}=0, \quad \forall \, 0\leq i\leq k-1.
 \end{equation}
 For the proof see section 3 of \cite{zeltan} (the map $\partial_k$ there coincides with the map $\gr \partial_k$ here).
 In other words,
 \begin{equation}
 \label{kerk1}
 \ker\, \gr \partial_k\subset \displaystyle{\bigoplus_{i<0}{\rm Hom}(\mg^i,\mg^{i+k+1})}.
 \end{equation}
 Moreover, directly from the definition, $\ker \gr \partial_0\cong\mg^{k+1}$, where $\mg^{k+1}$ is the $(k+1)$st algebraic prolongation of the Lie algebra $\mathfrak m\oplus \mg^0$, as defined in \eqref{mgk}.$\Box$}
 \end{remark}
Now fix a subspace
\begin{equation*}
\mathcal N_k\subset \mathcal A_k
%\displaystyle{\bigoplus_{i<0}
% {\rm Hom}(\mg^{-1}\wedge\mg^i,\mg^i)},
 \end{equation*}
such that
%complementary to  $\text{Im}\, \partial_0$,
\begin{equation}
\label{normsplitgr}
%\displaystyle{\bigoplus_{i<0}
% {\rm Hom}(\mg^{-1}\wedge\mg^i,\mg^i)}
\gr A_k= \im \,\gr \partial_k\oplus\gr\mathcal N_k.
 \end{equation}
  By analogy with $G$-structures and with principle bundles of type $(\mathfrak m, \mg^0)$ the subspace $\mathcal N_k$ is called the \emph{normalization conditions for the $(k+1)$st prolongation}. From item (2) of Lemma \ref{lem_abc} it follows that
  \begin{equation}
\label{normsplitk}
%\displaystyle{\bigoplus_{i<0}
% {\rm Hom}(\mg^{-1}\wedge\mg^i,\mg^i)}
\mathcal A_k= \im \,\partial_k+\mathcal N_k.
 \end{equation}

 Given $\vf_k\in P^k$ denote by $P^{k+1}(\vf_k)$ the following space:
  \begin{equation}
  \label{fiberPk}
 P^{k+1}(\vf_k)=\bigl\{\vf^{\mathcal H_{k+1}}: \mathcal H_{k+1}=\{H^i\}_{i<k} \text{ satisfies } \eqref{Hk-} \text{ and } \eqref{Hk+}  \text{ and }C_{\mathcal H_{k+1}}^k\in \mathcal N_k\bigr\},
 \end{equation}
 where
 %, as before,
 %$\mathcal H=\{H^i\}_{i<0}$ are tuples of spaces satisfying \eqref{H1} and maps
 $\vf^{\mathcal H_{k+1}}$ is defined by
 \eqref{vfHk}. Then from the formulas \eqref{structransk} and  \eqref{normsplitk} it follows that $P^{k+1}(\vf_k)$ is not empty.
 Moreover, if $\vf^{\mathcal H_{k+1}}\in P^{k+1}(\vf_k)$ for some tuple of spaces $\mathcal H_{k+1}$, then $\vf^{\widetilde{\mathcal H}_{k+1}}\in P^{k+1}(\vf_k)$ for another tuple of spaces $\widetilde {\mathcal H}_{k+1}$ if and only if
 $$\partial_k f_{\mathcal H_{k+1}, \widetilde{\mathcal H}_{k+1}}\in \mathcal N_k.$$
 %Here $\partial_k$ is acting as in \eqref{from}.
 Therefore, $P^{k+1}(\vf_k)$ is an affine space over the linear space
\begin{equation}
\label{Lk+1}
%\begin{split}
%~&
L^{k+1}_{\vf_k}:=(\partial_k)^{-1}(\mathcal N_k)\subset \displaystyle{\bigoplus_{i<-k-1}{\rm Hom}(\mg^i,\mg^{i+k+1})}\oplus \displaystyle{\bigoplus_{-k-1\leq i<0}{\rm Hom}(\mg^i,L^{i+k+1})}\oplus%\\
%~&
\displaystyle{\bigoplus_{i=0}^{k-1}{\rm Hom}(L^i,L^k)}
%\displaystyle{\bigoplus_{i<-1}{\rm Hom}(\mg^i,\mg^{i+1})}\oplus{\rm Hom}(\mg^{-1},L^0)
.
%\end{split}
\end{equation}

From item (3) of Lemma
\ref{lem_abc} it follows that the corresponding graded space
$\gr L^{k+1}_{\vf_k}$ (with respect to the filtration %\eqref{domainfilt}
of the domain space of $\partial_k$) does not depend on the normalization condition $\mathcal N_k$ and coincides with $\ker \gr \partial_k$, which according to Remark \ref{kergrk} %we get that  under identification \eqref{isomg0} $\ker \gr \partial_k\cong
can be identified with $\mg^{k+1}$.
%, where $\mg^{k+1}$ is the $(k=1$st algebraic prolongation of the Lie algebra $\mathfrak m\oplus \mg^0$, as defined in \eqref{mgk}.
Besides, from \eqref{kerk} it follows that if $f\in L_{\vf_k}^{k+1}$, then $f|_{L^i}=0$ for all $0\leq i<k$. This implies that
$$L_{\vf_k}^{k+1}\subset \displaystyle{\bigoplus_{i<-k-1}{\rm Hom}(\mg^i,\mg^{i+k+1})}\oplus \displaystyle{\bigoplus_{-k-1\leq i<0}{\rm Hom}(\mg^i,L^{i+k+1})}.$$ Moreover, this implies that  for any $i$, $0\leq i<k$, the restrictions
$\vf_{k+1}|_{L^i}$ are the same for all points $\vf_{k+1}$ from the same fiber $P^{k+1}(\vf_k)$.
Note also that by \eqref{vfHk} the restriction $\vf_{k+1}|_{L^k}$ coincides with the identification ${\rm Id}_{\vf_k}^k$.
The bundle $P^{k+1}$ over $P^k$ with the fiber $P^{k+1}(\vf_k)$ over a point $\vf_k\in P^k$ is called
the \emph {$(k+1)$st (geometric)  prolongation} of the bundle $P^0$.
\medskip

Further, all spaces $L_\vf^{k+1}$ can be canonically identified with one vector space. For this take a subspace $\mathcal M_{k+1}$ of the space %$\mg^0(m)\subset \gl(\mg^{-1}$
%$\displaystyle{\bigoplus_{i<-1}{\rm Hom}(\mg^i,\mg^{i+1})}\oplus{\rm Hom}(\mg^{-1},L^0)$
$ \displaystyle{\bigoplus_{i<-k-1}{\rm Hom}(\mg^i,\mg^{i+k+1})}\oplus \displaystyle{\bigoplus_{-k-1\leq i<0}{\rm Hom}(\mg^i,L^{i+k+1})}.$
such that the corresponding graded space $\gr \mathcal  M_1$ is complementary to $\gr L_{\vf}^1$ in
%$\displaystyle{\bigoplus_{i<-1}{\rm Hom}(\mg^i,\mg^{i+1})}\oplus{\rm Hom}(\mg^{-1},L^0)$, i.e.
$$ \displaystyle{\bigoplus_{i<-k-1}{\rm Hom}(\gr \mg^i,\gr \mg^{i+k+1})}\oplus \displaystyle{\bigoplus_{-k-1\leq i<0}{\rm Hom}(\gr \mg^i,\gr L^{i+k+1})}.$$
\begin{equation*}
%\label{splitM10}
{\bigoplus_{i<-1}{\rm Hom}(\gr\mg^i,\gr\mg^{i+1})}\oplus{\rm Hom}(\gr\mg^{-1},\gr L^0)=\gr L_{\vf}^1\oplus \gr \mathcal M_{k+1}.
\end{equation*}
By above, the space $\gr L_{\vf}^{k+1}$ is equal to $\mg^{k+1}$, i.e. does not depend on $\vf_k$, so the choice of $\mathcal M_{k+1}$ as above is indeed possible. Therefore
\begin{equation*}
%\label{splitM}
%\gl (\mg^{-1})
%{\bigoplus_{i<-1}{\rm Hom}(\mg^i,\mg^{i+1})}\oplus{\rm Hom}(\mg^{-1}, L^0)
 \displaystyle{\bigoplus_{i<-k-1}{\rm Hom}(\mg^i,\mg^{i+k+1})}\oplus \displaystyle{\bigoplus_{-k-1\leq i<0}{\rm Hom}(\mg^i,L^{i+k+1})}
=L_{\vf_k}^{k+1}\oplus \mathcal M_{k+1}.
\end{equation*}
 for any $\vf_k\in P^{k}$. This splitting defines the identification
 ${\rm Id}_{\vf_k}^{k+1}$
 between the factor-space
\begin{equation}
\label{L1}
L^{k+1}:=\left(\displaystyle{\bigoplus_{i<-k-1}{\rm Hom}(\mg^i,\mg^{i+k+1})}\oplus \displaystyle{\bigoplus_{-k-1\leq i<0}{\rm Hom}(\mg^i,L^{i+k+1})}\right)\bigl/\mathcal M_1\bigr..
\end{equation}
%${\rm Id}_\vf:L_\vf^0\mapsto L^0$.
and the space $L_\vf^{k+1}$.
The space $L^{k+1}$ has the natural filtration
 induced by the filtration on $ \displaystyle{\bigoplus_{i<-k-1}{\rm Hom}(\mg^i,\mg^{i+k+1})}\oplus \displaystyle{\bigoplus_{-k-1\leq i<0}{\rm Hom}(\mg^i,L^{i+k+1})}.$ This identification preserves the filtrations on the spaces  $L^{k+1}$ and $L_{\vf_k}^{k+1}$.
The space $\mathcal M_{k+1}$ is called the \emph{identifying space for the $(k+1)$-st prolongation}.

By our constructions the space $L^{k+1}$ is canonically identified with tangent space to the fiber $P^{k=1}(\vf_k)$  at any point $\vf_{k+1}$. Denote the identifying isomorphism by ${\rm Id}_{\vf_{k+1}}^{k+1}$.
In this way we finish the induction step by constructing the space $L^{k+1}$ and the bundle $P^{k+1}$ satisfying properties (A1)-(A4) and (B1)-(B3) above for $l=k+1$.

Finally, assume that there exists $\bar l\geq 0$ such that $\mg^{\bar l}\neq 0$ but $\mg^{\bar l+1}= 0$. Since the symbol $\mathfrak m$ is fundamental, it follows that $\mg ^l=0$ for all $l>\bar l$ . Hence, for all $l>\bar l$ the fiber of $P^l$ over a point $\lambda_{l-1}\in P^{l-1}$ is a single point belonging to
%\begin{equation*}
$$
\bigoplus_{i\leq -1} \text{Hom}\bigl(\mg^i,
\Delta_{l-1}^i(\vf_{l-1})/ \Delta_{l-1}^{i+l+1}(\vf_{l-1})\bigr)\oplus\bigoplus_{i=0}^{l-1} \text{Hom}\bigl(L^i,
\Delta_{l-1}^i(\vf_{l-1})
%/ D_{l-1}^{i+l+1}(\lambda_{l-1})
\bigr) .
%\displaystyle{\bigoplus_{i=-\mu}^{l-1} \text{Hom}\bigl(\mg^i,
%D_{l-1}^i(\lambda_{l-1})/ D_{l-1}^{i+l+1}(\lambda_{l-1})\bigr)},
$$
%\end{equation*}
where, as before, $\mu$ is the degree of nonholonomy of the distribution $\Delta$.
Moreover, by our assumption, $\Delta_l^i=0$ if $l\geq\bar l$ and $i\geq \bar l$. Therefore, if $l=\bar l+\mu$, then $i+l+1>\bar l$ for $i\geq -\mu$ and the fiber of $P^l$ over $P^l$ is an element of $${\rm Hom} \left(\displaystyle{\bigoplus_{i=-\mu}^{-1} \mg^i\oplus\bigoplus_{i=0}^{l-1} L^i, T_{\lambda_{l-1}}P^{l-1}}\right).$$ In other words, $P^{\bar l+\mu}$ defines a canonical frame on $P^{\bar l+\mu-1}$.
But all bundles $P^l$ with $l\geq \bar l$ are identified one with each other by the canonical projections (which are diffeomorphisms in that case).
This completes the proof of Theorem \ref{maintheor}.
\end{document}